\documentclass{tran-l}

\usepackage[a4paper,]{geometry}
\newtheorem{theorem}{Theorem}[section]
\newtheorem{lemma}[theorem]{Lemma}

\newtheoremstyle{exampstyle}
  {} 
  {0cm} 
  {\itshape} 
  {} 
  {\bfseries} 
  {.} 
  {.5em} 
  {} 
\theoremstyle{exampstyle} 

\newtheorem{definition}[theorem]{Definition}
\newtheorem{proposition}[theorem]{Proposition}
\newtheorem{corollary}[theorem]{Corollary}
\newtheorem{example}[theorem]{Example}

\newcounter{alg}
{
 \refstepcounter{alg}%
  \textbf{Algorithm \thealg}\begin{itshape}%
}%
{\end{itshape}}

\newtheorem{innercustomthm}{Theorem}
\newenvironment{mythm}[1]
  {\renewcommand\theinnercustomthm{#1}\innercustomthm}
  {\endinnercustomthm}
  
\newtheorem{innercustomprop}{Proposition}
\newenvironment{myprop}[1]
  {\renewcommand\theinnercustomprop{#1}\innercustomprop}
  {\endinnercustomprop}

\usepackage{xcolor}
\definecolor{orange}{RGB}{255, 102, 0}
\definecolor{grey}{RGB}{194, 197, 204}
\definecolor{celeste}{RGB}{172,237,255}
\definecolor{naranja}{RGB}{255,174,133}
\definecolor{verde}{RGB}{189,236,182}

\usepackage{setspace}

\newtheoremstyle{observation}
  {} 
  {0cm} 
  {} 
  {} 
  {\itshape} 
  {.} 
  {.5em} 
  {} 
\theoremstyle{observation}
\newtheorem{remark}[theorem]{Remark}

\usepackage{lipsum} 

\usepackage{amsmath}
\allowdisplaybreaks  
\usepackage{amssymb} 
\usepackage{amsfonts} 
\usepackage[inline]{enumitem}
\usepackage{multicol}
\usepackage{multirow}
\numberwithin{equation}{section}

\usepackage{yhmath}
\usepackage{mathdots}
\usepackage{MnSymbol}
\usepackage{tasks}

\usepackage{graphicx}
\usepackage{caption}
\usepackage{subcaption}
\usepackage{algorithm}    
\usepackage[noend]{algpseudocode}
\makeatletter
\renewcommand{\ALG@beginalgorithmic}{\footnotesize}
\makeatother



\makeatletter
\DeclareRobustCommand{\Lcorner}{\mathbin{\mspace{1mu}\text{\L@corner}\mspace{1mu}}}
\newcommand{\L@corner}{%
  \setlength{\unitlength}{\fontcharht\font`T}%
  \begin{picture}(0.8,1)
  \roundcap
  \Line(0,1)(0,0)(0.8,0)
  \end{picture}%
}
\makeatother

\usepackage{amssymb,amsmath,array,diagbox}
\usepackage{graphicx}
\usepackage{yhmath}
\usepackage{mathdots}
\usepackage{MnSymbol}
\newcommand{\Rone}{\mathbb{Z}}
\newcommand{\Rmor}[1]{\mathbb{Z}^{#1}}
\newcommand{\Tone}[1]{\mathbb{Z}_{#1}}
\newcommand{\Tmor}[2]{(\mathbb{Z}_{#1})^{#2}}

\usepackage[bookmarks,breaklinks,colorlinks=true, allcolors=blue]{hyperref}

\usepackage{comment}

\usepackage{tabularray}

\begin{document}

\title{Positive $3$-braids, Khovanov homology and Garside theory}

\author{Álvaro del Valle Vílchez}
\address{Departamento de Álgebra de la Universidad de Sevilla \& Instituto de Matemáticas de la Universidad de Sevilla (IMUS). SPAIN. }
\email{adelvalle3@us.es}
\thanks{The three autors were supported in part by the grant PID2020-117971GB-C21 funded by MCIN/AEI/10.13039/501100011033. The first author was also supported in part by the grant VII PPIT-US}

\author{Juan González-Meneses}
\address{Departamento de Álgebra de la Universidad de Sevilla \& Instituto de Matemáticas de la Universidad de Sevilla (IMUS). SPAIN.}
\email{meneses@us.es}

\author{Marithania Silvero}
\address{Departamento de Álgebra de la Universidad de Sevilla \& Instituto de Matemáticas de la Universidad de Sevilla (IMUS). SPAIN.}
\email{marithania@us.es}





\begin{abstract}
Khovanov homology is a powerful invariant of oriented links that categorifies the Jones polynomial. Nevertheless, computing Khovanov homology of a given link remains challenging in general with current techniques. In this work we focus on links that are the closure of positive 3-braids. Starting with a classification of conjugacy classes of 3-braids arising from the Garside structure of braid groups, we compute, for any closed positive 3-braid, the first four columns (homological degree) and the three lowest rows (quantum degree) of the associated Khovanov homology table. Moreover, the number of rows and columns we can describe increases with the infimum of the positive braid (a Garside theoretical notion). We will show how to increase the infimum of a 3-braid to its maximal possible value by a conjugation, maximizing the number of cells in the Khovanov homology of its closure that can be determined, and show that this can be done in linear time.
\end{abstract}

\maketitle

\vspace{1cm}


\section{Introduction}

Khovanov homology is a celebrated link invariant introduced by Mikhail Khovanov \cite{Khovanov_2000} as a categorification of the Jones polynomial. Given an oriented diagram $D$ representing a link $L$, Khovanov constructed a family of $\mathbb{Z}$-graded chain complexes whose bigraded homology groups $H^{i,j}(D)$ are link invariants categorifying the Jones polynomial of the link. The groups $H^{i,j}(L)$, known as the \textit{Khovanov homology groups} of $L$, are indexed by the {\it homological} grading $i$ and the {\it quantum} grading $j$. 

This powerful invariant provides geometric and topological information of a link: it gives a lower bound on the slice genus of a knot \cite{Rasmussen_2010} (this allowed to give the first combinatorial proof of Milnor's Conjecture), and can detect fiberedness among positive links \cite{Kegel_Manikandan_Mousseau_Silvero_2023}. Surprisingly, it detects the unknot \cite{Kronheimer_Mrowka_2011}, both trefoils \cite{Baldwin_Sivek_2022}, the figure eight knot \cite{Baldwin_Dowlin_Levine_Lidman_Sazdanovic_2021}, and the cinquefoil $T(2,5)$ \cite{Baldwin_Hu_Sivek_2024}, among others.  

It is common to represent the Khovanov homology of a given link in a table, where the columns (resp. rows) are indexed by the homological index $i$ (resp. quantum index $j$). If the group $H^{i,j}(L)$ is non-trivial for certain values of $i$ and $j$, we include it in the cell indexed by $(i,j)$. Since every link has finitely many non-trivial homology groups, the indices in its Khovanov homology table range between the minimal and maximal values of $i$ (resp. $j$) for which there exists a non-trivial group $H^{i,*}(L)$ (resp. $H^{*,j}(L)$). 

There is a number of works devoted to compute Khovanov homology at certain homological and/or quantum gradings of some families of links. In the present paper we focus on braid positive links, i.e., those links which are closure of positive braids in terms of Artin generators. 
As a consequence of \cite{Khovanov_2003,Stosic_2005, Przytycki_Sazdanovic_2014, Przytycki_Silvero_2020}, the first two columns and lowest two rows of the Khovanov homology table of braid positive links is well known (in this paper we will refer to these groups as the $\Lcorner_{2,2}$-shape of the Khovanov homology of the link). More precisely, given a positive word $w$ representing a braid $\beta \in \mathbb{B}_n$ on $n$ strands, write $\underline{j} = l(w) - n$, where $l(w)$ is the length of $w$. Then, the associated closed braid $\widehat{\beta}$, if not split, satisfies:
$$H^{i,j}(\widehat{\beta})  = \left\{ \begin{array}{cl} \mathbb{Z} &  \mbox{ if } i=0 \, \mbox{ and} \, \, j \in \{\underline{j}, \underline{j}+2\};  \\  0 &  \mbox{ if } i=0 \, \mbox{ and} \, \, j \notin \{\underline{j}, \underline{j}+2\}; \\ 0 &  \mbox{ if }  i \neq 0 \, \, \mbox{ and} \, \, j \in \{\underline{j}, \underline{j}+2\};  \\ 0 &  \mbox{ if }  i=1; \\ 0 &  \mbox{ if } i<0 \, \mbox{ or } \, \, j<\underline{j}. \\  \end{array} \right.$$ 

In this work we go a step further and determine the $\Lcorner_{4,3}$-shape of the Khovanov homology of braid positive links with braid index at most $3$. We state our main result:

\begin{theorem}\label{th:main}
The Khovanov homology of a closed positive $3$-braid $\beta$ is as one of the Tables~\ref{tab:H(1)}--\ref{table:Kh_case_C4}. Moreover, $\beta$ is conjugate to a braid belonging to either $N=\{ 1, \sigma_1, \sigma_1^2, \sigma_1\sigma_2, \sigma_1^2 \sigma_2^2, \Delta \}$ (Tables~\ref{tab:H(1)}--\ref{tab:H(aba)}, respectively) or to some of the families $\mathbf{C1}$, $\mathbf{C2}$, $\mathbf{C3}$, $\mathbf{C4}$ (Tables~\ref{table:Kh_case_C1}--\ref{table:Kh_case_C4}, respectively), where
\[ \arraycolsep=1.4pt\def\arraystretch{1.8}
\begin{array}{rclcrcl} 
\mathbf{C1} & = & \{ \sigma_1^{k_1} \; | \; k_1 \geq 3 \}, & \text{\hspace{2cm}} & \mathbf{C4a} & = & \{ \beta \in \mathbb{B}_3 \; | \;  \inf(\beta) > 0 \} \setminus \{ \Delta \}, \\
\mathbf{C2} & = & \{ \sigma_1^{k_1}\sigma_2^2 \; | \; k_1 \geq 3 \}, & & \mathbf{C4b} & = & \{ \beta \in \mathbb{B}_3 \; | \; \inf_s(\beta) = 0 \text{ and } \operatorname{sl}(\beta) \geq 4 \}, \\
\mathbf{C3} & = & \{ \sigma_1^{k_1}\sigma_2^{k_2} \; | \; k_1,k_2 \geq 3 \}, & & \mathbf{C4} &  = & \mathbf{C4a} \cup \mathbf{C4b}.
\end{array}\]
\end{theorem}

\begin{table}[!htb]
    \begin{minipage}[b]{.3\linewidth}
      \centering
      \captionsetup{width=.8\linewidth}
        \begin{tiny}
           \begin{tblr}{|c||c|}
            \hline
            \backslashbox{\!$j$\!}{\!$i$\!} & $0$ \\
            \hline
            \hline
            $3$  & $ \Rone $ \\
            \hline
            $1$  & $ \Rmor{3} $ \\
            \hline
            $-1$  & $ \Rmor{3} $ \\
            \hline
            $-3$  & $ \Rone $ \\
            \hline
            \end{tblr}\end{tiny}
            \caption{$H(\hat{1})$.}\label{tab:H(1)}
    \end{minipage} 
    \begin{minipage}[b]{.3\linewidth}
      \centering
      \captionsetup{width=.8\linewidth}
        \begin{tiny}
           \begin{tblr}{|c||c|}
            \hline
            \backslashbox{\!$j$\!}{\!$i$\!} & $0$ \\
            \hline
            \hline
            $2$  & $ \Rone $ \\
            \hline
            $0$  & $ \Rmor{2} $ \\
            \hline
            $-2$  & $ \Rone $ \\
            \hline
            \end{tblr}\end{tiny}
            \caption{$H(\widehat{\sigma_1})$.}\label{tab:H(a)}
            
    \end{minipage} 
    \begin{minipage}[b]{.3\linewidth}
      \centering
       \captionsetup{width=.8\linewidth}
       \begin{tiny}
        \begin{tblr}{|c||c|c|c|}
            \hline
            \backslashbox{\!$j$\!}{\!$i$\!} & $0$ & $1$ & $2$ \\
            \hline
            \hline
            $7$  &   &   & $ \Rone $ \\
            \hline
            $5$  &   &   & $ \Rmor{2} $ \\
            \hline
            $3$  & $ \Rone $ &   & $ \Rone $ \\
            \hline
            $1$  & $ \Rmor{2} $ &   &   \\
            \hline
            $-1$  & $ \Rone $ &   &   \\
            \hline
        \end{tblr}   \end{tiny}
        \caption{$H(\widehat{\sigma_1^2})$.}\label{tab:H(aa)}
     
    \end{minipage}%
    
\end{table}

\begin{table}[!htb]
    \centering
    \begin{minipage}[b]{.3\linewidth}
      \centering
      \captionsetup{width=.8\linewidth}
        \begin{tiny}
           \begin{tblr}{|c||c|}
            \hline
            \backslashbox{\!$j$\!}{\!$i$\!} & $0$ \\
            \hline
            \hline
            $1$  & $ \Rone $ \\
            \hline
            $-1$  & $ \Rone $ \\
            \hline
         \end{tblr}  \end{tiny}
         \caption{$H(\widehat{\sigma_1\sigma_2})$.}\label{tab:H(ab)}
    \end{minipage} 
    \begin{minipage}[b]{.3\linewidth}
      \centering
      \captionsetup{width=.8\linewidth}
        \begin{tiny}
           \begin{tblr}{|c||c|c|c|c|c|}
                \hline
                \backslashbox{\!$j$\!}{\!$i$\!} & $0$ & $1$ & $2$ & $3$ & $4$ \\
                \hline
                \hline
                $11$  &   &   &   &   & $ \Rone $ \\
                \hline
                $9$  &   &   &   &   & $ \Rone $ \\
                \hline
                $7$  &   &   & $ \Rmor{2} $ &   &   \\
                \hline
                $5$  &   &   & $ \Rmor{2} $ &   &   \\
                \hline
                $3$  & $ \Rone $ &   &   &   &   \\
                \hline
                $1$  & $ \Rone $ &   &   &   &   \\
                \hline
            \end{tblr}  \end{tiny}
            \caption{$H(\widehat{\sigma_1^2\sigma_2^2})$.}\label{tab:H(aabb)}
    \end{minipage} 
    \begin{minipage}[b]{.3\linewidth}
      \centering
      \captionsetup{width=.8\linewidth}
       \begin{tiny}
        \begin{tblr}{|c||c|c|c|}
            \hline
            \backslashbox{\!$j$\!}{\!$i$\!} & $0$ & $1$ & $2$ \\
            \hline
            \hline
            $6$  &   &   & $ \Rone $ \\
            \hline
            $4$  &   &   & $ \Rone $ \\
            \hline
            $2$  & $ \Rone $ &   &   \\
            \hline
            $0$  & $ \Rone $ &   &   \\
            \hline
        \end{tblr} \end{tiny}
        \caption{$H(\widehat{\Delta})$.}\label{tab:H(aba)}
	\end{minipage}
\end{table}

\begin{table}[!htb]
    \begin{minipage}[b]{.49\linewidth}
      \centering
       \begin{tiny}
\begin{tblr}{|c||c|c|c|c|c|c|c|}
                \hline
                \backslashbox{\!$j$\!}{\!$i$\!} & $0$ & $1$ & $2$ & $3$ & $\cdots$ \\
                \hline
                \hline
                $\vdots$             &   &   &   &  & \SetCell[r=4]{c}{$W_{\widehat{\beta}}$}      \\
                \cline{1-5}
                $\underline{j}+10$  &   &   &   & $ \Rone $ &   \\
                \cline{1-5}
                $\underline{j}+8$  &   &   &   & $ \Rone \oplus \Tone{2}  $ &   \\
                \cline{1-5}
                $\underline{j}+6$  &             &   &  $ \Rone $   & $ \Tone{2} $ &   \\
                \hline
                $\underline{j}+4$  & $ \Rone $   &   & $ \Rone $ & &   \\
                \hline
                $\underline{j}+2$  & $ \Rone^2 $ &   &   & &    \\
                \hline
                $\underline{j}$    & $ \Rone $   &   &   & &  \\
                \hline
                \end{tblr} \end{tiny}
        \caption{$H(\widehat{\beta})$ with $\beta\in \mathbf{C1}$.}\label{table:Kh_case_C1}
    \end{minipage}%
    \begin{minipage}[b]{.49\linewidth}
      \centering
        \begin{tiny}
            \begin{tblr}{|c||c|c|c|c|c|c|c|}
                \hline
                \backslashbox{\!$j$\!}{\!$i$\!} & $0$ & $1$ & $2$ & $3$ & $\cdots$ \\
                \hline
                \hline
                $\vdots$             &   &   &   &  & \SetCell[r=3]{c}{$X_{\widehat{\beta}}$}         \\
                \cline{1-5}
                $\underline{j}+8$  &   &   &   & $ \Rone  $ & \\
                \cline{1-5}
                $\underline{j}+6$  &             &   &   $ \Rone  $  & $\Tone{2} $ & \\
                \hline
                $\underline{j}+4$  &           &   & $ \Rone^2 $ & &   \\
                \hline
                $\underline{j}+2$  & $ \Rone $ &   &   & &    \\
                \hline
                $\underline{j}$    & $ \Rone $   &   &   & &    \\
                \hline
            \end{tblr} \end{tiny}
            \caption{$H(\widehat{\beta})$ with $\beta\in \mathbf{C2}$.}\label{table:Kh_case_C2}
    \end{minipage} 
\end{table}

\begin{table}[!htb]
    \begin{minipage}[b]{0.49\textwidth}
        \centering
        \begin{tiny}
        \begin{tblr}{|c||c|c|c|c|c|c|c|}
            \hline
            \backslashbox{\!$j$\!}{\!$i$\!} & $0$ & $1$ & $2$ & $3$ & $\cdots$ \\
            \hline
            \hline
            $\vdots$             &   &   &   &  &    \SetCell[r=3]{c}{$Y_{\widehat{\beta}}$}    \\
            \cline{1-5}
            $\underline{j}+8$  &   &   &   & $ \Rone^2  $ &  \\
            \cline{1-5}
            $\underline{j}+6$  &  & &     & $(\Tone{2})^2$  & \\
            \hline
            $\underline{j}+4$  &        &  & $ \Rone^2 $ & &   \\
            \hline
            $\underline{j}+2$  & $ \Rone $ &  &    & &    \\
            \hline
            $\underline{j}$    & $ \Rone $   &    &   & &    \\
            \hline
        \end{tblr} \end{tiny}
        \caption{$H(\widehat{\beta})$ with $\beta\in \mathbf{C3}$.}\label{table:Kh_case_C3}
    \end{minipage} 
    \begin{minipage}[b]{0.49\textwidth}
        \centering
        \begin{tiny}
        \begin{tblr}{|c||c|c|c|c|c|c|c|}
            \hline
            \backslashbox{\!$j$\!}{\!$i$\!} & $0$ & $1$ & $2$ & $3$ & $\cdots$ \\
            \hline
            \hline
            $\vdots$             &   &   &   &  & \SetCell[r=3]{c}{$Z_{\widehat{\beta}}$}        \\
            \cline{1-5}
            $\underline{j}+8$  &   &   &   & $ \Rone $ &  \\
            \cline{1-5}
            $\underline{j}+6$  &   &   &   & $ \Tone{2} $ & \\
            \hline
            $\underline{j}+4$  &   &   & $ \Rone $ & &   \\
            \hline
            $\underline{j}+2$  & $ \Rone $ &   &   & &   \\
            \hline
            $\underline{j}$    & $ \Rone $ &   &   & &   \\
            \hline
        \end{tblr} \end{tiny}
        \caption{$H(\widehat{\beta})$ with $\beta\in \mathbf{C4}$.}\label{table:Kh_case_C4}
    \end{minipage}
\end{table}

Several authors have already analyzed Khovanov homology of certain families of braid positive links on $3$ strands: in \cite{Turner_2008} Turner computed Khovanov homology of the torus links $T(3,q)$ with coeffients in $\mathbb{Q}$ or $\mathbb{Z}_p$ for an odd prime $p$ (compare to the work by Stošić about Khovanov homology of torus knots \cite{Stosic_2007} and to that of Benheddi about reduced Khovanov homology of $T(3,q)$ over $\mathbb{Z}_2$ \cite{Benheddi_2017}). Chandler, Lowrance, Sazdanovi\'c and Summers \cite{Chandler_Lowrance_Sazdanovic_Summers_2022} computed Khovanov homology for those links arising as closures of the first four families $\Omega_0 - \Omega_3$ of 3-braids in Murasugi's classification (i.e., those corresponding to powers of the half twist $\Delta = \sigma_1\sigma_2\sigma_1$ concatenated with a unique simple factor), while in \cite{Przytycki_Silvero_2024} Przytycki and Silvero determined the extreme Khovanov homology groups (i.e., those corresponding to the lowest row in the Khovanov homology table) for closed braids on at most $4$ strands. See also \cite{Lowrance_2011}. Our results are complementary to the aforementioned previous works: we analyze all closed positive $3$-braids (and not just some specific families), and we completely describe the first four columns and three lowest rows of their Khovanov homology tables over $\mathbb{Z}$. Moreover, combining Theorem \ref{th:main} with a result by Jaeger \cite{Jaeger_2011} we can explicitly determine the Khovanov homology table of closed positive 3-braids for a \textit{bigger} number of columns and rows. The exact part of the table that is determined depends on the \textit{infimum} of the braid, which is a parameter appearing in its Garside normal form (see Definition~\ref{Th:left_normal_form}). More precisely, we obtain the following:

\begin{theorem}\label{main2}
    Let $\beta$ be a positive 3-braid and write $p=\inf(\beta)$. Define the quantities $\mathbf{i}(p) =  4 \lfloor \frac{p}{2} \rfloor + 3$ and $\mathbf{j}(p)= \underline{j}(\widehat{\beta}) +6 \lfloor \frac{p}{2} \rfloor + 4$. Then, $H^{i,j}(\widehat{\beta})$ for every $(i,j)$ with $i=0,1, \dots, \mathbf{i}(p)$ or $j=\underline{j}(\widehat{\beta}), \underline{j}(\widehat{\beta}) + 2, \dots, \mathbf{j}(p)$ is as shown in one of the Tables \ref{tab:Delta_p_even}--\ref{tab:Delta_p_C4}. The precise table corresponding to $\widehat{\beta}$ can be deduced from its normal form.
\end{theorem}

Our results also provide a criterion to obstruct braid positivity among links with braid index $3$. In particular, as a consequence of the fact that any positive braid link of braid index $3$ is realized as a closed positive braid on $3$ strands \cite[Th. 1.3]{Stoimenow_2017}, we establish the following result:

\begin{corollary} Let $L$ be a link whose Khovanov homology is not as any of Tables 1--10. Then:
\begin{enumerate}[label=(\roman*),topsep=0pt]
\item If $L$ is braid positive, then its braid index is at least $4$.
\item If the braid index of $L$ is $3$, then $L$ is not braid positive.
\end{enumerate}
\end{corollary}

It is worth mentioning that the infimum of a braid is not preserved under conjugation. The maximal value of the infimum in the conjugacy class of a braid $\beta$ is called the \emph{summit infimum} of $\beta$, and it is well known that a braid is conjugate to a positive braid if and only if its summit infimum is non-negative. One can compute this value by using standard procedures from Garside theory. In this paper we will also show how to conjugate a 3-braid $\beta$ to a braid in a suitable family, whose infimum is the summit infimum of $\beta$, in linear time. By doing so, we can determine the maximal number of cells from Theorem~\ref{main2}. From the other perspective, this result implies that part of the Garside structure of a given $3$-braid is captured in the Khovanov homology table of its closure. 

Despite their conceptually simple definitions, computing the Jones polynomial and (therefore) the Khovanov homology of an arbitrary link is an NP-hard problem. However, when the link is given as a closed braid with a fixed number of strands, the computation of its Jones polynomial can be performed in polynomial time with respect to the number of crossings \cite{Morton_Short_1987}. A similar result is conjectured for Khovanov homology \cite{Przytycki_Silvero_2024}. As an application of our results, we prove that:

\begin{proposition}\label{prop:complexity}
The $\Lcorner_{4\lfloor p/2\rfloor+4,3\lfloor p/2\rfloor+3}$-shape of the Khovanov homology of the closure of a 3-braid with summit infimum $p\geq 0$ can be computed in linear time.
\end{proposition}

Our proofs are based in two main steps: the first one consists of analyzing the Garside structure of $3$-braids to obtain a classification into $5$ families, up to conjugation, and a description of their normal forms (Proposition \ref{prop:representatives_conjugacy_classes_5_families}). Then, given a positive $3$-braid in one of the former families represented by a word $w$, we use Khovanov skein exact sequence and the associated long exact sequence on homology to describe the homology of the associated link represented by $D=\widehat{w}$ in terms of the homology of two simpler diagrams $D_A$ and $D_B$. We select the special crossing in $D$ in such a way that $D_B$ represents a rational (and therefore) alternating link. In Proposition \ref{prop:w(D_R)_w(D_R')} we present an algorithm which recibes as input a standard diagram of a rational link and produces an equivalent alternating rational diagram, while keeping track of the relation on the number of positive and negative crossings in both diagrams. This result will be crucial in the proof of Theorem \ref{th:main}.

Our techniques and results can be dualized to the case of braid negative links. In that case, each homology table would be mirrored (see~\cite[Cor. 11]{Khovanov_2000}). In particular, Theorem \ref{th:main} would explicitly describe the last $4$ columns and the uppermost $3$ rows of the Khovanov homology table of the given (negative) link.

The plan of the paper is as follows: In Section~\ref{sec:braid_groups} we recall some preliminaries on braid groups and their Garside structure, while in Section~\ref{sec:khovanov_homology_semi-adequate} we briefly review the definition of Khovanov homology, with an emphasis on some particularities about semiadequate links. In Section~\ref{sec:rational_links} we present our algorithm transforming a rational link diagram into an equivalent rational alternating diagram, and analyze the relation among the writhes of both diagrams. This result is fundamental in the proof of Theorem~\ref{th:main}, that we defer to Section \ref{sec:L43-shape}. Finally, in Section \ref{sec:final} we prove Theorem~\ref{main2} and Proposition~\ref{prop:complexity}.

\section{Garside structure of braid groups}\label{sec:braid_groups}

In this work, we will use the classical presentation of the braid group on $n \geq 1$ strands, $$ \mathbb{B}_n=\left\langle\sigma_1, \ldots, \sigma_{n-1} \left\lvert\, \begin{array}{cc}
\sigma_i \sigma_j=\sigma_j \sigma_i, & |i-j|>1 \\
\sigma_i \sigma_j \sigma_i=\sigma_j \sigma_i \sigma_j, & |i-j|=1
\end{array}\right.\right\rangle ,$$
introduced by Artin in \cite{Artin_1925}. Every generator $\sigma_i$ and its inverse $\sigma_i^{-1}$ correspond to the geometric braids depicted in Figures \ref{fig:sigma_i} and \ref{fig:sigma_i_inv}, respectively.
\begin{figure}[h!]
\centering
\begin{subfigure}{.2\textwidth}
  \centering
  \includegraphics[width=.75\linewidth]{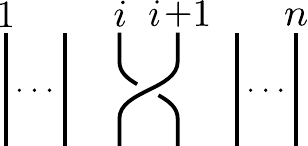}
  \caption{$\sigma_i$}\label{fig:sigma_i}
\end{subfigure}
{\hspace{0.5cm}}
\begin{subfigure}{.2\textwidth}
  \centering
  \includegraphics[width=.75\linewidth]{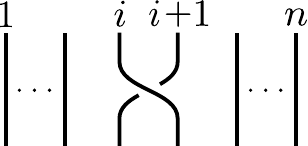}
  \caption{$\sigma_i^{-1}$}\label{fig:sigma_i_inv}
\end{subfigure}
\caption{Artin generators and their inverses.}\label{fig:artin_generators_inverses}
\end{figure}

Each braid $\beta \in \mathbb{B}_n$ can be represented by infinitely many words in the Artin generators and their inverses. The braids that can be represented by a word involving only non-negative powers of the generators are called \emph{positive braids}. The subset of $\mathbb{B}_n$ consisting of all positive braids is a submonoid denoted by $\mathbb{B}_n^+$. Given $\alpha, \beta \in \mathbb{B}_n$, we say that $\alpha \preccurlyeq \beta $ if there exists $\gamma \in \mathbb{B}_n^+$ such that $\alpha \gamma = \beta $. If this is the case, we say that $\alpha$ is a {\em prefix} of $\beta$. This is a lattice order of $\mathbb{B}_n$. In particular, given two $n$-braids $\alpha, \beta$, they admit a unique greatest common divisor $\alpha\wedge \beta$ with respect to $\preccurlyeq$.

There are some parameters of braid words that will play an important role in subsequent sections. The \emph{length of a word} $w$, $l(w)$, counts the number of letters that appear in it. Since the relations in the standard presentation of $\mathbb{B}_n$ are homogeneous, all positive words that represent a positive braid have the same length. The \emph{length of a braid} $\beta \in \mathbb{B}_n^+$, denoted $l(\beta)$, can be defined as the length of any of its positive representatives. The \emph{syllable length} of a word $w=\sigma_{i_1}^{k_1} \cdots \sigma_{i_m}^{k_m}$ with each $k_r \neq 0$, and $\sigma_{i_{r}} \neq \sigma_{i_{r+1}}$ for $r=1, \dots, m-1$, is defined as $\operatorname{sl}(w)=m$.

The Garside element $$\Delta = \sigma_1 (\sigma_2 \sigma_1) \cdots (\sigma_{n-1} \cdots \sigma_2 \sigma_1) \in \mathbb{B}_n$$ satisfies $\sigma_i \Delta^{\pm 1} = \Delta^{\pm 1} \sigma_{n-i}$ for every $i=1, \dots, n-1$, and $\Delta = \sigma_i x_i$ for some $x_i \in \mathbb{B}_n^+$, for each $i=1, \dots, n-1$. Then, given an $n$-braid word, we can replace each letter $\sigma_i^{-1}$ by $x_i \Delta^{-1}$, and slide all occurrences of $\Delta^{-1}$ to the beginning of the word, obtaining another representative of the form $\Delta^p x$, where $p \in \mathbb{Z}$ and $x \in \mathbb{B}_n^+$~\cite{Garside_1969}. In~\cite{Elrifai_Morton_1994} and \cite{Epstein_1992} this was refined, leading to the following notion.
\begin{definition}\label{Th:left_normal_form}
    The \emph{(left) normal form} of $\beta \in \mathbb{B}_n$ is the unique decomposition $\beta = \Delta^p a_1 \cdots a_{\ell}$ so that $p \in \mathbb{Z}$, $\ell \geq 0$, $1 \prec a_k \prec \Delta$ for every $k = 1, \dots, \ell$, and $(a_k a_{k+1}) \wedge \Delta = a_k$ for every $k = 1, \dots, \ell - 1$. Given such a decomposition, each $a_i$ is called a 
 \emph{simple factor} of $\beta$. Moreover, the \emph{infimum} and the \emph{supremum} of $\beta$ are defined as $\inf (\beta)=p$ and $\sup (\beta)=p+\ell$, respectively.
\end{definition}

Given a braid $\beta\in \mathbb{B}_n$, consider its conjugacy class $\beta^{\mathbb{B}_n}$. The \emph{summit infimum} of $\beta$ is defined as $$ {\inf}_s(\beta) = \max \{ \inf(\alpha) \; | \; \alpha \in \beta^{\mathbb{B}_n} \}.$$  
Almost directly from the foundational work of Garside  (see~{\cite[Sec. 4]{Garside_1969}} and \cite[Sec. 1]{Elrifai_Morton_1994} for more details), it follows that this value actually exists, and in {\cite{Elrifai_Morton_1994}} it is shown that \[\operatorname{SS}(\beta) = \{\alpha \in \beta^{\mathbb{B}_n} \mid \inf (\alpha)={\inf}_s(\beta) \}\] is a non-empty and finite subset of $\beta^{\mathbb{B}_n}$. This subset is known as the \emph{summit set} of $\beta$, and its study has been extensive, with the aim of developing efficient algorithms to solve the conjugacy problem in braid groups. We recall that a braid $\beta$ is conjugate to a positive braid if and only if $\inf_s(\beta)\geq 0$.

\subsection{Left normal form of $3$-braids} Given a braid $\beta \in \mathbb{B}_3$, the structure of its normal form is particularly simple. Since $\Delta = \sigma_1 \sigma_2 \sigma_1 = \sigma_2 \sigma_1 \sigma_2$, the positive prefixes of $\Delta$ are $1$, $\sigma_1$, $\sigma_2$, $\sigma_1\sigma_2$, $\sigma_2\sigma_1$ and $\Delta$, and therefore the only possible simple factors in a left normal form are $\sigma_1$, $\sigma_2$, $\sigma_1\sigma_2$, and $\sigma_2\sigma_1$. Each of these four braids has a unique positive word representing it. Therefore, for these four factors, we will sometimes make no distinction between the braid and the positive word.

Observe that if the infimum of a $3$-braid is $0$, then there is a unique word representing its normal form. We define the \emph{syllable length} of such a braid as the syllable length of that word. For instance, if $\beta = \sigma_1.\sigma_1\sigma_2.\sigma_2.\sigma_2.\sigma_2\sigma_1$, where the dots separate the factors in its normal form, then $\operatorname{sl}(\beta) = 3$, as $\beta= \sigma_1^2 \sigma_2^4 \sigma_1$.

The condition $(a_k a_{k+1}) \wedge \Delta = a_k$ in the definition of the left normal form is equivalent to the fact that $a_k$ is the largest positive prefix of $\Delta$ in any decomposition of $a_ka_{k+1}$ as a product of two positive prefixes of $\Delta$. In such a case, we say that the decomposition $a_ka_{k+1}$ is \emph{left weighted}. Note that in $\mathbb{B}_3$ two simple factors are left weighted if and only if the last letter of the first factor coincides with the first letter of the second factor. 

As a consequence of the previous discussion, a word representing a 3-braid is in normal form if and only if it belongs to one of the following families:
\begin{enumerate}[label=(\roman*)]
\item $\Delta^p \sigma_i^{k}$, with $p \in \mathbb{Z}$, $i \in \{1,2\}$ and $k \geq 0$;
\item $ \Delta^p \sigma_i^{k_1} \sigma_{j}^{k_2}\sigma_i^{k_3} \cdots  \sigma_j^{k_{2t}}$, with $p \in \mathbb{Z}$, $\{i,j\} = \{1,2\}$, $t \geq 1$, $k_1, k_{2t} \geq 1$ and  $k_2, \dots, k_{2t-1} \geq 2$;
\item $ \Delta^p \sigma_i^{k_1} \sigma_{j}^{k_2}\sigma_i^{k_3} \cdots  \sigma_i^{k_{2t+1}}$, with $p \in \mathbb{Z}$, $\{i,j\} = \{1,2\}$, $t \geq 1$, $k_1, k_{2t+1} \geq 1$ and $k_2, \dots, k_{2t} \geq 2$.
\end{enumerate} 

Observe that, in the above description, we are not indicating the simple factors of the normal form, but collecting powers of the generators instead.

There exist several classifications of 3-braids up to conjugation (see, for example, \cite[Prop. 2.1]{Murasugi_1974} and \cite[Prop. 3.2]{Truol_2023}). The next result provides a classification of 3-braids that will be useful in the proof of Theorem \ref{th:main}.

\begin{proposition}\label{prop:representatives_conjugacy_classes_5_families}
Every braid in $\mathbb{B}_3$ is conjugate to a braid in one of the following families:
\begingroup
\renewcommand{\arraystretch}{1.8}
$$
\begin{array}{l}
\Lambda_1= \left\{\Delta^p \ | \ p \in \mathbb{Z} \right\}, \\
\Lambda_2= \left\{\Delta^p\sigma_1^{k_1} \ | \ p\in \mathbb{Z},\ k_1>0\right\}, \\
\Lambda_3= \left\{\Delta^{2u}\sigma_1\sigma_2 \ | \ u \in \mathbb{Z} \right\}, \\
\Lambda_4= \left\{\Delta^{2u}\sigma_1^{k_1}\sigma_2^{k_2}\cdots \sigma_2^{k_{2t}} \ | \ u \in \mathbb{Z},\ t>0,\ k_1,\ldots,k_{2t}\geq 2\right\}, \\
\Lambda_5= \left\{\Delta^{2u+1}\sigma_1^{k_1}\sigma_2^{k_2}\cdots \sigma_1^{k_{2t+1}} \ | \ u \in \mathbb{Z},\ t>0,\ k_1,\ldots,k_{2t+1}\geq 2\right\}. 
\end{array}
$$
\endgroup
\end{proposition}
\begin{proof}
Let $\beta$ be a 3-braid with $\inf(\beta)=p$. Its left normal form will be $\Delta^p a_1\cdots a_{\ell}$, with $p \in \mathbb{Z}$ and $\ell \geq 0$.

Every simple factor $a_i$ can be written in a unique way as a positive braid word, and the last letter of $a_i$ is equal to the first letter of $a_{i+1}$ for every $i=1,\ldots,\ell-1$. Moreover, up to conjugation by $\Delta$ (which swaps $\sigma_1$ and $\sigma_2$), we can assume that the first letter of $a_1$ is $\sigma_1$. Hence $\beta$ can be written as
$$
    \Delta^p \sigma_{[1]}^{k_1}\sigma_{[2]}^{k_2}\cdots \sigma_{[m]}^{k_m},
$$
where $\sigma_{[i]}$ means $\sigma_1$ if $i$ is odd and $\sigma_2$ if $i$ is even. Furthermore, if $m>0$, then $k_1,k_m\geq 1$ and $k_2,\ldots,k_{m-1}\geq 2$.

We will show the result by induction on the number of letters $n_{\beta}=k_1+\cdots+k_m$ in the non-$\Delta$ part of the above expression. 

If $n_{\beta}=0$ then $\beta=\Delta^p \in \Lambda_1$ and the result holds. If $n_{\beta}=1$ then $\beta=\Delta^p\sigma_1\in\Lambda_2$. If $n_{\beta}=2$ then $\beta$ is equal to either $\Delta^p\sigma_1^2$ or $\Delta^p\sigma_1\sigma_2$. In the former case, $\beta\in \Lambda_2$. In the latter case, if $p$ is even $\beta \in \Lambda_3$, and if $p$ is odd we can conjugate it by $\sigma_2$ to obtain $\Delta^p \sigma_1^2$, which belongs to $\Lambda_2$.

Therefore, we can assume that $n_{\beta}\geq 3$ and that the result holds for braids $\gamma$ with $n_\gamma < n_\beta$. We distinguish the following cases:
\begin{enumerate}

\item If $m=1$, then $\beta\in \Lambda_2$.

\item If $m>1$, we have:
\begin{itemize}

 \item If $m-p$ is odd, then $\sigma_{[m]}$ conjugated by $\Delta^p$ equals $\sigma_{[1]}$. We can then conjugate $\beta$ by $\sigma_{[m]}^{k_m}$ and we obtain
     $$
       \gamma = \Delta^p \sigma_{[1]}^{k_1+k_m}\sigma_{[2]}^{k_2}\cdots \sigma_{[m-1]}^{k_{m-1}},
     $$
      which belongs to either $\Lambda_4$ (if $p$ is even), or $\Lambda_2$ (if $p$ is odd and $m=2$) or $\Lambda_5$ (if $p$ is odd and $m>2$).

  \item If $m-p$ is even, $k_1=1$ and $k_m=1$, then conjugating $\beta$ by $\sigma_{[m]}$ we create a new $\Delta$ factor, we obtain $\gamma=\Delta^{p+1}\sigma_{[2]}^{k_2-1}\sigma_{[3]}^{k_3}\cdots \sigma_{[m-1]}^{k_{m-1}}$ and $n_{\gamma} = n_{\beta} - 3$. 
               
  \item If $m-p$ is even, $k_1=1$ and $k_m>1$, we can conjugate $\beta$ by $\sigma_{[m]}$, creating a new $\Delta$ factor, obtaining $\gamma=\Delta^{p+1}\sigma_{[2]}^{k_2-1}\sigma_{[3]}^{k_3}\cdots \sigma_{[m-1]}^{k_{m-1}}\sigma_{[m]}^{k_m-1}$ and $n_{\gamma} = n_{\beta} - 3$.
  
  \item If $m-p$ is even, $k_1>1$ and $k_m=1$, then conjugating $\beta$ by $\sigma_{[m-1]}\sigma_{[m]}$, we create a new $\Delta$ factor, we get $\gamma=\Delta^{p+1}\sigma_{[1]}^{k_1-1}\sigma_{[2]}^{k_2}\cdots \sigma_{[m-1]}^{k_{m-1}-1}$ and $n_{\gamma} = n_{\beta} - 3$.
  
  \item If $m-p$ is even, $k_1>1$ and $k_m>1$, then $\beta\in \Lambda_4$ if $p$ is even, and $\beta\in \Lambda_5$ if $p$ is odd.

\end{itemize}
\end{enumerate} 
Therefore, every braid either belongs to some $\Lambda_i$ or can be conjugated to a braid $\gamma$ with $n_{\gamma} < n_{\beta}$. The result follows by induction hypothesis.
\end{proof}

\begin{remark}\label{remark:summit_infimum_Lambda_i}
It is straightforward to check that braids in the families $\Lambda_1$ to $\Lambda_5$ belong to their summit set, so they reach their summit infimum. One can see this by checking that iterated \emph{cycling} does not increase the infimum of the elements in each $\Lambda_i$ \cite{Elrifai_Morton_1994}.
\end{remark}

Observe that if we restrict Proposition \ref{prop:representatives_conjugacy_classes_5_families} to the monoid $\mathbb{B}_3^+$, the power of $\Delta$ in each of the families $\Lambda_i$ becomes a non-negative integer.



\section{Some highlights of  Khovanov homology and semiadequate links.}\label{sec:khovanov_homology_semi-adequate}

In this section, we give a brief description of Khovanov homology using the approach introduced by Viro in~\cite{Viro2_2004}. In addition, we review some properties concerning the Khovanov homology of semiadequate links, which will be useful in Section \ref{sec:L43-shape}. 

\subsection{Khovanov homology}\label{subsec:khovanov_homology} Let $D$ be a diagram of an oriented link, $c=c(D)$ its number of crossings, and assume that there is a fixed ordering for the crossings $x_1, \dots, x_c$  of $D$. Each crossing of $D$ can be \emph{smoothed} in two possible ways, $A$ or $B$, as shown in Figure \ref{fig:suavizado_A_B}. Given $s = (s_1, \dots, s_c) \in \{A,B\}^{c}$, we denote by $sD$ the system of circles obtained by performing a $s_k$-smoothing at crossing $x_k$. We will refer to $s$ (or possibly to $sD$) as a \emph{state}. It is common to denote the number of circles in $sD$ by $|sD|$, and the difference between the number of coordinates of $s$ being equal to $A$ and those being equal to $B$ by $\sigma(s)$.

\begin{figure}[h!]
\centering
\includegraphics[scale=0.6]{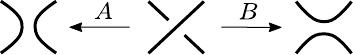}
\caption{$A$ and $B$-smoothing of a crossing.}
\label{fig:suavizado_A_B}
\end{figure}

For every state $s$, it is possible to assign a sign $\pm$ to each circle in $sD$, obtaining an \emph{enhancement} $S$ of~$s$. For every enhanced state $S$ (of $s$), we denote by $\tau(S)$ the difference between the number of circles labeled with $+$ and those labeled with $-$. We define the \textit{degrees}
\[ i(s)=i(S) = \frac{w(D)-\sigma(s)}{2} \quad \text{ and } \quad j(S)=\frac{3w(D)-\sigma(s)+2\tau(S)}{2}, \]
where $w(D)=p(D)-n(D)$ is the \emph{writhe} of $D$, with $p(D)$ and $n(D)$ the numbers of positive and negative crossings of $D$, respectively. The indices $i$ and $j$ are known as \emph{homological} and \emph{quantum index} (or \emph{grading}), respectively.

\begin{definition}\label{def:adjacent_state}
    Let $S$ and $T$ be enhanced states (of states $s$ and $t$, respectively) of an oriented link diagram $D$. We say that $T$ is \emph{adjacent} to $S$ if the following conditions are satisfied: 
    \begin{enumerate}[topsep=-1pt,itemsep=-1ex,partopsep=1ex,parsep=1ex]
\item $i(T)=i(S)+1$ and $j(T)=j(S)$.
\item The states $s$ and $t$ are identical except at one coordinate $k$ associated with the (change) crossing $x=x(s,t)$, where $s_k=A$ and $t_k=B$.
\item The signs assigned to the common circles in $sD$ and $tD$ are equal.
    \end{enumerate}
\end{definition}

As a consequence of Definition \ref{def:adjacent_state}, if $T$ is adjacent to $S$, then either two circles in $sD$ merge into a single circle in $tD$ or one circle in $sD$ splits into two circles in $tD$. The possibilities for the enhancements of the involved circles are $(++ \rightarrow +)$, $(+- \rightarrow -)$, $(-+ \rightarrow -)$, $(+ \rightarrow +-)$, $(+ \rightarrow -+)$ and $(- \rightarrow --)$. We will write $(S:T)=1$ if $T$ is adjacent to $S$ and $(S:T)=0$ otherwise. 

The \emph{Khovanov complex} can be set as follows: for every $i,j \in \mathbb{Z}$, define $C^{i,j}(D)$ as the free $\mathbb{Z}$-module with basis $\{ S \; | \; i(S)=i, j(S)=j \}$. Let us introduce the maps {$d^{i}: C^{i,j}(D) \longrightarrow C^{i+1,j}(D)$} given (on the generators) by $d^{i}(S) = \sum_T (-1)^{\kappa}(S:T)T$, where ${\kappa}$ is the number of $B$-coordinates of $S$ coming after the one that corresponds to the change crossing $x$. It turns out that $d^{i} \circ d^{i-1}=0$ and hence $(C^{*,*}(D),d^{*})$ is a chain complex. By construction, this can be seen as the direct sum (over $j$) of the family of subcomplexes $\{ (C^{*,j}(D),d^*) \}_{j \in \mathbb{Z}}$. 

Khovanov proved that the homology groups $H^{*,*}(D)$ of the Khovanov complex are link invariants~\cite[Sec. 5]{Khovanov_2000}. This allows us to write $H^{*,*}(L)=H^{*,*}(D)$, where $D$ is any diagram representing $L$. These groups are known as the \emph{Khovanov homology groups} of the link.

Given an oriented link diagram $D$, we set \[ j_{\min }(D)=\min \{j(S) \mid S \text { is an enhanced state of } D\} . \]
Similarly, one can define $j_{\max }(D)$. Let $s_A$ (resp. $s_B$) denote the state associating an $A$ label (resp. $B$ label) to every crossing of $D$.

\begin{proposition}[{\cite[Cor. 4.2]{Gonzalez-Meneses_Manchon_Silvero_2018}}]\label{prop:j_min_j_max}
    Let $D$ be an oriented link diagram with $c$ crossings, $n$ negative and $p$ positive. Then $j_{\min }(D)=c-3 n-\left|s_A D\right|$ and $j_{\max }(D)=-c+3 p+\left|s_B D\right|$.
\end{proposition}

Notice that $j_{\min}(D)$ is the minimal value of $j$ for which the complex $(C^{*,j}(D),d^*)$ is non-trivial, and its value depends on the precise diagram of the link. Indeed, given two diagrams $D$ and $D'$ representing the same link, one might have $j_{\min}(D)\neq j_{\min}(D')$. 

If we define $\underline{j}(D)$ as the minimum value of $j$ such that there exists a non-trivial group $H^{i, j}(D)$, then $j_{\min}(D) \leq \underline{j}(D)$. Since Khovanov homology is a link invariant, the value $\underline{j}(D)$ does not depend on the chosen diagram $D$ representing the link $L$, and therefore we can write $\underline{j}(L)=\underline{j}(D)$. In a similar manner, we can define $\underline{i}(L)=\underline{i}(D) = \min\{i \mbox{ } | \mbox{ } H^{i,j}(D) \neq 0\}$. 

Next we briefly recall the long exact sequence on Khovanov homology introduced by Viro in {\cite[Sec.~6.2]{Viro2_2004}}, which will be a key tool in the proof of Theorem \ref{th:main}. Given a crossing $e$ of a link diagram $D$, we denote by $D_A$ and $D_B$ the two diagrams obtained from $D$ by smoothing the crossing $e$ following an $A$ and $B$ label, respectively. There is a long exact sequence relating the (so-called framed) Khovanov homology of the three diagrams. 

In order to rewrite the long exact sequence in terms of (original) Khovanov homology, one needs to take into account the writhes of the involved diagrams. The gradings in the resulting sequence depend on the sign of the smoothed crossing $e$. In particular, if $e$ is positive, then $D_A$ inherits its orientation; however, one should choose an orientation for the diagram $D_B$. Then the long exact sequence  becomes (see {\cite[Lem. 2.2]{Mukherjee_Przytycki_Silvero_Wang_Yang_2017}}):\begin{equation}\label{slab_exact_sequence}     
\begin{array}{rlclclc}
\cdots & \longrightarrow H^{\frac{w\left(D_B\right)-w(D)-3}{2}+i, \frac{3\left(w\left(D_B\right)-w(D)\right)-1}{2}+j}\left(D_B\right) & \longrightarrow & H^{i-1, j}(D) & \longrightarrow & H^{i-1, j-1}\left(D_A\right) \\
& \longrightarrow H^{\frac{w\left(D_B\right)-w(D)-1}{2}+i, \frac{3\left(w\left(D_B\right)-w(D)\right)-1}{2}+j}\left(D_B\right) & \longrightarrow & H^{i, j}(D) & \longrightarrow & H^{i, j-1}\left(D_A\right) \\
& \longrightarrow H^{\frac{w\left(D_B\right)-w(D)+1}{2}+i, \frac{3\left(w\left(D_B\right)-w(D)\right)-1}{2}+j}\left(D_B\right) & \longrightarrow & \cdots \quad .
\end{array} 
    \end{equation}

\subsection{Semiadequate links} 
In \cite{Lickorish_Thistlethwaite_1988} Lickorish and Thistlethwaite introduced  (semi)adequate links as a genaralization of alternating links in the setting of their proof of the First Tait Conjecture. 

\begin{definition}
    A link diagram $D$ is \emph{$A$-adequate} (resp. \emph{B-adequate}) if for each crossing $e$ the two arcs obtained when performing an $A$-smoothing (resp. $B$-smoothing) to $e$ belong to different circles in $s_AD$ (resp. $s_BD$). The diagram $D$ is called \emph{semiadequate} if it is either $A$-adequate or $B$-adequate. If $D$ is both $A$-adequate and $B$-adequate, it is said to be \emph{adequate}. A link is said to be \emph{(semi)adequate} if it admits a (semi)adequate diagram.
\end{definition}

Examples of $A$-adequate diagrams include positive diagrams. Reduced alternating diagrams are adequate. 

Recall that $j_{\min}(D) \leq \underline{j}(D)$. If $D$ is $A$-adequate, the equality holds and $H^{*, j_{\min}}(D) = H^{-n, j_{\min}}(D) = \mathbb{Z}$, with $n$ the number of negative crossings\footnote{$A$-adequate diagrams minimize the number of negative crossings over all diagrams representing a link \cite[Cor. 5.14]{Lickorish_1997}.} in $D$. Khovanov homology of (semi)adequate links has been widely studied; see for example \cite{Asaeda_Przytycki_2004, Dasbach_Lowrance_2020, Przytycki_Silvero_2020}.


\section{Generating alternating diagrams of rational links}\label{sec:rational_links}

Rational links, also known as $2$-bridge links, are those that can be represented by a diagram having two minima and two maxima. Given $a_1, \ldots, a_m \in \mathbb{Z}$, we write $D(a_1, \ldots, a_m)$ for the diagram of a rational link depicted in Figure \ref{fig:standard_diagram_rational}, where a box labeled by $a_i \in \mathbb{Z}$ represents a sequence of $a_i$ twists as in Figure~\ref{fig:twist_positivo} if $a_i\geq0$, or $|a_i|$ twists as in Figure~\ref{fig:twist_negativo} if $a_i<0$. Observe that the parity of $m$ determines whether the diagram aligns with \ref{fig:racional_impar} or with \ref{fig:racional_par}. We will refer to these diagrams as \emph{standard rational diagrams}.

\begin{figure}[h!]
\centering
\begin{subfigure}[b]{.2\textwidth}
  \centering
  \includegraphics[width=.5\linewidth]{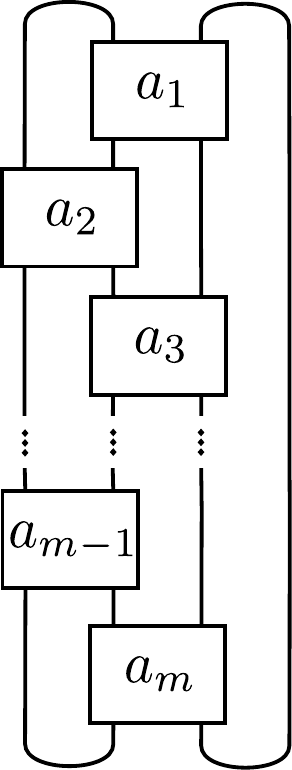}
  \caption{$m$ odd}\label{fig:racional_impar}
\end{subfigure}
{\hspace{0.5cm}}
\begin{subfigure}[b]{.2\textwidth}
  \centering
  \includegraphics[width=.5\linewidth]{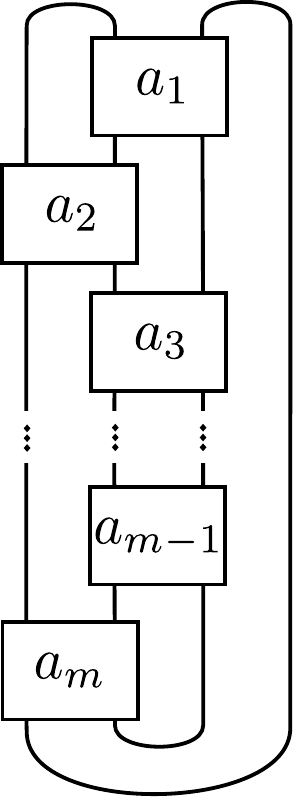}
  \caption{$m$ even}\label{fig:racional_par}
\end{subfigure}%
{\hspace{0.5cm}}
\begin{subfigure}[b]{.2\textwidth}
  \centering
  \includegraphics[width=.8\linewidth]{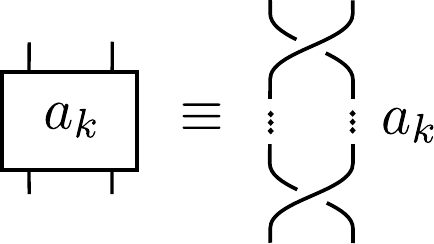}
  \caption{$a_k \geq 0$}\label{fig:twist_positivo}
\end{subfigure}%
{\hspace{0.5cm}}
\begin{subfigure}[b]{.2\textwidth}
  \centering
  \includegraphics[width=.8\linewidth]{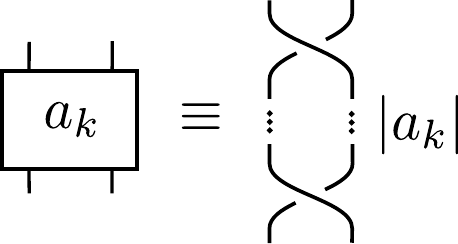}
  \caption{$a_k < 0$}\label{fig:twist_negativo}
\end{subfigure}%
\caption{Standard diagram of a rational link}\label{fig:standard_diagram_rational}
\end{figure}

\begin{remark}
    We will assume the convention $D(a_1, \dots, a_{i-1},0,a_{i+1}, \dots, a_m)=D(a_1, \dots, a_{i-1}+a_{i+1}, \dots, a_m).$
\end{remark}

It is well known that rational links are alternating \cite{Bankwitz_Schumann_1934,Goeritz_1934}. If we assume that every $a_i \neq 0$, standard rational diagrams are alternating if and only if $a_{i}a_{i+1} < 0$ for every $i=1, \dots, m-1$. In Proposition~\ref{prop:w(D_R)_w(D_R')}, we transform any given diagram $D=D(a_1, \ldots, a_m)$ with $a_i \geq 2$ into an equivalent alternating standard rational diagram $D'$, and determine the number of positive and negative crossings of $D'$ in terms of those of $D$. This result will be crucial in the proof of Theorem \ref{th:main}.

Proposition \ref{prop:w(D_R)_w(D_R')} will be based on two transformations that we introduce in the following lemmas.

\begin{lemma}\label{lemma:U_transformation}
    Let $D=D(a_1, \dots, a_m)$ be a standard rational diagram with $m > 1$ and $a_1,a_2 \geq 1$. Then, the diagram
$$
    \mathbf{U}(D) = D(a_1-1,-1, a_2-1, a_3, \dots, a_m)  \\
$$
   is equivalent to $D$. We say that $\mathbf{U}(D)$ is obtained from $D$ by applying a $\mathbf{U}$-transformation.
\end{lemma}
\begin{proof}
In Figure \ref{fig:U_transformation} we describe a $\mathbf{U}$-transformation as a combination of simpler transformations representing isotopies. Observe that in the third transformation, which resembles the action of ``opening a book'' as if the grey rectangle $R$ were the cover, $R$ consists of the entire diagram except a part of the arc connecting the boxes labeled by $a_1$ and $a_m$.  
\end{proof}
    \begin{figure}[h!]
        \centering
        \includegraphics[width=0.8\linewidth]{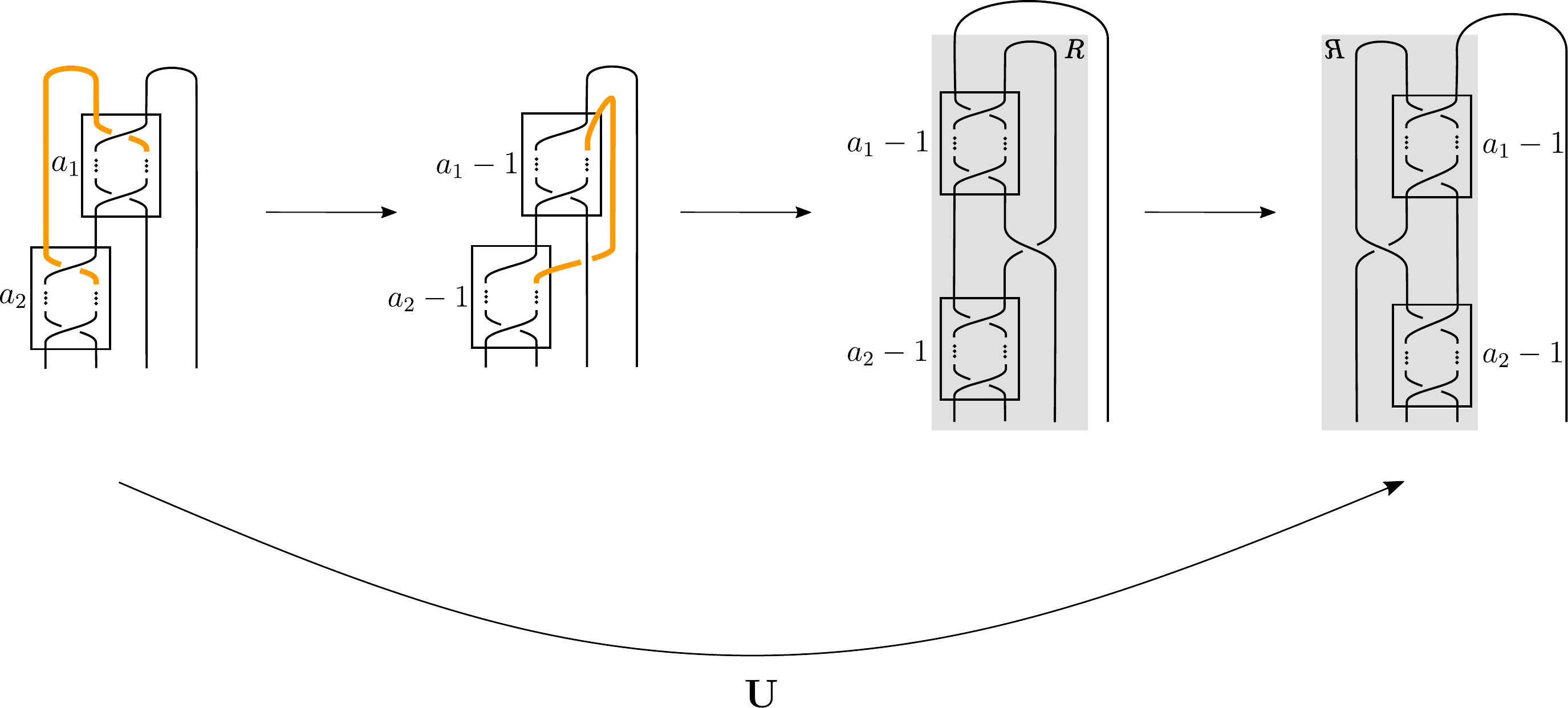}
        \caption{A description of the $\mathbf{U}$-transformation illustrating proof of Lemma~\ref{lemma:U_transformation}.}
        \label{fig:U_transformation}
    \end{figure}

\begin{lemma}\label{lemma:T_transformation}
    Let $D = D(a_1, \dots, a_{i-1}, -1, a_{i+1}, a_{i+2}, \dots, a_m)$ be a standard rational diagram with $m> 3$ and $i \in \{ 2, \dots, m-2\}$, such that $a_{i-1},a_{i+1},a_{i+2} \geq 1$. Then, the diagram  
$$
    \mathbf{T}_i(D) = D(a_1,\dots, \, a_{i-1}+1, \, a_{i+1}, \, -1, \,a_{i+2}-1, \dots, \, a_m)  \\
$$
   is equivalent to $D$. We say that $\mathbf{T}_i(D)$ is obtained from $D$ by applying a $\mathbf{T}_i$-transformation.
\end{lemma}
\begin{proof}
In Figure \ref{fig:transformation_T_i} we describe a $\mathbf{T}_i$-transformation as a combination of simpler transformations representing isotopies. If $i$ is even (resp. odd), the box labeled $a_{i-1}$ is located at the right (resp. left), as shown in case \ref{fig:T_i_even_transformation} (resp. case \ref{fig:T_i_odd_transformation}).  
\end{proof}
     \begin{figure}[!ht]
     \centering
        \begin{subfigure}{.47\textwidth}
            \centering
            \includegraphics[width=\linewidth]{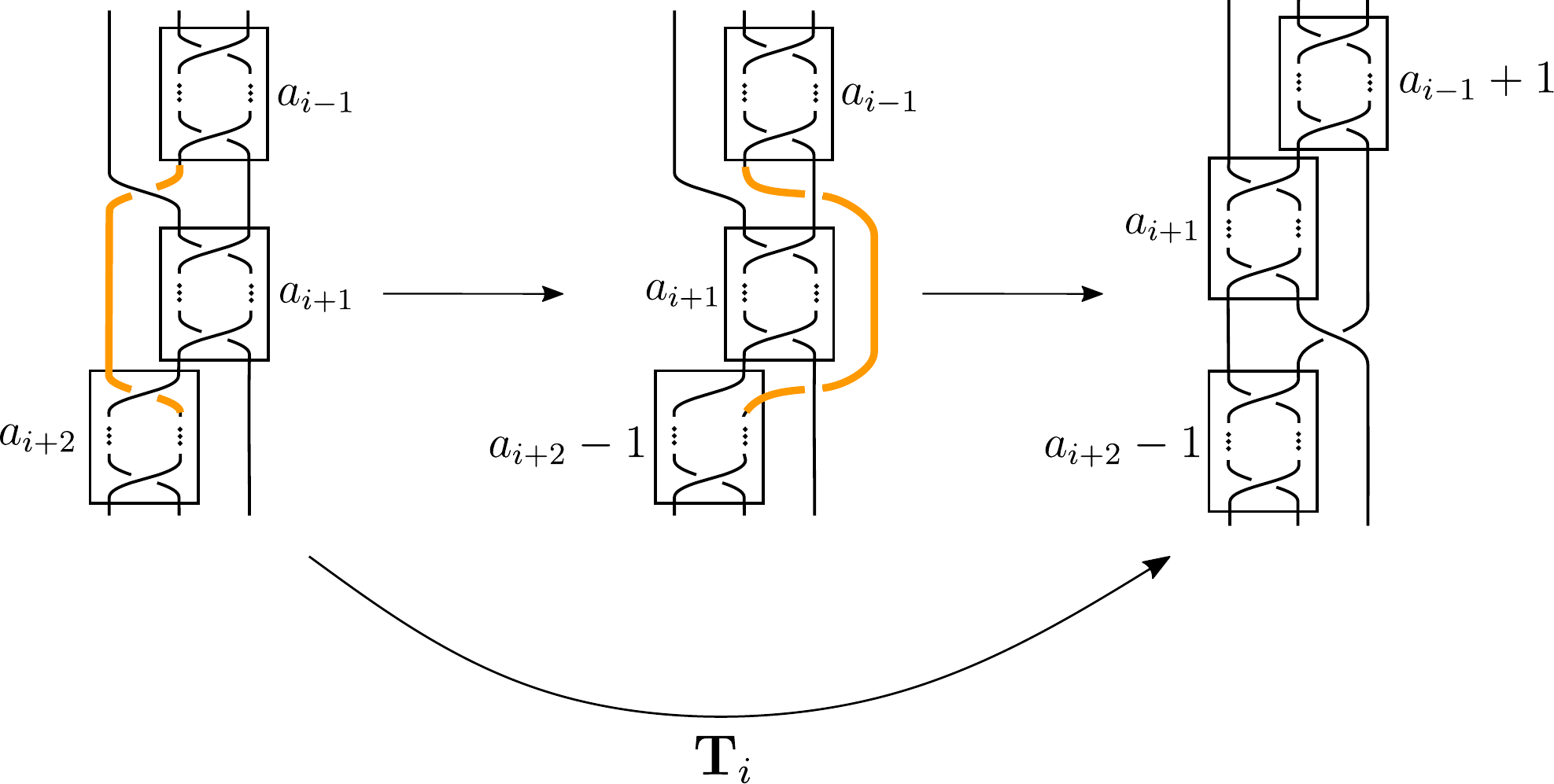}
            \caption{$i$ even}\label{fig:T_i_even_transformation}
        \end{subfigure}%
        {\hspace{0.3cm}}
        \begin{subfigure}{.47\textwidth}
            \centering
            \includegraphics[width=\linewidth]{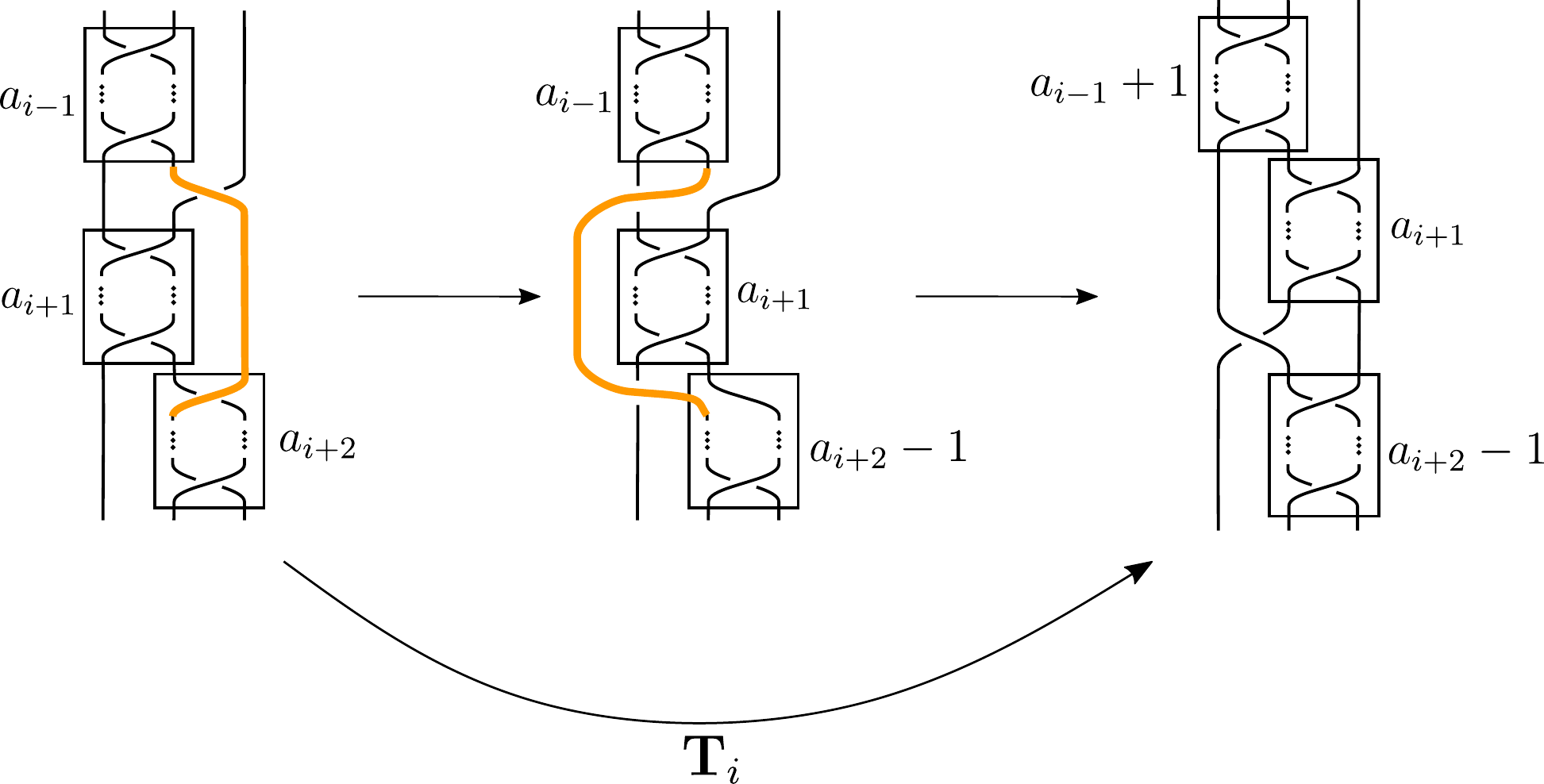}
            \caption{$i$ odd}\label{fig:T_i_odd_transformation}
        \end{subfigure}%
     \caption{A description of the $\mathbf{T}_i$-transformation illustrating proof of Lemma~\ref{lemma:T_transformation}.}\label{fig:transformation_T_i}
    \end{figure}

\begin{proposition}\label{prop:w(D_R)_w(D_R')}
    Let $D=D(a_1, \dots, a_m)$ be a standard rational diagram with $m>1$ and $a_1, \dots, a_m \geq 2$, and let $D'=D(a_1-1, -1, a_2-2, -1, a_3-2, -1, \dots, a_{m-1}-2,-1,a_m-1)$. Then, $D'$ is an alternating diagram equivalent to $D$ and the following relations hold: $p(D')=p(D)$, $n(D')=n(D)-(m-1)$ and $w(D')=w(D)+(m-1)$.
\end{proposition}
\begin{proof}
    $D'$ can be obtained from $D$ by applying a sequence of transformations $\mathbf{U}$ and $\mathbf{T}$, which preserve the equivalence class of the link. 
    
    $$ 
    \begin{array}{rclcl}
        D_{m} & = & D, \\
        
        D_{m-1} & = &(\mathbf{T}_{m-1} \circ \cdots \circ \mathbf{T_2} \circ \mathbf{U})(D_{m}) & = &  D (a_1, \dots, a_{m-2}, a_{m-1}-1, -1, a_{m}-1), \\
        
        D_{m-2} & = & (\mathbf{T}_{m-2} \circ \cdots \circ \mathbf{T_2} \circ \mathbf{U})(D_{m-1}) & = & D(a_1, \dots, a_{m-3}, a_{m-2} - 1, -1, a_{m-1}-2, -1, a_{m}-1), \\
        
        \hspace{-0.3cm} \vdots & & \hspace{1.4cm} \vdots & & \hspace{1.7cm} \vdots \\
         
        D_i & = & (\mathbf{T}_{i} \circ \cdots \circ \mathbf{T}_2 \circ \mathbf{U})(D_{i+1}) & = & D(a_1, \dots, a_{i-1}, a_i-1, -1, a_{i+1}-2, -1, \dots, a_{m-1}-2,-1,a_m -1), \\
        
      \hspace{-0.3cm} \vdots & & \hspace{1.4cm} \vdots & & \hspace{1.7cm} \vdots \\
        
        D_2 & = &  (\mathbf{T}_2 \circ \mathbf{U})(D_3) & = & D(a_1, a_2-1, -1, a_3 -2,-1, \dots, a_{m-1}-2, -1, a_m-1), \\
        
        D' & = & \mathbf{U}(D_2) & = & D(a_1-1, -1, a_2-2, -1, a_3-2, -1, \dots, a_{m-1}-2,-1,a_m-1).
        
    \end{array}
    $$

Next we analyze the signs of the crossings in $D$ and $D'$. Observe that in a standard rational diagram all crossings within the same box share the same sign; however, the sign of the crossings in a box labeled by $a_i$ depends on the global orientation of the diagram and therefore is not uniquely determined by the sign of $a_i$. 

A $\mathbf{U}$-transformation preserves the number of positive crossings and reduces the number of negative crossings by one, whereas a $\mathbf{T}_i$-transformation preserves the number of positive and negative crossings. This is shown in Figures~\ref{fig:U_transformation_signs} and~\ref{fig:T_1_transformation_signs}, where all possible orientations have been considered, up to reversing orientation. The blue sign towards a box indicates the sign of the crossings in that box. 

The result follows since $m-1$ transformations of type $\mathbf{U}$ are applied to obtain $D'$ from $D$.
\end{proof}

    \begin{figure}[h!]
    \centering
    \includegraphics[scale=0.25]{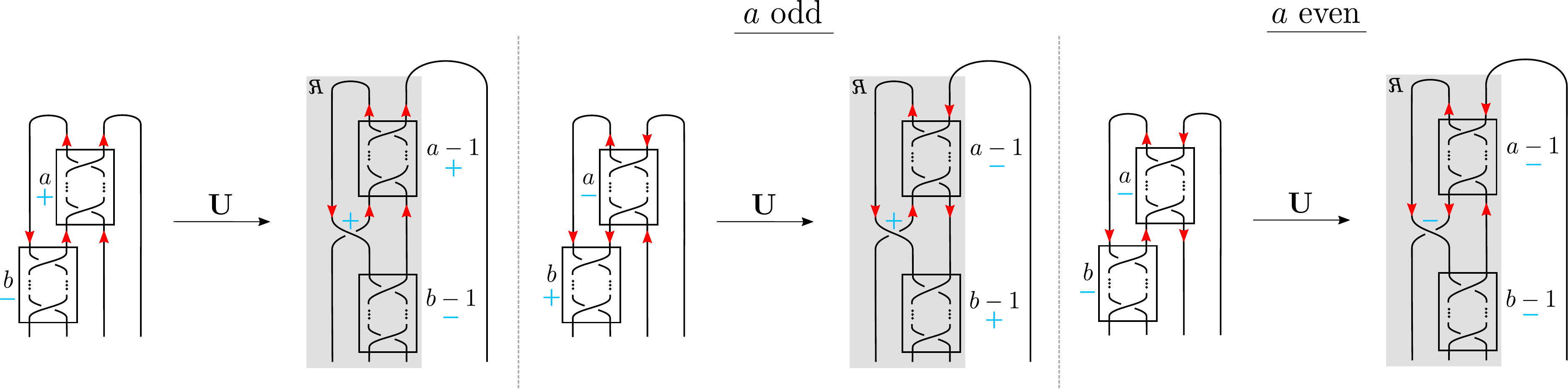}
    \caption{Transformation $\mathbf{U}$: all possible (local) orientations.}
    \label{fig:U_transformation_signs}
    \end{figure}

    \begin{figure}[h!]
    \centering
    \includegraphics[scale=0.25]{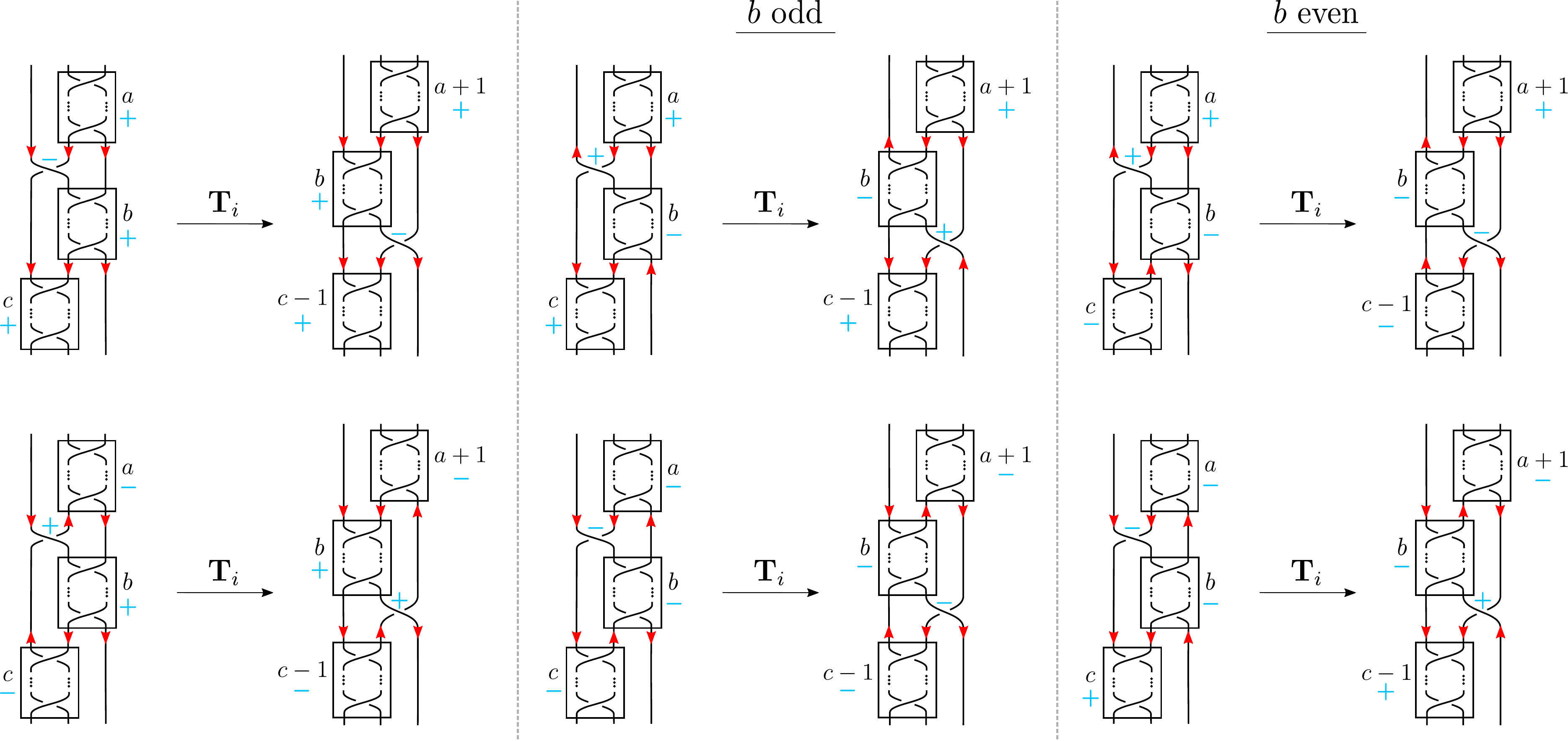}
    \caption{Transformation $\mathbf{T}_i$ ($i$ even): all possible (local) orientations.}
    \label{fig:T_1_transformation_signs}
    \end{figure}

\begin{remark}
    In the setting of Proposition \ref{prop:w(D_R)_w(D_R')}, the diagram $D'$ is as shown in Figure~\ref{fig:D_R_prime_2}, and is therefore reduced and alternating. Hence $D'$ is also $A$-adequate and $|s_AD'| = m+1$ (Figure~\ref{fig:s_AD_R_prime}).
\end{remark}

\begin{figure}[h!]
\centering
\begin{subfigure}{.2\textwidth}
  \centering
  \includegraphics[width=.4\linewidth]{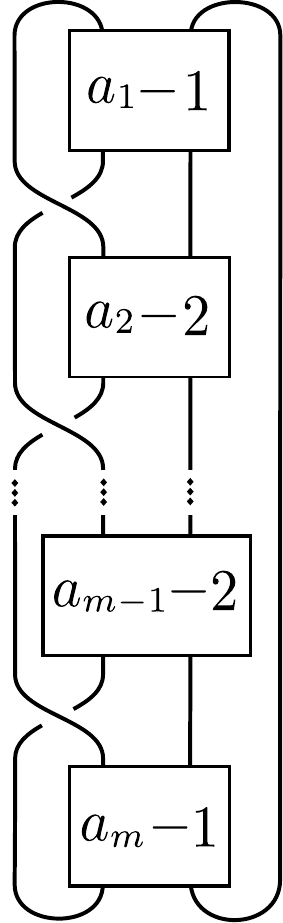}
  \caption{$D'$}\label{fig:D_R_prime_2}
\end{subfigure}%
{\hspace{0.5cm}}
\begin{subfigure}{.19\textwidth}
  \centering
   \includegraphics[width=0.4\linewidth]{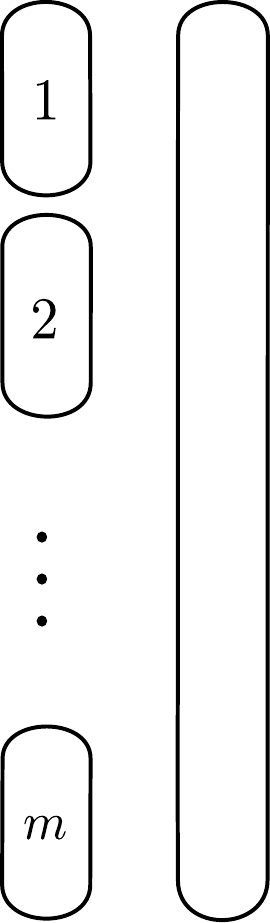}
  \caption{$s_AD'$}\label{fig:s_AD_R_prime}
\end{subfigure}
\caption{Reduced alternating diagram (a) obtained when applying Proposition \ref{prop:w(D_R)_w(D_R')} and the associated $s_A$-state (b).}\label{fig:alternating_rational_diagram}
\end{figure}


\section{The $\Lcorner_{4,3}$-shape of the Khovanov homology of closed positive 3-braids}\label{sec:L43-shape}

It is common to represent non-trivial Khovanov homology groups of a link in a table whose columns are indexed by the homological degree $i$, and its rows are indexed by the quantum degree $j$ (which \textit{jumps} by $2$). Hence, when we refer to the first column (resp. lowest row) of the Khovanov homology of a given link $L$, we mean the smallest homological degree $\underline{i}$ (resp. smallest quantum degree $\underline{j}$) where $H^{i,j}(L)$ is non-trivial. The first $a$ columns (resp. lowest $b$ rows)  thus correspond to homological degrees $i=\underline{i},\underline{i}+1,\dots, \underline{i}+a-1$ (resp. quantum degrees $j=\underline{j}, \underline{j}+2, \dots, \underline{j}+2(b-1)$); we shall refer to all of them collectively as the $\Lcorner_{a,b}$-\emph{shape} of the Khovanov homology. Note that if $D$ is a positive diagram of a link $L$, then $\underline{i}(L)=0$ \cite[Prop. 6.1]{Khovanov_2003} and, since $D$ is $A$-adequate, $\underline{j}(L) = j_{\min}(D) = c(D) - |s_AD|$.

The main purpose of this section is to prove Theorem \ref{th:main}, which establishes a criterion to determine the $\Lcorner_{4,3}$-shape of the Khovanov homology of every closed positive 3-braid. 

\begin{mythm}{1.1}
The Khovanov homology of a closed positive $3$-braid $\beta$ is as one of the Tables~\ref{tab:H(1)}--\ref{table:Kh_case_C4}. Moreover, $\beta$ is conjugate to a braid belonging to either $N=\{ 1, \sigma_1, \sigma_1^2, \sigma_1\sigma_2, \sigma_1^2 \sigma_2^2, \Delta \}$ (Tables~\ref{tab:H(1)}--\ref{tab:H(aba)}, respectively) or to some of the families $\mathbf{C1}$, $\mathbf{C2}$, $\mathbf{C3}$, $\mathbf{C4}$ (Tables~\ref{table:Kh_case_C1}--\ref{table:Kh_case_C4}, respectively), where
\[ \arraycolsep=1.4pt\def\arraystretch{1.8}
\begin{array}{rclcrcl} 
\mathbf{C1} & = & \{ \sigma_1^{k_1} \; | \; k_1 \geq 3 \}, & \text{\hspace{2cm}} & \mathbf{C4a} & = & \{ \beta \in \mathbb{B}_3 \; | \;  \inf(\beta) > 0 \} \setminus \{ \Delta \}, \\
\mathbf{C2} & = & \{ \sigma_1^{k_1}\sigma_2^2 \; | \; k_1 \geq 3 \}, & & \mathbf{C4b} & = & \{ \beta \in \mathbb{B}_3 \; | \; \inf_s(\beta) = 0 \text{ and } \operatorname{sl}(\beta) \geq 4 \}, \\
\mathbf{C3} & = & \{ \sigma_1^{k_1}\sigma_2^{k_2} \; | \; k_1,k_2 \geq 3 \}, & & \mathbf{C4} &  = & \mathbf{C4a} \cup \mathbf{C4b}.
\end{array}\]
\end{mythm}
\begin{proof}[Proof of Theorem~\ref{th:main}:]
Recall that the closures of conjugate braids are equivalent links. By Proposition~\ref{prop:representatives_conjugacy_classes_5_families}, any positive $3$-braid is conjugate to a braid $\beta$ in one of the families $\Lambda_i$ for some $i=1, \ldots, 5$. If $\inf(\beta) >0$, then either $\beta = \Delta$ or it belongs to $\mathbf{C4a}$. Otherwise, Remark \ref{remark:summit_infimum_Lambda_i}
implies that $\inf_s(\beta)=0$ and it follows from the description of families $\Lambda_i$ that its syllable length is either $0$, $1$ or $2t$, with $t>0$. We analyze these cases: 
\begin{itemize}[topsep=0pt, itemsep=-2pt]
\item If $\operatorname{sl}(\beta)=0$ then $\beta$ is trivial. \newline
\item If $\operatorname{sl}(\beta) = 1$ then $\beta = \sigma_1^k$, with $k>0$, and therefore either $\beta$ is $\sigma_1$ or $\sigma_1^2$ or it belongs to $\mathbf{C1}$. \newline
\item If $\operatorname{sl}(\beta)=2$ then $\beta = \sigma_1^{k_1} \sigma_1^{k_2}$, and therefore either $\beta = \sigma_1\sigma_2$ or $\beta = \sigma_1^2\sigma_2^2$ or $\beta$ belongs to either $\mathbf{C2}$ or $\mathbf{C3}$.  \newline
\item If $\operatorname{sl}(\beta)=2t$ with $t>1$, then $\beta \in \Lambda_4$ and it belongs to $\mathbf{C4b}$.
\end{itemize}

Khovanov homology of braids in $N$ is shown in Tables~\ref{tab:H(1)}--\ref{tab:H(aba)}. Then, to complete the proof it suffices to show that, for each $r=\mathbf{1}, \mathbf{2}, \mathbf{3}, \mathbf{4}$, the closure of every braid in $\mathbf{C}r$ has the expected $\Lcorner_{4,3}$-shape in its Khovanov homology. The rest of this section is devoted to prove this fact.
\end{proof}

\subsection{Strategy for the proof of Theorem \ref{th:main}}\label{subsec:strategy_proof_theorem_5.2}
We will address the proof of Theorem \ref{th:main} in two parts. First, in Section \ref{subsection:inf_s>0} we deal with the case $\mathbf{C4a}$: the braids with strictly positive infimum. In Section \ref{subsection:inf_s=0} we deal with the case of $3$-braids with infimum equal to $0$, and it comprises the cases $\mathbf{C1}$,  $\mathbf{C2}$, $\mathbf{C3}$, and $\mathbf{C4b}$.

In all cases we mostly adhere to a common strategy, that we outline here\footnote{Possible exceptions will be specifically handled when needed.}. The proof is by induction on the length of $\beta$. After checking the base cases, we assume that the closures of all braids of length $l' < l(\beta)$ in the same family $\mathbf{C}r$ have the same $\Lcorner_{4,3}$-shape in their Khovanov homology.

\begin{remark}\label{rem:Delta_gamma}
Observe that by Proposition~\ref{prop:representatives_conjugacy_classes_5_families} the braid $\beta$ is conjugate to a braid in one of the families $\Lambda_k$, for $1 \leq k \leq 5$; it is important to know that both braids belong to the same family $\mathbf{C}r$. Moreover, they have the same length, since both braids are positive and conjugate.
\end{remark}

By the above remark, we assume that $\beta$ is in one of the families $\Lambda_k$, for $1\leq k \leq 5$. Let $w$ be the word representing the normal form of $\beta$, with the convention that $\Delta = \sigma_1\sigma_2\sigma_1$ (i.e., $w$ equals the expression in the corresponding family $\Lambda_k$, with the aforementioned convention). We decompose $w$ as $w = (\sigma_1\sigma_2\sigma_1)^p \sigma_1 x$, with $p\geq 0$ as big as possible (if $\beta = \Delta^{q}$, we write $w=(\sigma_1\sigma_2\sigma_1)^{q-1} \sigma_1 \sigma_2 \sigma_1$).

We denote by $D$ the diagram associated to the word $w$ representing the link $\widehat{\beta}$. By performing an $A$ or $B$ smoothing at the first crossing $\sigma_1$ in the non-$\Delta$ part in $w$, we obtain two diagrams, $D_A$ and $D_B$, respectively (see Figure~\ref{fig:smoothings_standard_braid_diagram}).

 \begin{figure}[ht!]
    \centering
    \includegraphics[width=0.7\linewidth]{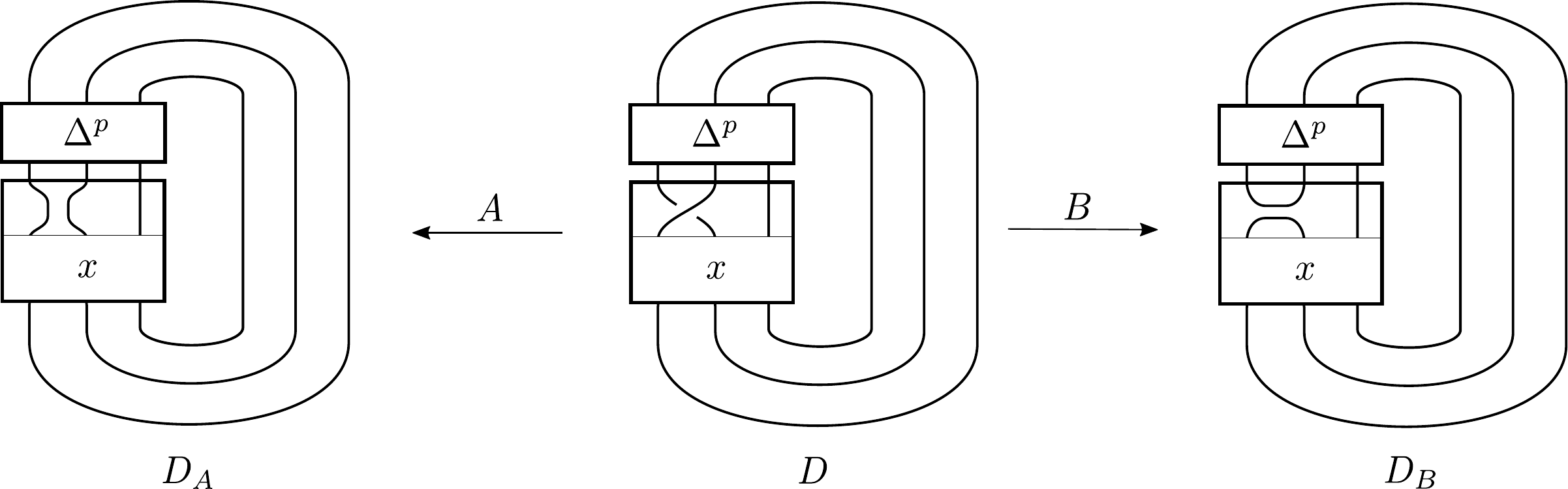}
    \caption{$A$ or $B$ smoothing at the first crossing $\sigma_1$ in the non-$\Delta$ part.}
    \label{fig:smoothings_standard_braid_diagram}
\end{figure}

Diagram $D_A$, whose orientation is inherited from that in $D$, is given as the closure of the word $w_A=\Delta^px$ of length $l(\beta)-1$ and we will prove that the associated braid $\beta_A$ belongs to the same family $\mathbf{C}r$ as $\beta$ for some $r\in \{\mathbf{1,2,3,4}\}$. 

Recall the long exact sequence in Khovanov homology relating the homology groups of $D$, $D_A$ and $D_B$: \begin{equation}\label{eq:shortexactsequencestrategy}    \begin{array}{rcrcl} \cdots & \longrightarrow & H^{\frac{w(D_B)-w(D)-1}{2}+i,\frac{3(w(D_B)-w(D))-1}{2}+j}(D_B) & 
         \longrightarrow & H^{i,j}(D) \longrightarrow H^{i,j-1}(D_A) \\
        & \longrightarrow & H^{\frac{w(D_B)-w(D)+1}{2}+i,\frac{3(w(D_B)-w(D))-1}{2}+j}(D_B) & \longrightarrow & \cdots
    \end{array}  
    \end{equation}

Diagrams $D$ and $D_A$ are positive (hence $A$-adequate) and therefore $\underline{i}(D)=\underline{i}(D_A)=0$ and $\underline{j}(D)=\underline{j}(D_A)+1$ by Proposition \ref{prop:j_min_j_max}. The key point is to prove that $H^{i,j}(D) \cong H^{i,j-1}(D_A)$ for $i\in I = \{0,1,2,3\}$ and $j\in J=\{\underline{j}(D), \, \underline{j}(D)+2, \, \underline{j}(D)+4\}$ (i.e., they have the same $\Lcorner_{4,3}$-shape). Since $D$ and $D_A$ represent the links $\widehat{\beta}$ and $\widehat{\beta_A}$, the result follows by induction hypotesis, since $l(\beta_A) = l(\beta) -1$. 

In order to prove that $H^{i,j}(D)$ and $H^{i,j}(D_A)$ have the same $\Lcorner_{4,3}$-shape it is enough to show that for $i\in I$ and $j\in J$ we have
\begin{equation}\label{inequalities_proof_Th_5.2}
\frac{w(D_B)-w(D) + 1}{2}+i < \underline{i}(D_B) \quad \quad \text{and} \quad \quad \frac{3(w(D_B)-w(D))-1}{2}+j < \underline{j}(D_B),
\end{equation}
since this implies that the corresponding homology groups of $D_B$ in the long exact sequence \eqref{eq:shortexactsequencestrategy} are trivial. Notice that in principle one should consider one additional inequality, similar to the first one in \eqref{inequalities_proof_Th_5.2} but with numerator equals to $w(D_B)-w(D)-1$; since this inequality is weaker than the displayed one, we can omit it. 

The tricky part is to analyze the values of those parameters involving $D_B$ in \eqref{inequalities_proof_Th_5.2}. To do so we transform $D_B$ into an equivalent alternating standard rational diagram $D_R'$, allowing us to rewrite the former inequalities in terms of parameters depending just on $D_R'$ (and not on $D_B$). We explain now how to obtain $D_R'$ from $D_B$ and how their writhes are related. 

We can perform an isotopy to $D_B$ which, roughly speaking, \emph{removes} the crossings corresponding to the $p$ initial factors $\Delta=\sigma_1\sigma_2\sigma_1$ in $w$, once at a time; write $D_B'$ for the obtained diagram (see Figure~\ref{fig:isotopies_Delta_DB_DB_prime}). We make two observations: \\
O1) The arrangement of the strands depends on the parity of $p$, as shown in the rightmost part of Figure~\ref{fig:isotopies_Delta_DB_DB_prime}. \newline
O2) Regardless of the orientation of $D_B$, each factor $\sigma_1\sigma_2\sigma_1$ in $w$ contributes with two negative and one positive crossings to $D_B$, as shown in the leftmost part of Figure \ref{fig:isotopies_Delta_DB_DB_prime}.

Next, if the first syllable of $x$ is $\sigma_1^k$, we perform $k$ Reidemeister I moves to the associated $k$ removable negative twists in $D_B'$. We write $D_R$ for the resulting diagram, which is of the form shown in Figures \ref{fig:racional_impar} and \ref{fig:racional_par}, and therefore corresponds to a standard rational diagram. This fact together with O2 leads to the following relation:  
\begin{equation}\label{eq:w(D_B)_w(D_R)} w(D_B) = w(D_R) - p - k.  \end{equation}

\begin{figure}[ht!]
    \centering
    \includegraphics[width=0.99\linewidth]{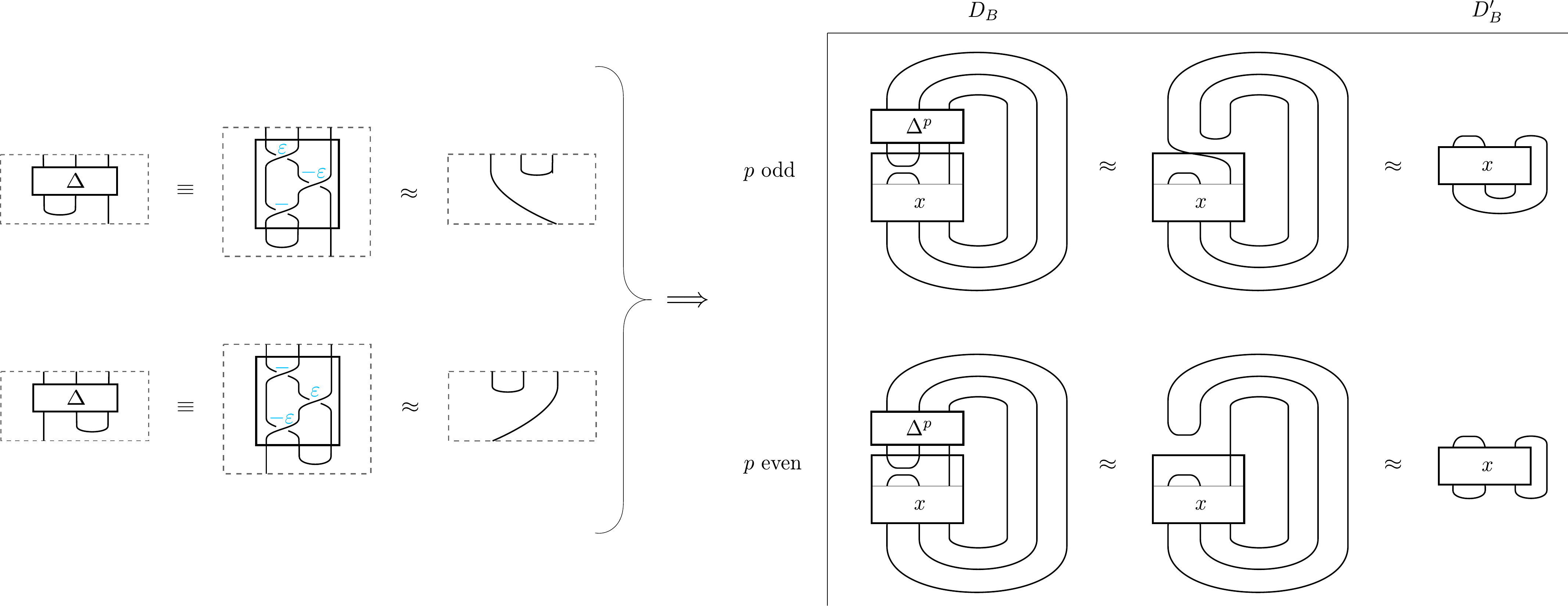}
    \caption{Isotopy from $D_B$ to $D_B'$.}
    \label{fig:isotopies_Delta_DB_DB_prime}
\end{figure}

Now, by Proposition~\ref{prop:w(D_R)_w(D_R')} we can transform $D_R$ into a reduced alternating diagram $D_R'$ as the one shown in Figure~\ref{fig:D_R_prime_2}. For our purposes it is crucial not just that $D_R'$ is reduced and alternating (and therefore adequate and $H$-thin \cite[Th. 3.12]{Lee_2005}) but also the fact that by combining equation \eqref{eq:w(D_B)_w(D_R)} and Proposition~\ref{prop:w(D_R)_w(D_R')} we know the precise relation between $w(D_B)$ and $w(D_R')$.

Since $D_B$ and $D_R'$ are equivalent diagrams, it is straightforward that $\underline{j}(D_R')= \underline{j}(D_B)$ and $\underline{i}(D_R')= \underline{i}(D_B)$. The advantage of using diagram $D_R'$ to analyze these parameters relies on the fact that $D_R'$ represents an $H$-thin link (and therefore its Khovanov homology is supported on two adjacent diagonals), and since it is adequate, $\underline{j}(D_R') = j_{\min}(D_R')$ and $H^{*, \underline{j}}(D_R') = H^{-n(D_R'), \underline{j}}(D_R') = \mathbb{Z}$, and therefore \begin{equation}\label{eq:real_i_min}\underline{i}(D_R') = -n(D_R').\end{equation}

Hence, the proof of Theorem \ref{th:main} boils down to proving \eqref{inequalities_proof_Th_5.2}, the base case of the induction for each family $\mathbf{C}r$, and the fact that 
the braids $\beta$ and $\beta_A$ belongs to the same family $\mathbf{C}r$, for $r=\mathbf{1,2,3,4}$; we will see that under certain circumstances, it might happen that $\beta \in \mathbf{C4b}$ while $\beta_A\in \mathbf{C4a}$, but since the $\Lcorner_{4,3}$-shape of the homology of these families agree, the inductive argument will work. The rest of this section is devoted to the proof of these three assertions for each of the families $\mathbf{C}r$.

\subsection{3-braids with infimum greater than 0}\label{subsection:inf_s>0} In this section we focus on braids with infimum greater or equal to one. Recall that in  Proposition~\ref{prop:representatives_conjugacy_classes_5_families} the conjugations transforming $\beta$ into a braid in some family $\Lambda_k$ do not decrease the infimum. Therefore, in this case Proposition~\ref{prop:representatives_conjugacy_classes_5_families}  and Remark~\ref{rem:Delta_gamma} can be rewritten by replacing each $\Lambda_k$ occurrence by the corresponding family $\Lambda_k^+$, where: 
$$
\begin{array}{l}
\Lambda_1^{+}= \left\{\Delta^p \ | \ p \geq 1 \right\} \subset \Lambda_1 , \\
\Lambda_2^{+}= \left\{\Delta^p\sigma_1^{k_1} \ | \  p \geq 1,\ k_1>0\right\} \subset \Lambda_2, \\
\Lambda_3^{+}= \left\{\Delta^{2u}\sigma_1\sigma_2 \ | \ u \geq 1 \right\} \subset \Lambda_3, \\
\Lambda_4^{+}= \left\{\Delta^{2u}\sigma_1^{k_1}\sigma_2^{k_2}\cdots \sigma_2^{k_{2t}} \ | \ u \geq 1,\ t>0,\ k_1,\ldots,k_{2t}\geq 2\right\} \subset \Lambda_4, \\
\Lambda_5^{+}= \left\{\Delta^{2u+1}\sigma_1^{k_1}\sigma_2^{k_2}\cdots \sigma_1^{k_{2t+1}} \ | \ u \geq 0,\ t>0,\ k_1,\ldots,k_{2t+1}\geq 2\right\} \subset \Lambda_5. 
\end{array}
$$

\begin{proof}[Proof of the case $\mathbf{C4a}$ of Theorem \ref{th:main}]
     Let $\beta$ denote a 3-braid with positive infimum, $\beta \neq \Delta$. Following the general strategy depicted in Section \ref{subsec:strategy_proof_theorem_5.2}, we will use induction on $l=l(\beta)$; note that $l \geq 3$. As $\Delta$ is excluded from case $\mathbf{C4a}$, the base case would be $l=4$, which corresponds to the braid $\beta = \Delta \sigma_1$ (up to conjugation). The link $\widehat{\Delta \sigma_1}$ satisfies the statement, as shown in Table \ref{tab:abaa}. 
     
     Let $l \geq 5$ and suppose that for every $\beta' \in \mathbb{B}_3^+$ with positive infimum, $\beta' \neq \Delta$ and $l(\beta') \leq l-1$ the Khovanov homology of $\widehat{\beta'}$ has the expected $\Lcorner_{4,3}$-shape. Let $\beta \in \mathbf{C4a}$ with $l(\beta) = l$. According to Remark~\ref{rem:Delta_gamma} we can assume without loss of generality that $\beta \in \Lambda_k^{+}$ for some $k$. We will address five subcases, one for each of the families $\Lambda_k^{+}$ above. Observe that for each subcase the corresponding braid $\beta_A$ leading to $D_A$ belongs to the family $\mathbf{C4a}$, and therefore the induction hypothesis applies. We start analyzing the case $\Lambda_5^{+}$. 

\begin{table}[!h]
\centering
\begin{tiny}
\begin{tblr}{|c||c|c|c|c|}
\hline
\backslashbox{\!$j$\!}{\!$i$\!} & $0$ & $1$ & $2$ & $3$ \\
\hline
\hline
$9$  &   &   &   & $ \Rone $ \\
\hline
$7$  &   &   &   & $ \Tone{2} $ \\
\hline
$5$  &   &   & $ \Rone $ &   \\
\hline
$3$  & $ \Rone $ &   &   &   \\
\hline
$1$  & $ \Rone $ &   &   &   \\
\hline
\end{tblr} \end{tiny}
\caption{$H(\widehat{\Delta \sigma_1})$.}\label{tab:abaa}
\end{table}

    \textbf{Subcase} $\mathbf{C4a.5}$ ($w \in \Lambda_5^{+}$). 
    Let $D$ be the diagram associated to a word $w \in \Lambda_5^+$, and consider diagrams $D_A$ and $D_B$ obtained as explained in Section~\ref{subsec:strategy_proof_theorem_5.2} (see Figure~\ref{fig:C4a.5}). Diagram $D_B$ can be isotopied into an equivalent standard rational diagram $D_R$, and by applying Proposition \ref{prop:w(D_R)_w(D_R')} we obtain the reduced alternating diagram $D_R'$ shown in the rightmost part of Figure~\ref{fig:C4a.5}.

    \begin{figure}[h!]
    \centering
    \includegraphics[scale=0.24]{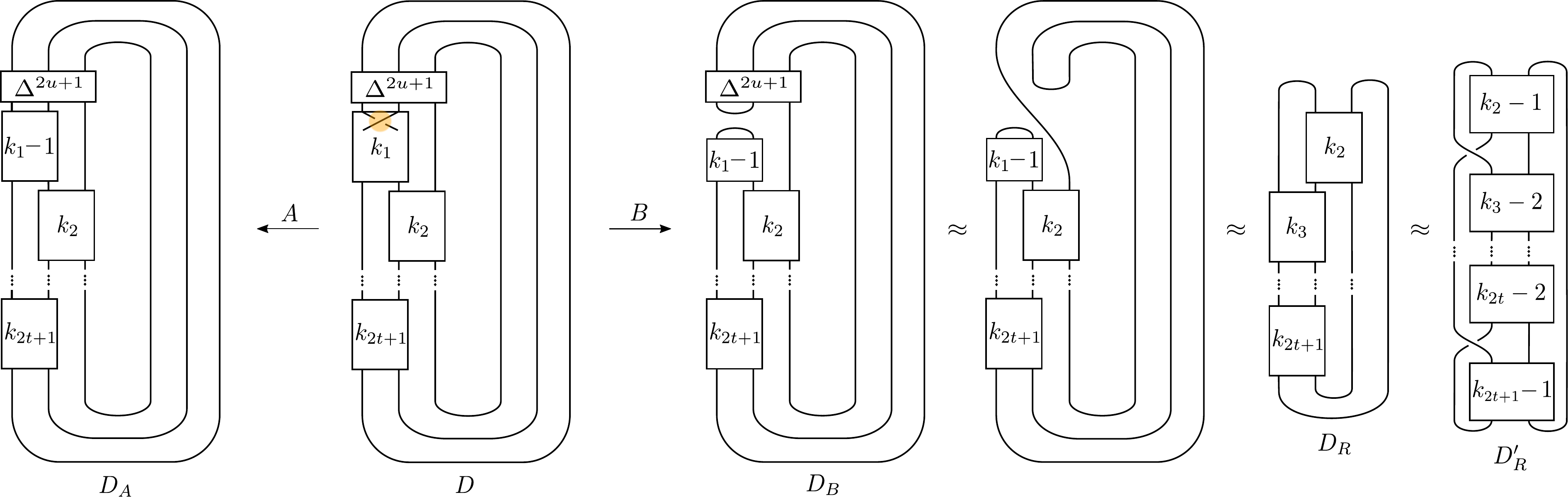}
    \caption{Diagrams illustrating the proof of subcase $\mathbf{C4a.5}$. The smoothed crossing is highlighted in yellow.}
    \label{fig:C4a.5}
    \end{figure}

The orientation of $D_B$ is not fixed, but by (\ref{eq:w(D_B)_w(D_R)}) we get that $$w(D_B)=w(D_R)-(2u+1)-(k_1-1) = w(D_R)-2u-k_1,$$ regardless of its orientation. As there are $2t$ boxes in $D_R$, from Proposition~\ref{prop:w(D_R)_w(D_R')} we have that $w(D_R)=w(D_R')-(2t-1)$, so $$w(D_B)=w(D_R')-2u-2t-k_1+1.$$ 

We analyze now extreme values for homological and quantum degrees of the Khovanov homology of $D_R'$ (which coincide with those of $D_B$, since both diagrams are equivalent). Recall that, since $D_R'$ is reduced and alternating (hence adequate) it holds that $\underline{i}(D_R')=-n(D_R')$. Moreover, since $|s_AD_R'|=2t+1$ (see Figure~\ref{fig:s_AD_R_prime}) applying Proposition~\ref{prop:j_min_j_max} leads to 
$$\underline{j}(D_R') = c(D_R')-3n(D_R')- (2t+1),$$
where $c(D_R')=\kappa - k_1 - (2t-1)$ and $\kappa =  k_1 + \cdots + k_{2t+1}$.

With regard to the diagram $D$, note that $$c(D)=6u+3+\kappa = w(D), \quad n(D)=0 \quad \mbox{ and } \quad |s_AD|=3,$$ and since it is positive (and therefore $A$-adequate) we get $$\underline{j}(D)= 6u+3+\kappa-3=6u+\kappa.$$ 

Recall that our goal is to prove inequalities \eqref{inequalities_proof_Th_5.2}, which we rewrite using the data computed above: 
        \begin{equation}\label{inequality_i_C4a.5}
   \begin{array}{crcl}
    &  \frac{1}{2}( w(D_B)-w(D)+1) + i & < & \underline{i}(D_B)=\underline{i}(D_R') \\
    \Leftrightarrow &   w(D_R') -8u -2t - k_1 - \kappa -1 +2i & < & -2n(D_R') \\
   \Leftrightarrow &  c(D_R') -8u -2t - k_1 - \kappa -1 +2i & < & 0 \\
   \Leftrightarrow &  2i & < & 8u + 4t + 2k_1 \\
    \end{array} 
    \end{equation}
    and
\begin{equation}\label{inequality_j_C4a.5}
    \begin{array}{crcl}
     & \frac{1}{2}\left[ 3 (w(D_B)-w(D))-1 \right] + j  & < & \underline{j}(D_B)= \underline{j}(D_R') \\
   \Leftrightarrow  &   3\left[(w(D_R')-2u-2t-k_1+1)-(6u+3+\kappa)\right]-1 + 2j & < & 2c(D_R') - 6n(D_R') - 2(2t+1) \\
    \Leftrightarrow &  3(p(D_R')-n(D_R'))  - 24u - 6t -3k_1 - 3\kappa -7 +2j  & < & 2p(D_R') - 4n(D_R') -4t -2   \\
    \Leftrightarrow & c(D_R') - 24u - 6t  -3k_1 - 3\kappa -7 +2j& < & -4t-2 \\
    \Leftrightarrow & 2j-4 & < & 24u +4t  + 4k_1 + 2\kappa. 
    \end{array} 
    \end{equation}

As we discussed while outlining the general strategy, we need to prove the above inequalities for $i=0,1,2,3$ and $j = 6u+\kappa, 6u+\kappa+2, 6u+\kappa+4$, and therefore it is enough to prove them for the maximum values. For $i=3$, the last inequality in (\ref{inequality_i_C4a.5}) holds, since $ 8u + 4t + 2k_1 \geq 8 \cdot 0 + 4 \cdot 1+ 2 \cdot 2 = 8$. For $j = 6u + \kappa + 4$, the last inequality in (\ref{inequality_j_C4a.5}) is equivalent to $4 <  12u + 4t +4k_1$, which is true since $12u + 4t +4k_1 \geq  12\cdot 0 + 4\cdot 1  + 4 \cdot 2 = 12 $. 
   
\vspace{0.17cm}

\textbf{Subcase} $\mathbf{C4a.4}$  ($w \in \Lambda_4^{+}$).  
The corresponding diagrams $D$, $D_A$ and $D_B$ are depicted in Figure~\ref{fig:C4a.4}. As in the previous case, the goal is to prove inequalities \eqref{inequalities_proof_Th_5.2} for $i\in \{0,1,2,3\}$ and $j\in \{\underline{j}(D), \, \underline{j}(D)+2, \, \underline{j}(D)+4\}$.

\begin{figure}[h!]
\centering
\includegraphics[scale=0.24]{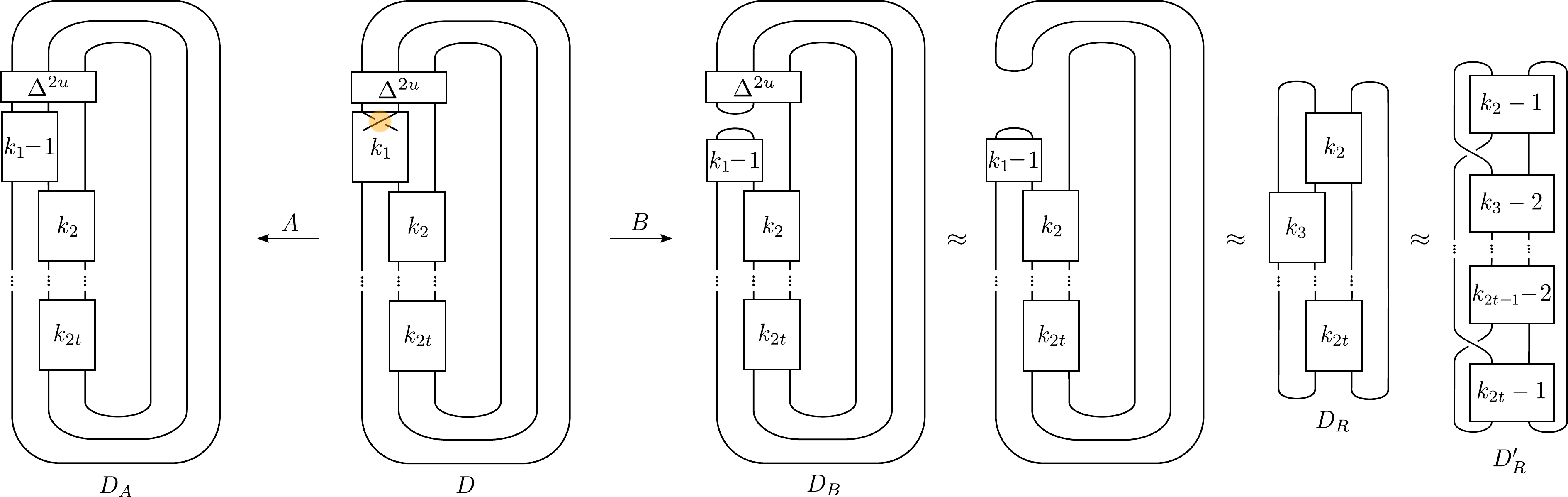}
\caption{Diagrams illustrating the proof of subcase $\mathbf{C4a.4}$.}
\label{fig:C4a.4}
\end{figure}

We list all the necessary data (which can be computed in the same way as we did in subcase $\mathbf{C4a.5}$): 
    \[ \begin{array}{rcl}
        w(D_B)  & = & w(D_R') - 2u - 2t - k_1 + 3, \\
        \underline{j}(D_R') & = & c(D_R') - 3n(D_R') - 2t, \\
        c(D_R')  & = & \kappa - 2t - k_1 + 2, \text{ with } \kappa = k_1 + \cdots + k_{2t}, \\
        \underline{i}(D_R') & = & -n(D_R'), \\
         w(D) & = & c(D) \; \; = \; \; 6u + \kappa, \\
        \underline{j}(D) & = & 6u + \kappa -3.
    \end{array} \]

As in the analysis of the previous case, we transform \eqref{inequalities_proof_Th_5.2} into equivalent inequalities using the above data: 
 \begin{equation}\label{inequality_i_C4a.4}
        \begin{array}{crcl}
         & \frac{1}{2}\left( w(D_B)-w(D)+1 \right) + i  & < & \underline{i}(D_R') \\
       \Leftrightarrow  &   w(D_R') - 8u - 2t - k_1 - \kappa + 4 + 2i & < & -2n(D_R') \\
        \Leftrightarrow &  c(D_R') - 8u - 2t - k_1 - \kappa + 4 + 2i & < & 0  \\
        \Leftrightarrow & 6 + 2i & < & 8u + 4t + 2k_1
        \end{array} 
    \end{equation} 
    and 
    \begin{equation}\label{inequality_j_C4a.4}
        \begin{array}{crcl}
         & \frac{1}{2}\left[ 3 (w(D_B)-w(D))-1 \right] + j  & < & \underline{j}(D_R') \\
       \Leftrightarrow  &   3(w(D_R')-2u-2t-k_1+3-6u- \kappa)-1+ 2j & < & 2c(D_R') - 6n(D_R') - 4t \\
        \Leftrightarrow &  3(p(D_R')-n(D_R'))  - 24u - 6t -3k_1 - 3\kappa +8 +2j  & < & 2p(D_R') - 4n(D_R') -4t    \\
        \Leftrightarrow & c(D_R') - 24u - 6t  -3k_1 - 3\kappa + 8 + 2j& < & -4t \\
        \Leftrightarrow & 10 + 2j & < & 24u +4t  + 4k_1 + 2\kappa. 
        \end{array} 
    \end{equation}
    
Again, it is enough to prove inequalities above for the maximum values $i=3$ and $j=6u+\kappa+1$. For $i=3$, the last inequality in (\ref{inequality_i_C4a.4}) holds, since $8u + 4t + 2k_1 \geq  8 \cdot 1 + 4 \cdot 1 + 2 \cdot 2 = 16$. For $j = 6u + \kappa + 1$, the last inequality in (\ref{inequality_j_C4a.4}) is equivalent to $12 < 12u + 4t +4k_1$, which is true since $12u + 4t +4k_1 \geq 12\cdot 1 + 4\cdot 1 + 4 \cdot 2 = 24 $. 

\vspace{0.17cm}
   
 \textbf{Subcase} $\mathbf{C4a.3}$ ($w \in \Lambda_3^{+}$). The corresponding diagrams $D$, $D_A$ and $D_B$ obtained as explained in Section \ref{subsec:strategy_proof_theorem_5.2} are shown in Figure~\ref{fig:C4a.3}. In this case $D_B$ is a (non-trivial) diagram of the trivial knot. In fact, the associated diagram $D_R$ contains a single (positive) crossing. Therefore, the only non-trivial groups in the homology of $D_B$ are $H^{0,1}(D_B)=H^{0,-1}(D_B) = \mathbb{Z}$. The strategy outlined in the previous cases can be simplified, and using the fact that $w(D_B)=1-2u$ when analyzing \eqref{eq:shortexactsequencestrategy} the statement holds. 

\begin{figure}[h!]
\centering
\includegraphics[scale=0.24]{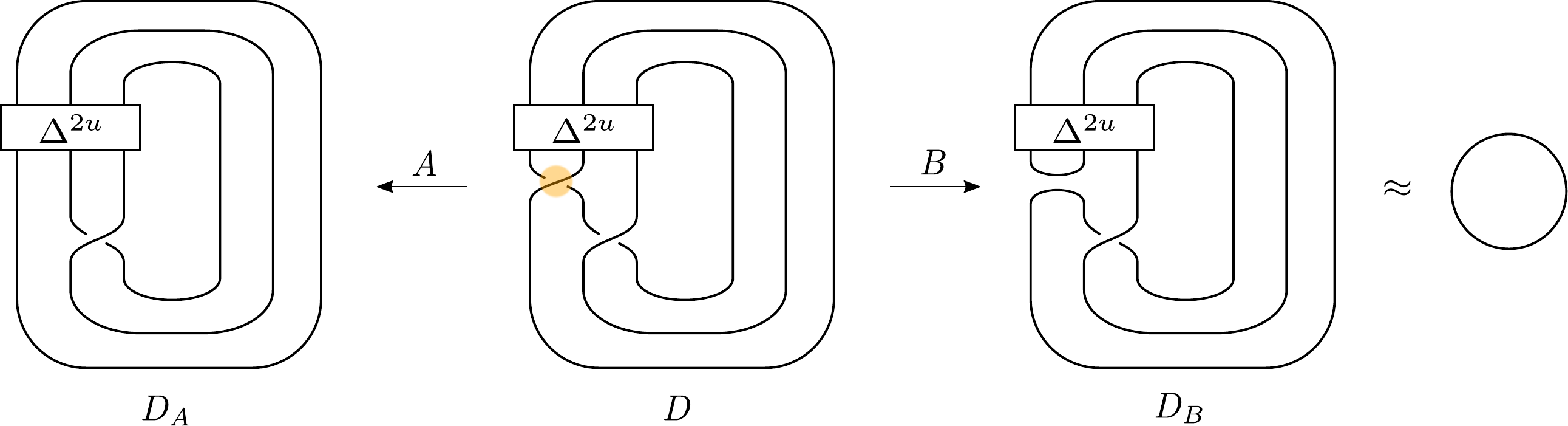}
\caption{Diagrams illustrating the proof of subcase $\mathbf{C4a.3}$.}
\label{fig:C4a.3}
\end{figure}

Notice that our results agree with \cite[Cor. 5.7]{Chandler_Lowrance_Sazdanovic_Summers_2022} when restricting to $u>0$. 

 \vspace{0.17cm}
 
\textbf{Subcase} $\mathbf{C4a.2}$  ($w \in \Lambda_2^{+}$).
Diagrams $D$, $D_A$ and $D_B$ obtained when following the strategy in Section~\ref{subsec:strategy_proof_theorem_5.2} are shown in Figure~\ref{fig:C4a.2}. 

\begin{figure}[h!]
\centering
\includegraphics[scale=0.24]{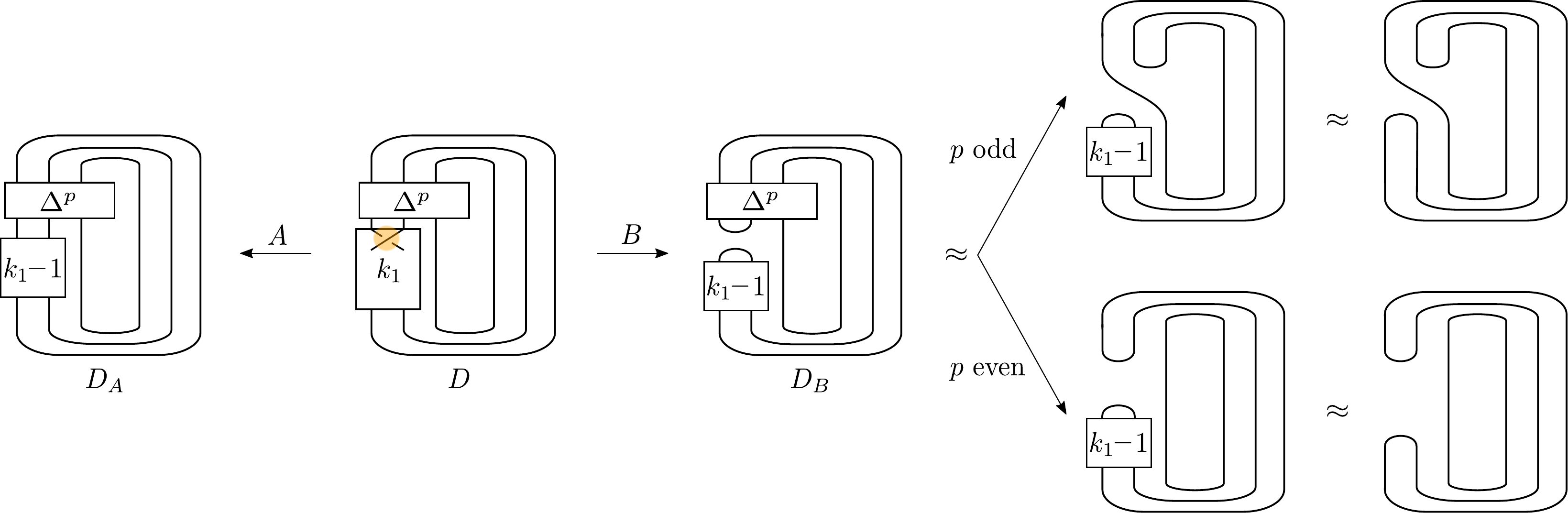}
\caption{Diagrams illustrating the proof of subcase $\mathbf{C4a.2}$.}
\label{fig:C4a.2}
\end{figure}

Observe that $D_B$ is a (non-trivial) diagram of the unknot if $p$ is odd, and a (non-standard) diagram of the $2$-component trivial link if $p$ is even, and therefore we get:  if $p$ is odd, the non-trivial homology groups are $H^{0,-1}(D_B)=H^{0,1}(D_B)=\mathbb{Z}$; if $p$ is even, the non-trivial homology groups are $H^{0,-2}(D_B)=H^{0,2}(D_B)=\mathbb{Z}$ and $H^{0,0}(D_B)=\mathbb{Z}^2$.

 We list the needed data (which can be computed directly from the diagrams in Figure~\ref{fig:C4a.2}):
 $$w(D_B)=1-p-k_1, \; \quad \underline{i}(D_B) = 0, \; \quad  \underline{j}(D_B)=-\frac{1}{2}(3+(-1)^p),   \;  \quad  w(D)=3p+k_1,  \;  \quad \underline{j}(D) = 3p+k_1-3.$$

After substituting the above parameters in \eqref{inequalities_proof_Th_5.2} the inequalities are rewritten as 
\begin{equation}\label{inequality_i_j_C4a.2}
     2 + 2i < 4p + 2k_1 \quad \text{ and } \quad 5 + (-1)^p + 2j < 12p + 6k_1.
 \end{equation}

We prove them for the maximal values $j=3p+k_1+1$ and $i=3$. For $i=3$, the leftmost inequality in (\ref{inequality_i_j_C4a.2}) holds if $p \geq 2$ or $k_1 \geq 3$. For $j = 3p+k_1+1$, the rightmost inequality in (\ref{inequality_i_j_C4a.2}) holds, since $p \geq 1$ and $k_1 \geq 1$. The pathological situations correspond to the base case $(p,k_1)=(1,1)$ and the case $(p,k_1)=(1,2)$. In the latter, $w = \Delta \sigma_1^2$ and $w_A = \Delta \sigma_1$, and both words give rise to two braids whose closure have the same $\Lcorner_{4,3}$-shape in their Khovanov homology, as shown in Tables \ref{tab:abaa} and \ref{tab:abaaa}, which agrees with that of Table~\ref{table:Kh_case_C4}.

 \begin{table}[!h]
\centering
\begin{tiny}
    \begin{tblr}{|c||c|c|c|c|c|}
\hline
\backslashbox{\!$j$\!}{\!$i$\!} & $0$ & $1$ & $2$ & $3$ & $4$ \\
\hline
\hline
$12$  &   &   &   &   & $ \Rone $ \\
\hline
$10$  &   &   &   & $ \Rone $ & $ \Rone $ \\
\hline
$8$  &   &   &   & $ \Tone{2} $ &   \\
\hline
$6$  &   &   & $ \Rone $ &   &   \\
\hline
$4$  & $ \Rone $ &   &   &   &   \\
\hline
$2$  & $ \Rone $ &   &   &   &   \\
\hline
\end{tblr} \end{tiny}
\caption{$H(\widehat{\Delta \sigma_1^2})$.}\label{tab:abaaa}
\end{table}
  
   \vspace{0.17cm}
   
 \textbf{Subcase} $\mathbf{C4a.1}$  ($w \in \Lambda_1^{+}$). In {\cite[Cor. 5.7]{Chandler_Lowrance_Sazdanovic_Summers_2022}}, the Khovanov homology of 
links $\widehat{\Delta^p}$ is computed, and when $p>1$ one gets the expected $\Lcorner_{4,3}$-shape in the homology.
\end{proof}

In this section we have analyzed the Khovanov homology of the links associated to all $3$-braids with infimum greater than zero but one case, that of the braid $\beta = \Delta$, which is included in the set $N$ in the statement of Theorem \ref{th:main}.

\subsection{3-braids with infimum equal to 0}\label{subsection:inf_s=0} Now discuss positive 3-braids which do not belong to the conjugacy class of any braid of the form $\Delta \beta$ with $\beta \in \mathbb{B}_3^+$. This corresponds to braids whose conjugacy class contains a braid in the family $\{1,\sigma_1, \sigma_1^2, \sigma_1\sigma_2, \sigma_1^2\sigma_2^2\}$ (whose Khovanov homology can be easily computed) or in families $\mathbf{C1, C2, C3}$ and $\mathbf{C4b}$. In this section we prove Theorem~\ref{th:main} for these families.

\begin{proof}[Proof of Case  $\mathbf{C1}$ of Theorem \ref{th:main}] The closure of $\beta = {\sigma_1^{k_1}} \in \mathbb{B}_3$ consists of a disjoint union of the torus link $T(2,k_1)$ and an unknotted component, for each $k_1 \geq 3$. We could prove the result by induction on the length of $\beta$ by applying the strategy outlined in Section \ref{subsec:strategy_proof_theorem_5.2}. The result also holds from \cite[Prop. 26]{Khovanov_2000}, where the Khovanov homology of torus links $T(2,k_1)$ was computed, and the formulas to compute the Khovanov homology of a disjoint union of links \cite[Cor. 12]{Khovanov_2000}. 
\end{proof}

\begin{proof}[Proof of Cases  $\mathbf{C2}$ and $\mathbf{C3}$ of Theorem \ref{th:main}]
Let $\beta = \sigma_1^{k_1} \sigma_2^{k_2}$ with $k_1 \geq 3$ and $k_2 \geq 2$.  
If $k_2=2$ then $\beta$ belongs to $\mathbf{C2}$ and it belongs to $\mathbf{C3}$ otherwise. We follow the strategy in Section \ref{subsec:strategy_proof_theorem_5.2} and proceed by induction on the length of the braid $l(\beta) = k_1+k_2$. The base cases correspond to the braids $\sigma_1^3 \sigma_2^2$ and $\sigma_1^3 \sigma_2^3$, whose closures have the Khovanov homology presented in Tables \ref{tab:H(aaabb)} and \ref{tab:H(aaabbb)}. We can assume without loss of generality that $k_1 \geq k_2$ (otherwise, we conjugate the braid by $\Delta$, interchanging $\sigma_1$ and $\sigma_2$) and assume that the statement holds for those braids $\sigma_1^{k'_1} \sigma_2^{k'_2}$ of length $l'<k_1+k_2$. 
 
\begin{table}[!htb]
    \begin{minipage}[b]{.4\linewidth}
      \centering
          \begin{tiny}
          \begin{tblr}{|c||c|c|c|c|c|c|}
            \hline
            \backslashbox{\!$j$\!}{\!$i$\!} & $0$ & $1$ & $2$ & $3$ & $4$ & $5$ \\
            \hline
            \hline
            $14$  &   &   &   &   &   & $ \Rone $ \\
            \hline
            $12$  &   &   &   &   &   & $ \Tone{2} $ \\
            \hline
            $10$  &   &   &   & $ \Rone $ & $ \Rone $ &   \\
            \hline
            $8$  &   &   & $ \Rone $ & $ \Tone{2} $ &   &   \\
            \hline
            $6$  &   &   & $ \Rmor{2} $ &   &   &   \\
            \hline
            $4$  & $ \Rone $ &   &   &   &   &   \\
            \hline
            $2$  & $ \Rone $ &   &   &   &   &   \\
            \hline
           \end{tblr}
           \end{tiny}
    \caption{$H(\widehat{\sigma_1^3\sigma_2^2})$.}\label{tab:H(aaabb)}
    \end{minipage}
    \begin{minipage}[b]{.4\linewidth}
      \centering
        \begin{tiny}
        \begin{tblr}{|c||c|c|c|c|c|c|c|}
            \hline
            \backslashbox{\!$j$\!}{\!$i$\!} & $0$ & $1$ & $2$ & $3$ & $4$ & $5$ & $6$ \\
            \hline
            \hline
            $17$  &   &   &   &   &   &   & $ \Rone $ \\
            \hline
            $15$  &   &   &   &   &   & $ \Rone $ & $ \Tone{2} $ \\
            \hline
            $13$  &   &   &   &   &   & $ \Rone \oplus \Tone{2} $ &   \\
            \hline
            $11$  &   &   &   & $ \Rmor{2} $ & $ \Rone $ &   &   \\
            \hline
            $9$  &   &   &   & $ \Tmor{2}{2} $ &   &   &   \\
            \hline
            $7$  &   &   & $ \Rmor{2} $ &   &   &   &   \\
            \hline
            $5$  & $ \Rone $ &   &   &   &   &   &   \\
            \hline
            $3$  & $ \Rone $ &   &   &   &   &   &   \\
            \hline
        \end{tblr}
        \end{tiny}
    \caption{$H(\widehat{\sigma_1^3\sigma_2^3})$.}\label{tab:H(aaabbb)}
    \end{minipage}
    \end{table}

Following Section \ref{subsec:strategy_proof_theorem_5.2} we smooth the crossing corresponding to the first $\sigma_1$ occurrence in $w$ to obtain diagrams $D_A$ and $D_B$ shown in Figure~\ref{fig:Th1_C3}. Diagram $D_A$ corresponds to the standard diagram of the closure of the braid $\sigma_1^{k_1-1}\sigma_2^{k_2}$, which satisfies the statement by the induction hypothesis.

    \begin{figure}[h!]
    \centering
    \includegraphics[scale=0.24]{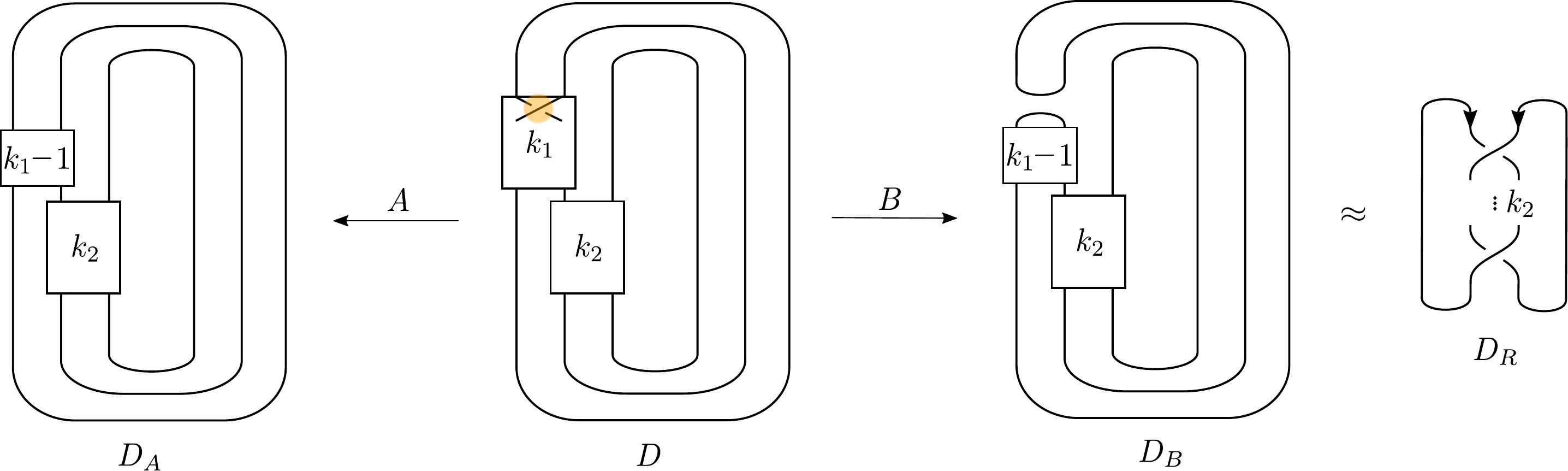}
    \caption{$A$ and $B$-smoothing of a crossing in $D$.}
    \label{fig:Th1_C3}
    \end{figure}
    
Observe that diagrams $D_B$ and $D_R$ are equivalent diagrams of the torus link $T(2,k_2)$, and we can orient them in such a way that $D_R$ becomes positive (and therefore $A$-adequate, as $D$), and we get
$$w(D_B)=k_2-k_1+1, \quad \underline{i}(D_B)=0, \quad \underline{j}(D_B) = j_{\min}(D_R)=k_2-2, \quad w(D)=k_1+k_2 \quad \underline{j}(D)= k_1+k_2-3.$$

We can rewrite inequalities in \eqref{inequalities_proof_Th_5.2} by using the above parameters. Moreover, we need to prove them for $i=0,1,2,3$ and $j=k_1+k_2-3,\, k_1+k_2-1, \, k_1+k_2+1$, so it is enough to prove them for maximum values of $i$ and $j$ :
    \begin{equation}\label{eq:inequality_i_C2_C3} \begin{array}{crcl}
     &  \frac{1}{2}(w(D_B)-w(D)+1) + i & < & \underline{i}(D_B) \\
    \Leftrightarrow  &  \frac{1}{2}( -2k_1+2) +3 & < & 0 \\
     \Leftrightarrow &   4 & < & k_1 \\
    \end{array} 
    \end{equation}
    and
    \begin{equation}\label{eq:inequality_j_C2_C3} \begin{array}{crcl}
     &  \frac{1}{2}[3(w(D_B)-w(D))-1] + j & < & \underline{j}(D_B) \\
    \Leftrightarrow  &  \frac{1}{2}( -6k_1+2) +k_1+k_2+1 & < & k_2-2 \\
     \Leftrightarrow &   4 & < & 2k_1 \\
    \end{array} 
    \end{equation}

Since $k_1\geq 4$, the inequality (\ref{eq:inequality_j_C2_C3}) holds. However, inequality  (\ref{eq:inequality_i_C2_C3}) holds when $k_1 \geq 5$. The pathological cases are $(k_1,k_2) = (4,2)$ corresponding to a braid in the family $\mathbf{C2}$, and $(k_1,k_2) \in \{ (4,3),(4,4) \}$ corresponding to braids in $\mathbf{C3}$. For these three cases, we computed their Khovanov homology and check that they have the desired $\Lcorner_{4,3}$-shape (i.e., that corresponding to the families they belong to), as shown in Tables \ref{tab:H(aaaabb)}--\ref{tab:H(aaaabbbb)}.
 \end{proof}
 
\begin{table}[!htb]
    \begin{minipage}[b]{.4\linewidth}
      \centering
          \begin{tiny}
          \begin{tblr}{|c||c|c|c|c|c|c|c|}
                \hline
                \backslashbox{\!$j$\!}{\!$i$\!} & $0$ & $1$ & $2$ & $3$ & $4$ & $5$ & $6$ \\
                \hline
                \hline
                $17$  &   &   &   &   &   &   & $ \Rone $ \\
                \hline
                $15$  &   &   &   &   &   & $ \Rone $ & $ \Rone $ \\
                \hline
                $13$  &   &   &   &   & $ \Rone $ & $ \Tone{2} $ &   \\
                \hline
                $11$  &   &   &   & $ \Rone $ & $ \Rmor{2} $ &   &   \\
                \hline
                $9$  &   &   & $ \Rone $ & $ \Tone{2} $ &   &   &   \\
                \hline
                $7$  &   &   & $ \Rmor{2} $ &   &   &   &   \\
                \hline
                $5$  & $ \Rone $ &   &   &   &   &   &   \\
                \hline
                $3$  & $ \Rone $ &   &   &   &   &   &   \\
                \hline
            \end{tblr}
            \end{tiny}
    \caption{$H(\widehat{\sigma_1^4\sigma_2^2})$.}\label{tab:H(aaaabb)}
    \end{minipage}
    \begin{minipage}[b]{.4\linewidth}
      \centering
          \begin{tiny}
          \begin{tblr}{|c||c|c|c|c|c|c|c|c|}
               \hline
                \backslashbox{\!$j$\!}{\!$i$\!} & $0$ & $1$ & $2$ & $3$ & $4$ & $5$ & $6$ & $7$ \\
                \hline
                \hline
                $20$  &   &   &   &   &   &   &   & $ \Rone $ \\
                \hline
                $18$  &   &   &   &   &   &   & $ \Rone $ & $ \Tone{2} $ \\
                \hline
                $16$  &   &   &   &   &   & $ \Rone $ & $ \Rone \oplus \Tone{2} $ &   \\
                \hline
                $14$  &   &   &   &   & $ \Rone $ & $ \Rone \oplus \Tone{2} $ &   &   \\
                \hline
                $12$  &   &   &   & $ \Rmor{2} $ & $ \Rmor{2} $ &   &   &   \\
                \hline
                $10$  &   &   &   & $ \Tmor{2}{2} $ &   &   &   &   \\
                \hline
                $8$  &   &   & $ \Rmor{2} $ &   &   &   &   &   \\
                \hline
                $6$  & $ \Rone $ &   &   &   &   &   &   &   \\
                \hline
                $4$  & $ \Rone $ &   &   &   &   &   &   &   \\
                \hline
           \end{tblr}
           \end{tiny}
    \caption{$H(\widehat{\sigma_1^4\sigma_2^3})$.}\label{tab:H(aaaabbb)}
    \end{minipage}
\end{table}

\begin{table}[!htb]
       \begin{tiny}
       \begin{tblr}{|c||c|c|c|c|c|c|c|c|c|}
            \hline
            \backslashbox{\!$j$\!}{\!$i$\!} & $0$ & $1$ & $2$ & $3$ & $4$ & $5$ & $6$ & $7$ & $8$ \\
            \hline
            \hline
            $23$  &   &   &   &   &   &   &   &   & $ \Rone $ \\
            \hline
            $21$  &   &   &   &   &   &   &   & $ \Rmor{2} $ & $ \Rone $ \\
            \hline
            $19$  &   &   &   &   &   &   & $ \Rone $ & $ \Tmor{2}{2} $ &   \\
            \hline
            $17$  &   &   &   &   &   & $ \Rone $ & $ \Rmor{2} \oplus \Tone{2} $ &   &   \\
            \hline
            $15$  &   &   &   &   & $ \Rmor{2} $ & $ \Rone \oplus \Tone{2} $ &   &   &   \\
            \hline
            $13$  &   &   &   & $ \Rmor{2} $ & $ \Rmor{3} $ &   &   &   &   \\
            \hline
            $11$  &   &   &   & $ \Tmor{2}{2} $ &   &   &   &   &   \\
            \hline
            $9$  &   &   & $ \Rmor{2} $ &   &   &   &   &   &   \\
            \hline
            $7$  & $ \Rone $ &   &   &   &   &   &   &   &   \\
            \hline
            $5$  & $ \Rone $ &   &   &   &   &   &   &   &   \\
            \hline
        \end{tblr} \end{tiny}
    \caption{$H(\widehat{\sigma_1^4\sigma_2^4})$.}\label{tab:H(aaaabbbb)}
\end{table}

\begin{proof}[Proof of Case  $\mathbf{C4b}$ of Theorem \ref{th:main}]
Let $\beta$ be a braid with $\inf_s(\beta)=0$ and $\operatorname{sl}(\beta) = 4$. Then we can consider the diagram $D$ associated to a word $w$ representing the link $\widehat{\beta}$. We can assume that $w = \sigma_1^{k_1} \sigma_2^{k_2} \cdots \sigma_1^{k_{2t-1}} \sigma_2^{k_{2t}}$, with $t\geq 2$ and $k_1 = \max\{k_1, \ldots, k_{2t}\}$ (recall that conjugating a braid by $\Delta$ permutes $\sigma_1$ and $\sigma_2$). Moreover we can assume that $k_i\geq 2$ for $i=1, \ldots, 2t$, since otherwise we could conjugate $\beta$ to obtain a braid with infimum greater than $0$, yielding a contradiction with the fact that $\mbox{inf}_s(\beta) = 0$.

Following the general strategy described in Section \ref{subsec:strategy_proof_theorem_5.2}, we proceed by induction on $l=l(\beta)$. Since $\operatorname{sl}(\beta)\geq 4$, the base case corresponds to $l=8$, and thus $\beta = \sigma_1^2 \sigma_2^2\sigma_1^2 \sigma_2^2$. The link $\widehat{\sigma_1^2 \sigma_2^2\sigma_1^2 \sigma_2^2}$ satisfies the statement, as shown in Table \ref{tab:H(aabbaabb)}. Let $l\geq 9$.

\begin{table}[!htb]
\begin{tiny}
    \begin{tblr}{|c||c|c|c|c|c|c|c|c|c|}
        \hline
        \backslashbox{\!$j$\!}{\!$i$\!} & $0$ & $1$ & $2$ & $3$ & $4$ & $5$ & $6$ & $7$ & $8$ \\
        \hline
        \hline
        $21$  &   &   &   &   &   &   &   &   & $ \Rone $ \\
        \hline
        $19$  &   &   &   &   &   &   &   &   & $ \Rone $ \\
        \hline
        $17$  &   &   &   &   &   & $ \Rone $ & $ \Rone $ &   &   \\
        \hline
        $15$  &   &   &   &   &   & $ \Tone{2} $ & $ \Tone{2} $ &   &   \\
        \hline
        $13$  &   &   &   & $ \Rone $ & $ \Rmor{3} $ & $ \Rone $ &   &   &   \\
        \hline
        $11$  &   &   &   & $ \Tone{2} $ & $ \Rmor{2} $ &   &   &   &   \\
        \hline
        $9$  &   &   & $ \Rone $ &   &   &   &   &   &   \\
        \hline
        $7$  & $ \Rone $ &   &   &   &   &   &   &   &   \\
        \hline
        $5$  & $ \Rone $ &   &   &   &   &   &   &   &   \\
        \hline
    \end{tblr}
    \end{tiny}
\caption{$H(\widehat{\sigma_1^2\sigma_2^2\sigma_1^2\sigma_2^2})$.}\label{tab:H(aabbaabb)}
\end{table}

\newpage

In Figure~\ref{fig:C4b} we show diagrams $D$, $D_A$ and $D_B$. We distinguish two cases: 
\begin{itemize}[topsep=0pt]
\item If $k_1 > 2$, then $D_A$ is the standard diagram of the closure of a braid $\gamma$ with $\inf_s(\gamma)=0$, $\operatorname{sl}(\gamma)\geq 4$ and $l(\gamma) = l-1$, and therefore by the induction hypothesis the homology of $\widehat{\gamma}$ has the expected $\Lcorner_{4,3}$-shape. 

\item If $k_1=2$, then then $D_A$ is the standard diagram of the closure of a braid $\gamma$ which is conjugate to a braid (not equal to $\Delta$) with a strictly positive infimum (i.e., a braid in $\mathbf{C4a}$), and its closure also provides the expected $\Lcorner_{4,3}$-shape, as proved in Section \ref{subsection:inf_s>0}. 
\end{itemize}

    \begin{figure}[h!]
    \centering
    \includegraphics[scale=0.24]{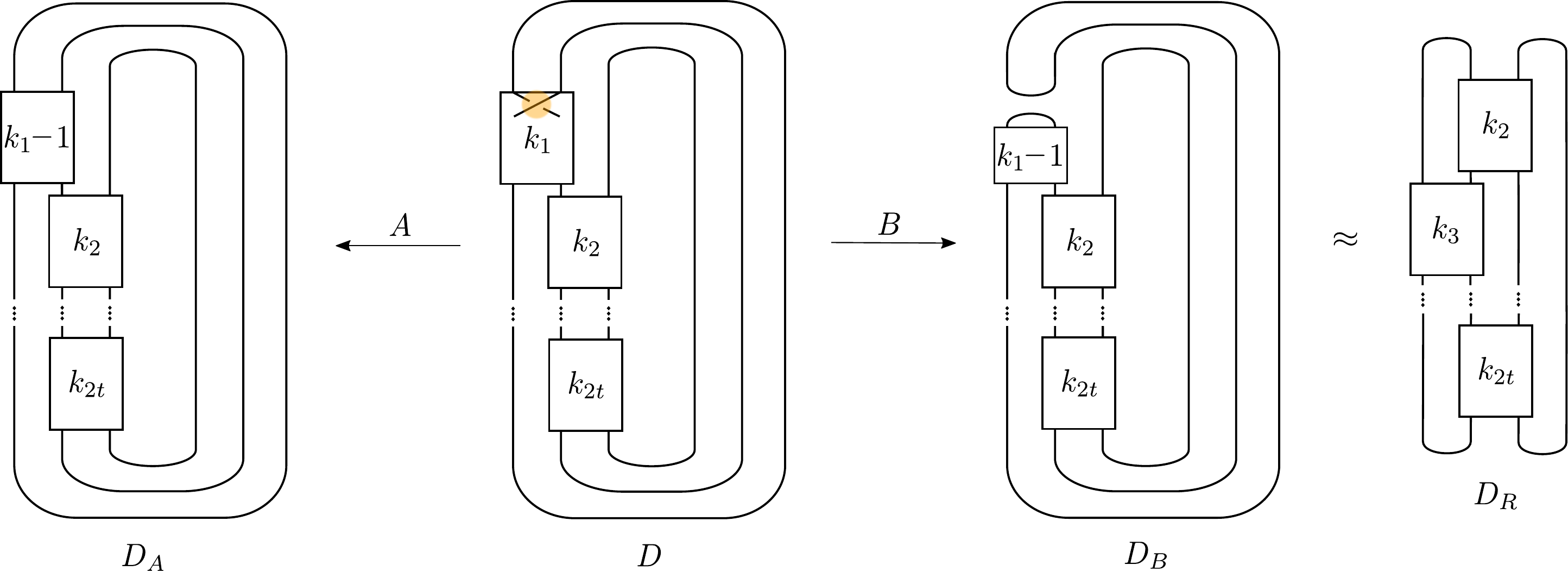}
    \caption{$A$ and $B$-smoothing of a crossing in $D$.}
    \label{fig:C4b}
    \end{figure}

To analyze inequalities in \eqref{inequalities_proof_Th_5.2}, we proceed as in the proof of subcase $\mathbf{C4a.4}$ setting $u=0$. Although $u=0$ is not considered in that subcase, the proof is analogous, since the parameters $t,k_1 \geq 2$ and at least one of them is greater than or equal to 3 (since $l \geq 9$), so the final inequalities remain true.
\end{proof}

The only cases of braids with summit infimum equal to zero that have not been covered so far correspond to those conjugate to $1$, $\sigma_1$, $\sigma_1^2$, $\sigma_1\sigma_2$ and $\sigma_1^2\sigma_2^2$, which are braids in the set $N$ in the statement of Theorem~\ref{th:main}.

\section{The $\Lcorner_{4\lfloor p/2\rfloor+4,3\lfloor p/2\rfloor+3}$-shape for closed 3-braids with infimum $p \geq 0$}\label{sec:final} 

In this section we will use Theorem \ref{th:main} and a result by T. Jaeger in \cite{Jaeger_2011} to obtain explicit expressions for the first $4\lfloor p/2\rfloor + 4$ columns and $3\lfloor p/2\rfloor + 3$ lowest rows of the Khovanov homology of any closed positive $3$-braid, where $p \geq 0$ is the infimum of the braid.

\begin{theorem}[{\cite[Th. 4.4.1]{Jaeger_2011}}]\label{th:Jaeger}
Given a word $w$ representing a positive 3-braid, consider the braid word
$$
r(w)= \begin{cases} 1 & \text { if } w=1, \\ \sigma_1 & \text { if } \operatorname{sl}(w)=1, \\ \sigma_1 \sigma_2 & \text { if } \operatorname{sl}(w) \geq 2. \end{cases}
$$

\begin{enumerate}[label=(\roman*)]
    \item  There exists a bigraded $\mathbb{Z}$-module $M_w$ such that\footnote{For any bigraded $\mathbb{Z}$-module $X=X^{i,j}$, $X\{ k \}$ (resp. $X[k]$) denotes the $\mathbb{Z}$-module given by $(X\{k\})^{i,j}=X^{i,j-k}$ (resp. $(X[k])^{i,j}=X^{i-k,j}$).}
$$
H(\widehat{w}) \cong H(\widehat{r(w)})\{l(w)-l(r(w))\}  \oplus M_w.
$$

\item It holds that $H(\widehat{\Delta^2 w}) \cong  H(\widehat{\Delta^2 r(w)})\{l(w)-l(r(w))\} \oplus M_w[4]\{12\} $. 
\end{enumerate}

\end{theorem}

Statement $(i)$ of Theorem $\ref{th:Jaeger}$ tells us that the Khovanov homology of a positive $3$-braid represented by $w$ can be decomposed as a direct sum of two $\mathbb{Z}$-modules. One of those is the Khovanov homology of the closure of $r(w)$ (with a proper shifting in the $j$-degree), which is an $n$-component trivial link, where $n \in \{ 1,2,3\}$ depends on $w$; in fact, $\widehat{r(w)}$ is the unknot, provided $w \neq 1, \sigma_i^{k}$. In any case, $M_w$ comprises \emph{almost all} the Khovanov homology of $\widehat{w}$, except for ``some pieces'' collected in $H(\widehat{r(w)})$. Statement $(ii)$ enables us to determine the Khovanov homology of $\widehat{\Delta^2 w}$ from the homology of $\widehat{w}$, by adding a suitable shift of  $H(\widehat{\Delta^2 r(w)})$ to a suitable shift of $M_w$.

\begin{remark}
    The underlying principle in deriving $H(\widehat{\Delta^2w})$ from $H(\widehat{w})$ through Theorem \ref{th:Jaeger} consists of removing certain pieces (specifically $H(\widehat{r(w)})$) from $H(\widehat{w})$ and then incorporating the block represented by $H(\widehat{\Delta^2r(w)})$. It is important to observe that the homology of $\widehat{r(w)}$ is concentrated at $i=0$ (see Tables \ref{tab:H(1)}, \ref{tab:H(a)} and \ref{tab:H(ab)}), and the smallest $j$ for which $H^{i,j}(\widehat{r(w)})\{l(w)-l(r(w))\} \neq 0$ coincides with $\underline{j}(\widehat{w})$. Therefore, the pieces that are removed correspond to the first column of $H(\widehat{w})$. In addition, it is worth noting that the shifting in $i$-degree in statement $(ii)$ implies that there are no nontrivial groups overlapping (as a direct sum), since $\overline{i}(\widehat{\Delta^2 r(w)}) \leq 5$ in any case (see Tables \ref{tab:abaaba}--\ref{tab:abaabaab}) and the column $i=1$ of $H(\widehat{w})$ does not contain non-trivial groups (see Theorem \ref{th:main} or \cite{Stosic_2005}). 
\end{remark}

\begin{table}[!htb]
 \begin{minipage}[b]{.3\linewidth}
      \centering
       \captionsetup{width=.7\linewidth}
\begin{tiny}
\begin{tblr}{|c||c|c|c|c|c|}
\hline
\backslashbox{\!$j$\!}{\!$i$\!} & $0$ & $1$ & $2$ & $3$ & $4$ \\
\hline
\hline
$13$  &   &   &   &   & $ \Rmor{2} $ \\
\hline
$11$  &   &   &   & $ \Rone $ & $ \Rmor{3} $ \\
\hline
$9$  &   &   &   & $ \Tone{2} $ & $ \Rone $ \\
\hline
$7$  &   &   & $ \Rone $ &   &   \\
\hline
$5$  & $ \Rone $ &   &   &   &   \\
\hline
$3$  & $ \Rone $ &   &   &   &   \\
\hline
\end{tblr} \end{tiny}
\caption{$H(\widehat{\Delta^2}).$}\label{tab:abaaba}
\end{minipage}
\begin{minipage}[b]{.3\linewidth}
      \centering
       \captionsetup{width=.8\linewidth}
\begin{tiny}
\begin{tblr}{|c||c|c|c|c|c|c|}
\hline
\backslashbox{\!$j$\!}{\!$i$\!} & $0$ & $1$ & $2$ & $3$ & $4$ & $5$ \\
\hline
\hline
$16$  &   &   &   &   &   & $ \Rone $ \\
\hline
$14$  &   &   &   &   &   & $ \Tone{2} $ \\
\hline
$12$  &   &   &   & $ \Rone $ & $ \Rmor{2} $ &   \\
\hline
$10$  &   &   &   & $ \Tone{2} $ & $ \Rone $ &   \\
\hline
$8$  &   &   & $ \Rone $ &   &   &   \\
\hline
$6$  & $ \Rone $ &   &   &   &   &   \\
\hline
$4$  & $ \Rone $ &   &   &   &   &   \\
\hline
\end{tblr} \end{tiny}
\caption{$H(\widehat{\Delta^2 \sigma_1}).$}\label{tab:abaabaa}
\end{minipage}
\hspace{0.15cm}
\begin{minipage}[b]{.32\linewidth}
      \centering
       \captionsetup{width=.8\linewidth}
\begin{tiny}
\begin{tblr}{|c||c|c|c|c|c|c|}
\hline
\backslashbox{\!$j$\!}{\!$i$\!} & $0$ & $1$ & $2$ & $3$ & $4$ & $5$ \\
\hline
\hline
$17$  &   &   &   &   &   & $ \Rone $ \\
\hline
$15$  &   &   &   &   &   & $ \Rone $ \\
\hline
$13$  &   &   &   & $ \Rone $ & $ \Rone $ &   \\
\hline
$11$  &   &   &   & $ \Tone{2} $ & $ \Rone $ &   \\
\hline
$9$  &   &   & $ \Rone $ &   &   &   \\
\hline
$7$  & $ \Rone $ &   &   &   &   &   \\
\hline
$5$  & $ \Rone $ &   &   &   &   &   \\
\hline
\end{tblr} \end{tiny}
\caption{$H(\widehat{\Delta^2 \sigma_1 \sigma_2}).$}\label{tab:abaabaab}
\end{minipage}
\end{table}

\begin{example} Consider the positive $3$-braid $\beta$ represented by $w=\sigma_1^5 \sigma_2^4$. We have that $l(w)=9$, $r(w)=\sigma_1 \sigma_2$ and $l(r(w))=2$. By statement $(i)$ of Theorem \ref{th:Jaeger}, we can decompose the Khovanov homology of $\widehat{\beta}$ as
\[ H(\widehat{\beta})= H(\widehat{\sigma_1 \sigma_2})\{7\} \oplus M_w . \]
The $\mathbb{Z}$-module $M_w$ corresponds to the yellow block in Table \ref{tab:H(aaaaabbbb)}. Table \ref{tab:H(Delta2_aaaaabbbb)} shows the Khovanov homology of $\widehat{\Delta^2 \beta}$, which can be computed by adding $H(\widehat{\Delta^2 \sigma_1 \sigma_2})\{ 7 \}$ (grey block) to $M_w[4]\{12\}$ (yellow block). This illustrates statement $(ii)$ of Theorem \ref{th:Jaeger}.
    \begin{table}[!htb]
    \begin{minipage}{\linewidth}
      \centering
          \begin{tiny}
          \begin{tblr}{colspec={|c||c|c|c|c|c|c|c|c|c|c|},cell{2-10}{4-11} = {yellow}}
\hline
\backslashbox{\!$j$\!}{\!$i$\!} & $0$ & $1$ & $2$ & $3$ & $4$ & $5$ & $6$ & $7$ & $8$ & $9$ \\
\hline
\hline
$26$  &   &   &   &   &   &   &   &   &   & $ \Rone $ \\
\hline
$24$  &   &   &   &   &   &   &   &   & $ \Rone $ & $ \Tone{2} $ \\
\hline
$22$  &   &   &   &   &   &   &   & $ \Rmor{2} $ & $ \Rone \oplus \Tone{2} $ &   \\
\hline
$20$  &   &   &   &   &   &   & $ \Rone $ & $ \Rone \oplus \Tmor{2}{2} $ &   &   \\
\hline
$18$  &   &   &   &   &   & $ \Rmor{2} $ & $ \Rmor{2} \oplus \Tone{2} $ &   &   &   \\
\hline
$16$  &   &   &   &   & $ \Rone $ & $ \Rone \oplus \Tmor{2}{2} $ &   &   &   &   \\
\hline
$14$  &   &   &   & $ \Rmor{2} $ & $ \Rmor{3} $ &   &   &   &   &   \\
\hline
$12$  &   &   &   & $ \Tmor{2}{2} $ &   &   &   &   &   &   \\
\hline
$10$  &   &   & $ \Rmor{2} $ &   &   &   &   &   &   &   \\
\hline
$8$  & $ \Rone $ &   &   &   &   &   &   &   &   &   \\
\hline
$6$  & $ \Rone $ &   &   &   &   &   &   &   &   &   \\
\hline
\end{tblr}
           \end{tiny}
    \caption{$H(\widehat{\sigma_1^5\sigma_2^4})$.}\label{tab:H(aaaaabbbb)}
    \end{minipage}
    \end{table}
    
    \begin{table}[!htb]
        \begin{minipage}{\linewidth}
      \centering
          \begin{tiny}
          \begin{tblr}{colspec={|c||c|c|c|c|c|c|c|c|c|c|c|c|c|c|},cell{2-10}{8-15} = {yellow}, cell{9-15}{2-7}={grey}}
\hline
\backslashbox{\!$j$\!}{\!$i$\!} & $0$ & $1$ & $2$ & $3$ & $4$ & $5$ & $6$ & $7$ & $8$ & $9$ & $10$ & $11$ & $12$ & $13$ \\
\hline
\hline
$38$  &   &   &   &   &   &   &   &   &   &   &   &   &   & $ \Rone $ \\
\hline
$36$  &   &   &   &   &   &   &   &   &   &   &   &   & $ \Rone $ & $ \Tone{2} $ \\
\hline
$34$  &   &   &   &   &   &   &   &   &   &   &   & $ \Rmor{2} $ & $ \Rone \oplus \Tone{2} $ &   \\
\hline
$32$  &   &   &   &   &   &   &   &   &   &   & $ \Rone $ & $ \Rone \oplus \Tmor{2}{2} $ &   &   \\
\hline
$30$  &   &   &   &   &   &   &   &   &   & $ \Rmor{2} $ & $ \Rmor{2} \oplus \Tone{2} $ &   &   &   \\
\hline
$28$  &   &   &   &   &   &   &   &   & $ \Rone $ & $ \Rone \oplus \Tmor{2}{2} $ &   &   &   &   \\
\hline
$26$  &   &   &   &   &   &   &   & $ \Rmor{2} $ & $ \Rmor{3} $ &   &   &   &   &   \\
\hline
$24$  &   &   &   &   &   & $ \Rone $ &   & $ \Tmor{2}{2} $ &   &   &   &   &   &   \\
\hline
$22$  &   &   &   &   &   & $ \Rone $ & $ \Rmor{2} $ &   &   &   &   &   &   &   \\
\hline
$20$  &   &   &   & $ \Rone $ & $ \Rone $ &   &   &   &   &   &   &   &   &   \\
\hline
$18$  &   &   &   & $ \Tone{2} $ & $ \Rone $ &   &   &   &   &   &   &   &   &   \\
\hline
$16$  &   &   & $ \Rone $ &   &   &   &   &   &   &   &   &   &   &   \\
\hline
$14$  & $ \Rone $ &   &   &   &   &   &   &   &   &   &   &   &   &   \\
\hline
$12$  & $ \Rone $ &   &   &   &   &   &   &   &   &   &   &   &   &   \\
\hline
\end{tblr}
            \end{tiny}
    \caption{$H(\widehat{\Delta^2\sigma_1^5\sigma_2^4})$.}\label{tab:H(Delta2_aaaaabbbb)}
    \end{minipage}
    \end{table}
\end{example}

\begin{mythm}{1.2}
    Let $\beta$ be a positive 3-braid and write $p=\inf(\beta)$. Define the quantities $\mathbf{i}(p) =  4 \lfloor \frac{p}{2} \rfloor + 3$ and $\mathbf{j}(p)= \underline{j}(\widehat{\beta}) +6 \lfloor \frac{p}{2} \rfloor + 4$. Then, $H^{i,j}(\widehat{\beta})$ for every $(i,j)$ with $i=0,1, \dots, \mathbf{i}(p)$ or $j=\underline{j}(\widehat{\beta}), \underline{j}(\widehat{\beta}) + 2, \dots, \mathbf{j}(p)$ is as shown in one of the Tables \ref{tab:Delta_p_even}--\ref{tab:Delta_p_C4}. The precise table corresponding to $\widehat{\beta}$ can be deduced from its normal form.
\end{mythm}
\begin{proof}
    We can write $\beta$ as $\Delta^{2 \lfloor \frac{p}{2} \rfloor} \gamma$ such that $\gamma=\Delta^{\inf (\gamma)} a_1 \cdots a_{\ell}$ is written in normal form, with $\inf(\gamma) \in \{ 0, 1\}$. The braid $\gamma$  must be conjugate to a braid $\gamma'$ fitting into one of the families $\mathbf{C}r$, $r=\mathbf{1}, \mathbf{2}, \mathbf{3}, \mathbf{4}$, or into $\{ 1, \sigma_1, \sigma_1^2, \sigma_1\sigma_2, \sigma_1^2\sigma_2^2, \Delta\}$. As $\Delta^2$ commutes with every generator in the braid group, $\beta$ must be conjugate to $\beta' = \Delta^{2 \lfloor \frac{p}{2} \rfloor} \gamma'$, their closures are equivalent and they have the same Khovanov homology. 
    
    If $\gamma'$ is not one of the exceptional cases, the $\Lcorner_{4,3}$-shape of the Khovanov homology of its closure is determined by Theorem \ref{th:main}. After applying $\lfloor \frac{p}{2} \rfloor$ times Theorem \ref{th:Jaeger}  we are done, since each time it is applied 4 more columns and 3 more rows are determined. Then we would obtain the $\Lcorner_{4\lfloor p/2\rfloor+4,3\lfloor p/2\rfloor+3}$-shape corresponding to $\beta'$, which aligns with the degrees described in the statement. 

    If $\gamma'$ is an exceptional case, we do not need to use Theorem \ref{th:main}, but use the Khovanov homology of its closure directly and apply again $\lfloor \frac{p}{2} \rfloor$ times Theorem \ref{th:Jaeger}.
\end{proof}

\begin{myprop}{1.4}
The $\Lcorner_{4\lfloor p/2\rfloor+4,3\lfloor p/2\rfloor+3}$-shape of the Khovanov homology of the closure of a 3-braid with summit infimum $p\geq 0$ can be computed in linear time.
\end{myprop}
\begin{proof}
Suppose we are given a braid word of length $k$. That is, a word in the letters $\{\sigma_1^{\pm 1}, \sigma_2^{\pm 1}\}$. The standard way to compute its left normal form starts with the following steps:
\begin{itemize}
\item Replace each $\sigma_1^{-1}$ with $\Delta^{-1}\sigma_1\sigma_2$, and each $\sigma_2^{-1}$ with $\Delta^{-1}\sigma_2\sigma_1$.
    
\item Slide each instance of $\Delta^{-1}$ to the left, conjugating all the letters on its left.
\end{itemize}

Since conjugation by $\Delta^{-1}$ swaps $\sigma_1$ and $\sigma_2$, this process can be done in time $O(k)$, and we obtain an equivalent braid word of the form $\Delta^p s_1\cdots s_t$, where $p\in \mathbb Z$, $s_i\in\{\sigma_1,\sigma_2\}$ for every $i=1,\ldots, t$, and $t\leq 2k$. This is not yet the left normal form of the input braid.

We will collect consecutive instances of $\sigma_1$ and consecutive instances of $\sigma_2$, to write braid words as $\Delta^p \sigma_{[i_1]}^{k_1}\cdots \sigma_{[i_m]}^{k_m}$, where $\sigma_{[i_j]}\in \{\sigma_1,\sigma_2\}$ and $\sigma_{[i_j]}\neq \sigma_{[i_{j+1}]}$ for every $j=1,\ldots,m-1$. We know that such a word will be in left normal form if and only if $k_2,\ldots,k_{m-1}\geq 2$.  We will store such a word as a list of integers $(p,a,[k_1,\ldots,k_m],b)$, where $a$ is the index of $\sigma_{[1]}$ and $b$ is the index of $\sigma_{[m]}$. The number $b$ is superfluous, as $b=a$ if $m$ is odd and $b=3-a$ if $m$ is even, but it will be helpful for explaining the forthcoming algorithm. In case the braid word is just $\Delta^p$, we will store it as $(p,0,[],0)$.

Suppose that we are given a braid word $(p,a,[k_1,\ldots,k_m],b)$ in left normal form, representing a braid $\alpha$, and let us denote $n_{\alpha}=k_1+\cdots+k_m$. Given $i\in \{1,2\}$, we will now see that we can compute the left normal form of $\alpha\sigma_i$ in constant time. Indeed, by the properties of normal forms, we have the following situations:
\begin{itemize}

\item If $n_{\alpha}=0$, the normal form of $\alpha\sigma_i$ is $(p,i,[1],i)$.

\item If $i=b$, the normal form of $\alpha\sigma_i$ is $(p,a,[k_1,\ldots,k_m+1],i)$.

\item If $i\neq b$, $k_m=1$ and $n_\alpha=1$, the normal form of $\alpha\sigma_i$ is $(p,a,[1,1],i)$.

\item If $i\neq b$, $k_m=1$ and $n_{\alpha}=2$, the normal form of $\alpha\sigma_i$ is $(p+1,0,[],0)$.

\item If $i\neq b$, $k_m=1$ and $n_{\alpha}>2$, the normal form of $\alpha\sigma_i$ is $(p+1,3-a,[k_1,\ldots,k_{m-1}-1],3-i)$.

\item If $i\neq b$ and $k_m>1$, the normal form of $\alpha\sigma_i$ is $(p,a,[k_1,\ldots,k_m,1],i)$.

\end{itemize}

Notice that in each case we just need to check the final letter of the normal form, the last exponent and whether the sum of exponents is either 1, 2, or greater. And we need to add/remove/modify at most 4 items of the list. Hence, this can be done in constant time.

Now we go back to our braid $\Delta^p s_1\cdots s_t$. We will compute its left normal form in time $O(t)=O(k)$. First, $\Delta^p s_1$ is already in left normal form. Now we assume we have computed the left normal form of $\Delta^p s_1\cdots s_{i-1}$, and we compute the left normal form of $\Delta^p s_1\cdots s_i$ in constant time, as we saw above. Hence, after $t$ steps we have computed the normal form of the whole braid in time $O(t)$.

Let $\Delta^p \sigma_{[1]}^{k_1}\cdots \sigma_{[m]}^{k_m}$ be the left normal form of the input braid $\beta$, which has been computed in time $O(k)$. We notice that $n_{\beta}\leq 2k$. We can now apply the conjugations explained in Proposition~\ref{prop:representatives_conjugacy_classes_5_families}, to transform the input braid $\beta$ into a conjugate which belongs to one of the families $\Lambda_i$. We assume that the initial letter of the non-$\Delta$ part, if it exists, is $\sigma_1$ (if it were $\sigma_2$, we assume that the input braid is the conjugate of $\beta$ by $\Delta$). Hence, the left normal form of $\beta$ is stored as $(p,1,[k_1,\ldots,k_m],b)$, where $b=1$ if $m$ is odd and $b=2$ if $m$ is even. To simplify the expessions, we will just store $(p,[k_1,\ldots,k_m])$, and also the number $n_{\beta}=k_1+\cdots+k_m$ (which can be computed in time $O(k)$).

Now we follow the proof of Proposition~\ref{prop:representatives_conjugacy_classes_5_families}. If $n_\beta\leq 2$, either $\beta$ already belongs to some $\Lambda_i$ or it can be conjugated to a braid in $\Lambda_2$ in constant time. We can then assume that $n_{\beta}\geq 3$, and we consider several cases, either to check that $\beta$ belongs to some $\Lambda_i$ or to conjugate it to another braid $\beta'$ with $n_{\beta'} < n_{\beta}$.
\begin{enumerate}

\item If $m=1$ then $\beta\in \Lambda_2$.

\item If $m>1$:
\begin{itemize}
  
  \item If $m-p$ is odd, set $\beta'=(p,[k_1+k_m,k_2,\ldots,k_{m-1}])$, which belongs to either $\Lambda_2$, or $\Lambda_4$ or $\Lambda_5$.
  
  \item If $m-p$ is even, $k_1=1$ and $k_m=1$, set $\beta'=(p+1,[k_2-1,k_3,\ldots,k_{m-1}])$ and $n_{\beta'}=n_\beta-3$.
      
  \item If $m-p$ is even, $k_1=1$ and $k_m>1$, set $\beta'=(p+1,[k_2-1,k_3,\ldots,k_{m-1},k_{m}-1])$ and $n_{\beta'}=n_\beta-3$.
  
  \item If $m-p$ is even, $k_1>1$ and $k_m=1$, set $\beta'=(p+1,[k_1-1,k_2,\ldots,k_{m-1}-1])$ and $n_{\beta'}=n_\beta-3$.
  
  \item If $m-p$ is even, $k_1>1$ and $k_m>1$, then $\beta\in \Lambda_4$.
  
\end{itemize}
\end{enumerate}

Notice that each case can be computed in constant time, and either we obtain an element in some $\Lambda_i$, or we obtain an element the sum of whose exponents is 3 units smaller. Since this sum started being $n\leq 2k$, iterating this process we will obtain an element $\beta''\in \Lambda_i$, conjugate to $\beta$, in at most $k$ steps, so we obtain such an element in time $O(k)$.

We recall that the infimum of $\beta''$ (say $p$) is the summit infimum of $\beta$ (Remark \ref{remark:summit_infimum_Lambda_i}), hence it is non-negative, and we know the $\Lcorner_{4\lfloor p/2\rfloor+4,3\lfloor p/2\rfloor+3}$-shape of the Khovanov homology of $\beta''$ (hence of $\beta$) by Theorem~\ref{main2}. 

Notice that the above procedure computes the summit infimum of $\beta$, so we can apply this algorithm without knowing, a priori, whether its summit infimum is non-negative. In that case, if the infimum of $\beta''$ is negative, we know that $\beta$ is not conjugate to a positive braid. As we have seen, the whole computation takes time $O(k)$.
\end{proof}

\newpage

\begin{table}[htb!]
    \centering
    \begin{tiny}
    \begin{tblr}{colspec={|c||c|c|c|c|c|c|c|c|c|c|c|c|c|c|c|c|c|c|c|}, cell{13-17}{4-7} = {celeste}, cell{10-14}{8-11} = {celeste}, cell{4-8}{13-16} = {celeste}, cell{7}{12}={celeste}, cell{2-5}{17-19}={verde}}
\hline
\backslashbox{\!$j$\!}{\!$i$\!} & $0$ & $1$ & $2$ & $3$ & $4$ & $5$ & $6$ & $7$ & $8$ & $9$ & $\cdots$ & $\cdots$ & $\cdots$ & $\cdots$ & $\cdots$ & $\cdots$ & $\cdots$ & $\cdots$  \\
\hline
\hline

$\vdots$ &   &   & & & & &  & & &  & & & & & & & & $\Rone^2$       \\ \hline

$\vdots$ &   &   &  & & & & & &   & &  & &      & & & & $\Rone$ & $\Rone^3$ \\ \hline

$\vdots$ & & & & & & & & & & & & & & &  $\Rone$ &  & $\Tone{2}$ & $\Rone$  \\ \hline

$\mathbf{j}(p)$ & &  & & & & & & & & & & & & & $\Rone$ & $\Rone$ & & \\ \hline

$\vdots$  & &  & & & &  & & & & & & & $\Rone$ & $\Rone$ & &  & & \\ \hline

$\vdots$ &   &   &  &  &  & & & & & & \SetCell[r=5]{c}{$\udots$}  & & $\Tone{2}$ & $\Rone$ &        &  & &  \\ \hline

$\vdots$ &   &   &   & &  & & & & &  & & $\Rone$ & & & & \\  \hline

$\vdots$ & & & & &  &  & &   & & & \\  \hline

$\vdots$ & & & & &  &  & &   & & $\Rone$ & & \\  \hline

$\vdots$ & & & & &  &  & &   & & $\Rone$ & & \\  \hline

$\vdots$ & & & & &  &  & & $\Rone$ & $\Rone$    & & & \\  \hline

$\vdots$ & & & & &  &  $\Rone$  & & $\Tone{2}$ & $\Rone$   & & & \\  \hline

$\vdots$ & & & & &  &  $\Rone$  & $\Rone$ &   & & & \\  \hline

$\vdots$ & & & & $ \Rone $ & $\Rone$  & & & & & &  &  & & \\ \hline

$\vdots$ & & & & $ \Tone{2} $ & $\Rone$ & & & & & & & & &  \\ \hline

$\underline{j}(\widehat{\beta}) + 4$  &   &   & $ \Rone $ & &  &   & & &  & & & & & \\ \hline

$\underline{j}(\widehat{\beta}) + 2$  & $ \Rone $ &   &   & &   & & & & & & & & & \\ \hline

$\underline{j}(\widehat{\beta})$    & $ \Rone $ &   &   & & & & & & & & & & &  \\ \hline
\end{tblr}
\end{tiny}
    \caption[Caption for delta_p_even]{$H(\widehat{\Delta^p})$, with\footnotemark $\; p > 0$ even. The number of blue blocks is $\frac{p}{2}-1$.}
    \label{tab:Delta_p_even}
\end{table}

\footnotetext{For $p=0$ see Table \ref{tab:H(1)}.}

\begin{table}[htb!]
    \centering
    \begin{tiny}
    \begin{tblr}{colspec={|c||c|c|c|c|c|c|c|c|c|c|c|c|c|c|c|c|c|c|c|},  
 cell{5-9}{13-16} = {celeste}, cell{11-15}{8-11} = {celeste}, cell{14-18}{4-7} = {celeste}, cell{8}{12}={celeste}, cell{2-6}{17-20} = {naranja}}
\hline
\backslashbox{\!$j$\!}{\!$i$\!} & $0$ & $1$ & $2$ & $3$ & $4$ & $5$ & $6$ & $7$ & $8$ & $9$ & $\cdots$ & $\cdots$ & $\cdots$ & $\cdots$ & $\cdots$ & $\cdots$ & $\cdots$ & $\cdots$ & $\cdots$   \\
\hline
\hline

$\vdots$ & & & & &   &   &   & &  & & & & & & & & & & $\Rone$  \\
\hline
$\mathbf{j}(p)$           &  & & & & &   &   & &  & & & & &  & & &  & & $\Tone{2}$  \\
\hline
$\vdots$           &   &   & & & & &  & &  & & & & & & & & $\Rone$ & $\Rone^2$          \\ \hline

$\vdots$           &   &   & & & & &  & &  & & & & & & $\Rone$  & & $\Tone{2}$ & $\Rone$ &          \\ \hline

$\vdots$ &   &   & & &  & & & & & & & & & & $\Rone$ & $\Rone$   \\ \hline

$\vdots$ &   &   & & &  & & & & & & & & $\Rone$  & $\Rone$  & & & &   \\ \hline

$\vdots$   &   &   &    &    & & & & & & & \SetCell[r=5]{c}{$\udots$} &  & $\Tone{2}$ & $\Rone$ &  &   &       \\ \hline

$\vdots$ & & & &   & &  & & & & & & $ \Rone $ &  &   & &      \\ \hline

$\vdots$ &   &   &   & &   & &  & &  &   &    \\ \hline

$\vdots$ &   &   &   & &   & &  & &  & $\Rone$   &    \\ \hline

$\vdots$ &   &   &   & &   & &  & &  & $\Rone$   &    \\ \hline

$\vdots$ &   &   &   & &   & &  & $\Rone$ & $\Rone$ &   &    \\ \hline

$\vdots$ &   &   &   & &   & $ \Rone $ &  & $\Tone{2}$ & $\Rone$ &  &   &   \\ \hline

$\vdots$ &   &   &   & &   & $ \Rone $ & $\Rone$  &  &  &  &      \\ \hline

$\vdots$  &   &   &   & $ \Rone $ & $\Rone$ & & & & & & &  \\ \hline

$\vdots$  &   &   &   & $ \Tone{2} $ & $\Rone$ & & & & & & &    \\ \hline

$\underline{j}(\widehat{\beta}) + 4$  &   &   & $ \Rone $ & &  &   & & & &  & & \\ \hline

$\underline{j}(\widehat{\beta}) + 2$  & $ \Rone $ &   &   & &   & & & &  &  & &  \\ \hline

$\underline{j}(\widehat{\beta})$    & $ \Rone $ &   &   & & & & & &  & & &    \\ \hline

\end{tblr}
\end{tiny}
    \caption[Caption for delta_p_sigma1]{$H(\widehat{\Delta^p \sigma_1})$, with\footnotemark $\; p > 0$ even. The number of blue blocks is $\frac{p}{2}-1$.}
    \label{tab:Delta_p_sigma1}
\end{table}

\footnotetext{For $p=0$ see Table \ref{tab:H(a)}.}

\newpage

\begin{table}[htb!]
    \centering
    \begin{tiny}
    \begin{tblr}{colspec={|c||c|c|c|c|c|c|c|c|c|c|c|c|c|c|c|c|c|c|c|c|},  
 cell{12-16}{8-11} = {celeste}, cell{6-10}{13-16}={celeste}, cell{9}{12}={celeste}, cell{15-19}{4-7} = {celeste}, cell{3-7}{17-20} = {naranja}}
\hline
\backslashbox{\!$j$\!}{\!$i$\!} & $0$ & $1$ & $2$ & $3$ & $4$ & $5$ & $6$ & $7$ & $8$ & $9$ & $\cdots$ & $\cdots$ & $\cdots$ & $\cdots$ & $\cdots$ & $\cdots$ & $\cdots$ & $\cdots$ & $\cdots$ & $\cdots$ \\
\hline
\hline

$\vdots$  &   &   &   & &  & & & & & & &  & &  & & & &    &  & $\Rone$   \\ \hline

$\vdots$ & & & & &   &   &   & &  & & & & & & &  &  &  & $\Rone$ & $\Rone^2$     \\ \hline

$\mathbf{j}(p)$ & & & & &   &   &   & &  & & & & &  & & &  & & $\Tone{2}$ & $\Rone$       \\ \hline

$\vdots$ & & & & &   &   &   & &  & & & & & & & & $\Rone$ & $\Rone^2$ &             \\
\hline
$\vdots$  & & & &   &   &   &   & &  & & & & & & $\Rone$  & & $\Tone{2}$ & $\Rone$ &        &     \\
\hline

$\vdots$ &  & & & & &   & & &  & & & & & & $\Rone$ & $\Rone$     \\ \hline

$\vdots$ & & & & & & & & & & &  & & $\Rone$  & $\Rone$  & & &   \\ \hline

$\vdots$   &   &   &  & & & &  &    & &  & \SetCell[r=5]{c}{$\udots$} &  & $\Tone{2}$ & $\Rone$ &  &   & &       \\ \hline

$\vdots$ & & & &   & & & & & & & & $ \Rone $ &  &   & & &      \\ \hline

$\vdots$ &   &   &   & &   & &  & &  &   & &     \\ \hline

$\vdots$ &   &   &   & &   & &  & &  & $\Rone$  & &     \\ \hline

$\vdots$ &   &   &   & &   & &  & &  & $\Rone$  & &     \\ \hline

$\vdots$ &   &   &   & &   & &  & $\Rone$ & $\Rone$ &   & &     \\ \hline

$\vdots$ &   &   &   & &   & $ \Rone $ &  & $\Tone{2}$ & $\Rone$ &   & &   \\ \hline

$\vdots$ &   &   &   & &   & $ \Rone $ & $\Rone$ &  &  &  & &    \\ \hline

$\vdots$  &   &   &   & $ \Rone $ & $\Rone$ & & & & & & & &  \\
\hline

$\vdots$  &   &   &   & $ \Tone{2} $ & $\Rone$ & & & & & & & &  \\ \hline

$\underline{j}(\widehat{\beta}) + 4$  &   &   & $ \Rone $ & &  &   & & & &  & & &   \\ \hline

$\underline{j}(\widehat{\beta}) + 2$  & $ \Rone $ &   &   & &   & & & &  & & &  \\ \hline

$\underline{j}(\widehat{\beta})$    & $ \Rone $ &   &   & & & & & &  & & &   \\ \hline

\end{tblr}
\end{tiny}
    \caption[Caption for Delta_p_sigma1-2]{$H(\widehat{\Delta^{p}\sigma_1^2})$, with\footnotemark $\; p > 0$ even. The number of blue blocks is $\frac{p}{2}-1$.}
    \label{tab:Delta_p_sigma1^2}
\end{table}

\footnotetext{For $p=0$ see Table \ref{tab:H(aa)}.}

\begin{table}[htb!]
    \centering
    \begin{tiny}
    \begin{tblr}{colspec={|c||c|c|c|c|c|c|c|c|c|c|c|c|c|c|c|c|}, cell{11-15}{4-7} = {celeste}, cell{2-6}{13-16} = {celeste}, cell{8-12}{8-11}={celeste}, cell{5}{12}={celeste}}
\hline
\backslashbox{\!$j$\!}{\!$i$\!} & $0$ & $1$ & $2$ & $3$ & $4$ & $5$ & $6$ & $7$ & $8$ & $9$ &  $\cdots$ & $\cdots$ & $\cdots$ & $\cdots$ & $\cdots$  \\
\hline
\hline

$\vdots$ & & & & & & &  & & & & & & & &  $\Rone$   \\ \hline

$\mathbf{j}(p)$ & &  & & & & & & & & & & & & & $\Rone$  \\ \hline

$\vdots$  & &  & & & & & & & &  & & & $\Rone$ & $\Rone$ &    \\ \hline

$\vdots$ &   &   &  & & & & & & & & \SetCell[r=5]{c}{$\udots$}  & & $\Tone{2}$ & $\Rone$ &           \\ \hline

$\vdots$ &   &   &   & &   &  & &  & & & & $\Rone$ &  \\  \hline

$\vdots$ & & & & &  &  & &     \\  \hline

$\vdots$ & & & & &  &  & & & & $\Rone$    \\  \hline

$\vdots$ & & & & &  &  & & & & $\Rone$     \\  \hline

$\vdots$ & & & & &  &  & & $\Rone$ & $\Rone$     \\  \hline

$\vdots$ & & & & &  &  $\Rone$  & & $\Tone{2}$ & $\Rone$     \\  \hline

$\vdots$ & & & & &  &  $\Rone$  & $\Rone$ &   & &  \\  \hline

$\vdots$ & & & & $ \Rone $ & $\Rone$  & & & & & &   \\ \hline

$\vdots$ & & & & $ \Tone{2} $ & $\Rone$ & & & & & &  \\ \hline

$\underline{j}(\widehat{\beta}) + 4$  &   &   & $ \Rone $ & &  &   & & &  & & \\ \hline

$\underline{j}(\widehat{\beta}) + 2$  & $ \Rone $ &   &   & &   & & & & & &  \\ \hline

$\underline{j}(\widehat{\beta})$    & $ \Rone $ &   &   & & & & & & & &   \\ \hline
\end{tblr}
\end{tiny}
    \caption{$H(\widehat{\Delta^{p}\sigma_1\sigma_2})$, with $p\geq 0$ even. The number of blue blocks is $\frac{p}{2}$.}
    \label{tab:Delta_p_sigma1_sigma2}
\end{table}


\begin{table}[htb!]
    \centering
    \begin{tiny}
    \begin{tblr}{colspec={|c||c|c|c|c|c|c|c|c|c|c|c|c|c|c|c|c|c|c|c|}, cell{13-17}{4-7} = {celeste}, cell{4-8}{13-16} = {celeste}, cell{10-14}{8-11}={celeste}, cell{7}{12}={celeste}}
\hline
\backslashbox{\!$j$\!}{\!$i$\!} & $0$ & $1$ & $2$ & $3$ & $4$ & $5$ & $6$ & $7$ & $8$ & $9$ &  $\cdots$ & $\cdots$ & $\cdots$ & $\cdots$ & $\cdots$ & $\cdots$ & $\mathbf{i}(p)$ & $\cdots$  \\
\hline
\hline

$\vdots$ &   &   &   & & & & & & &  & & & & & & & & $\Rone$       \\ \hline

$\vdots$ &   &   &   & &   & &  & & & & & &  & & & &  & $\Rone$ \\ \hline

$\vdots$ & & & & & & &  & & & & & & & &  $\Rone$ & $\Rone^2$  &  &   \\ \hline

$\mathbf{j}(p)$ & &  & & & & & & & & & & & & & $\Rone$ & $\Rone^2$ & & \\ \hline

$\vdots$  & &  & & & &  & & & & & & & $\Rone$ & $\Rone$ & &  & & \\ \hline

$\vdots$ &   &   &  &  & & & & & & & \SetCell[r=5]{c}{$\udots$}  & & $\Tone{2}$ & $\Rone$ &        &  & &  \\ \hline

$\vdots$ &   &   &   & & & & & & &  & & $\Rone$ & & & & \\  \hline

$\vdots$ & & & & &  &  & &   & & & \\  \hline

$\vdots$ & & & & &  &  & &   & & $\Rone$ & \\  \hline

$\vdots$ & & & & &  &  & &   & & $\Rone$ & \\  \hline

$\vdots$ & & & & &  &  & & $\Rone$  & $\Rone$ & & \\  \hline

$\vdots$ & & & & &  &  $\Rone$  & & $\Tone{2}$    & $\Rone$ & & \\  \hline

$\vdots$ & & & & &  &  $\Rone$  & $\Rone$ &   & & & \\  \hline

$\vdots$ & & & & $ \Rone $ & $\Rone$  & & & & & &  &  & & \\ \hline

$\vdots$ & & & & $ \Tone{2} $ & $\Rone$ & & & & & & & & &  \\ \hline

$\underline{j}(\widehat{\beta}) + 4$  &   &   & $ \Rone $ & &  &   & & &  & & & & & \\ \hline

$\underline{j}(\widehat{\beta}) + 2$  & $ \Rone $ &   &   & &   & & & & & & & & & \\ \hline

$\underline{j}(\widehat{\beta})$    & $ \Rone $ &   &   & & & & & & & & & & &  \\ \hline
\end{tblr}
\end{tiny}
    \caption{$H(\widehat{\Delta^p\sigma_1^2 \sigma_2^2})$, with $p\geq 0$ even. The number of blue blocks is $\frac{p}{2}$.}
    \label{tab:Delta_p_even_sigma1^2_sigma2^2}
\end{table}

\newpage

\begin{table}[htb!]
    \centering
    \begin{tiny}
    \begin{tblr}{colspec={|c||c|c|c|c|c|c|c|c|c|c|c|c|c|c|c|c|c|}, cell{11-15}{4-7} = {celeste}, cell{2-6}{13-16} = {celeste}, cell{8-12}{8-11}={celeste}, cell{5}{12}={celeste}}
\hline
\backslashbox{\!$j$\!}{\!$i$\!} & $0$ & $1$ & $2$ & $3$ & $4$ & $5$ & $6$ & $7$ & $8$ & $9$ &  $\cdots$ & $\cdots$ & $\cdots$ & $\cdots$ & $\cdots$ & $\cdots$ \\
\hline
\hline

$\vdots$ & & & & & & &  & & & & & & & &  $\Rone$ & $\Rone$  \\ \hline

$\mathbf{j}(p)$ & &  & & & & & & & & &  & & & & $\Rone$ & $\Rone$  \\ \hline

$\vdots$  & &  & & & &  & & & & & & &  $\Rone$ & $\Rone$ & &   \\ \hline

$\vdots$ &   &   &  &  & & & & & & & \SetCell[r=5]{c}{$\udots$}  & & $\Tone{2}$ & $\Rone$ &        &   \\ \hline

$\vdots$ &   &   &   & & & & & & &  & & $\Rone$ & &  \\  \hline

$\vdots$ & & & & &  &  & &   &  \\  \hline

$\vdots$ & & & & &  &           & &    & & $\Rone$  \\  \hline

$\vdots$ & & & & &  &           & &    & & $\Rone$  \\  \hline

$\vdots$ & & & & &  &           & & $\Rone $ & $\Rone$   &  \\  \hline

$\vdots$ & & & & &  &  $\Rone$  & & $\Tone{2}$ & $\Rone$   &  \\  \hline

$\vdots$ & & & & &  &  $\Rone$  & $\Rone$ &   & & & \\  \hline

$\vdots$ & & & & $ \Rone $ & $\Rone$  & & & & & &  &  \\ \hline

$\vdots$ & & & & $ \Tone{2} $ & $\Rone$ & & & & & & &  \\ \hline

$\underline{j}(\widehat{\beta}) + 4$  &   &   & $ \Rone $ & &  &   & & &  & & & \\ \hline

$\underline{j}(\widehat{\beta}) + 2$  & $ \Rone $ &   &   & &   & & & & & & & \\ \hline

$\underline{j}(\widehat{\beta})$    & $ \Rone $ &   &   & & & & & & & & &  \\ \hline
\end{tblr}
\end{tiny}
    \caption{$H(\widehat{\Delta^{p}})$, with $p\geq 0$ odd. The number of blue blocks is $\frac{p-1}{2}$.}
    \label{tab:Delta_p_odd}
\end{table}


\begin{table}[htb!]
    \centering
    \captionsetup{width=.99\linewidth}
    \begin{tiny}
    \begin{tblr}{colspec={|c||c|c|c|c|c|c|c|c|c|c|c|c|c|c|c|c|c|c|c|c|c|c|},  
 cell{14-18}{8-11} = {celeste}, cell{8-12}{13-16} = {celeste}, cell{17-21}{4-7} = {celeste}, cell{11}{12}={celeste}, cell{5-9}{17-20} = {naranja}}
\hline
\backslashbox{\!$j$\!}{\!$i$\!} & $0$ & $1$ & $2$ & $3$ & $4$ & $5$ & $6$ & $7$ & $8$ & $9$ & $\cdots$ & $\cdots$ & $\cdots$ & $\cdots$ & $\cdots$ & $\cdots$ & $\cdots$ & $\cdots$ & $\cdots$ & $\cdots$ & $\mathbf{i}(p)$ & $\cdots$  \\
\hline
\hline

$\vdots$           &   &   &   & &  & & & & & & & & & & & & & & & & & \SetCell[r=4]{c}{$W_{\widehat{\beta}}$}       \\
\cline{1-22}
$\vdots$           &   &   &   & &  & & & && & & & & & & & & & & & $\Rone$    \\
\cline{1-22}
$\vdots$           &   &   & & & & & & &  & & & & & & &  & &      &  & & $\Rone \oplus \Tone{2}$  \\
\cline{1-22}
$\vdots$           &   & & & & & &   & &  & & & & & & &  &  &  & $\Rone$ & $\Rone$ & $\Tone{2}$   \\
\hline
$\mathbf{j}(p)$  & & & & &   &   &   & &  & & & & &  & & &  & & $\Tone{2}$ & $\Rone$        \\
\hline
$\vdots$           & & & & & &   &   & &  & & & & & & & & $\Rone$ & $\Rone^2$ &            \\
\hline
$\vdots$  & & & & &   &   &   & &  & & & & & & $\Rone$  & & $\Tone{2}$ & $\Rone$ &        &   \\
\hline

$\vdots$ &   &   & & & & & & &  & & & & & & $\Rone$ & $\Rone$ &   \\ \hline

$\vdots$ & & & & & & & & & & & & & $\Rone$  & $\Rone$  & & & & &   \\ \hline

$\vdots$   &   &   &    & & & & &    & &  & \SetCell[r=5]{c}{$\udots$} &  & $\Tone{2}$ & $\Rone$ &  &   & &     \\ \hline

$\vdots$ & & & & & & & &  & &  & & $ \Rone $ &  &   & & &    \\ \hline

$\vdots$ &   &   &   & &   & &  & &  &   & & &    \\ \hline
$\vdots$ &   &   &   & &   & &  & &  & $\Rone$  & & &    \\ \hline
$\vdots$ &   &   &   & &   & &  & &  & $\Rone$  & & &    \\ \hline
$\vdots$ &   &   &   & &   & &  & $\Rone$ & $\Rone$ &   & & &    \\ \hline

$\vdots$ &   &   &   & &   & $ \Rone $ & & $\Tone{2}$ & $\Rone$ &   & &  \\ \hline

$\vdots$ &   &   &   & &   & $ \Rone $ & $\Rone$  &  &  &  & &     \\ \hline

$\vdots$  &   &   &   & $ \Rone $ & $\Rone$ & & & & & & & &   \\
\hline

$\vdots$  &   &   &   & $ \Tone{2} $ & $\Rone$ & & & & & & & &    \\ \hline

$\underline{j}(\widehat{\beta}) + 4$  &   &   & $ \Rone $ & &  &   & & & &  & & &  \\ \hline

$\underline{j}(\widehat{\beta}) + 2$  & $ \Rone $ &   &   & &   & & & &  & & & &   \\ \hline

$\underline{j}(\widehat{\beta})$    & $ \Rone $ &   &   & & & & & &  & & & &  \\ \hline

\end{tblr}
\end{tiny}
    \caption{$H(\widehat{\beta})$, with $\beta = \Delta^p \sigma_1^{k_1}$, $p \geq 0$ even and $k_1 \geq 3$. The number of blue blocks is $\frac{p}{2}-1$.}
    \label{tab:Delta_p_C1}
\end{table}

\newpage

\begin{table}[htb!]
    \centering
    \captionsetup{width=.95\linewidth}
    \begin{tiny}
    \begin{tblr}{colspec={|c||c|c|c|c|c|c|c|c|c|c|c|c|c|c|c|c|c|c|c|}, cell{13-17}{4-7} = {celeste}, cell{10-14}{8-11} = {celeste}, cell{14}{11}={celeste}, cell{7}{12}={celeste}, cell{4-8}{13-16} = {celeste}}
\hline
\backslashbox{\!$j$\!}{\!$i$\!} & $0$ & $1$ & $2$ & $3$ & $4$ & $5$ & $6$ & $7$ & $8$ & $9$ &  $\cdots$ & $\cdots$ & $\cdots$ & $\cdots$ & $\cdots$ & $\cdots$ & $\mathbf{i}(p)$ & $\cdots$  \\
\hline
\hline

$\vdots$ &   &   &   & & & & & & &  & & & & & & & & \SetCell[r=3]{c}{$X_{\widehat{\beta}}$}       \\ \hline

$\vdots$ &   &   &   & & & & &  & & &  & &      & & & & $\Rone$  \\ \hline

$\vdots$ & & & & & & & & & &  & & & & &  $\Rone$ & $\Rone$ & $\Tone{2}$  \\ \hline

$\mathbf{j}(p)$ & &  & & & & & & & & & & & & & $\Rone$ & $\Rone^2$ & & \\ \hline

$\vdots$  & &  & & & &  & & & & & & & $\Rone$ & $\Rone$ & &  & & \\ \hline

$\vdots$ &   &   &  &  & & & & & & & \SetCell[r=5]{c}{$\udots$}  & & $\Tone{2}$ & $\Rone$ &        &  & &  \\ \hline

$\vdots$ &   &   &   & & & & &   & &  & & $\Rone$ & & & & \\  \hline

$\vdots$ & & & & &  &  & &  & & & & & & \\  \hline
$\vdots$ & & & & &  &  & &  & & $\Rone$ & & & & \\  \hline
$\vdots$ & & & & &  &  & &  & & $\Rone$ & & & & \\  \hline
$\vdots$ & & & & &  &  & & $\Rone$    & $\Rone$ & & & & & \\  \hline

$\vdots$ & & & & &  &  $\Rone$  & & $\Tone{2}$ & $\Rone$ & & & & & \\  \hline

$\vdots$ & & & & &  &  $\Rone$  & $\Rone$ & & & & & & & \\  \hline

$\vdots$ & & & & $ \Rone $ & $\Rone$  & & & & & & & & & &  & & \\ \hline

$\vdots$ & & & & $ \Tone{2} $ & $\Rone$ & & & & & & & & & & & &  \\ \hline

$\underline{j}(\widehat{\beta}) + 4$  &   &   & $ \Rone $ & &  & & &&  & & &  & & & & & \\ \hline

$\underline{j}(\widehat{\beta}) + 2$  & $ \Rone $ &   & & & & & &   & & & & & & & & & \\ \hline

$\underline{j}(\widehat{\beta})$    & $ \Rone $ &   & & & & & & & & & & & & & & &  \\ \hline
\end{tblr}
\end{tiny}
    \caption{$H(\widehat{\beta})$, with $\beta = \Delta^p \sigma_1^{k_1}\sigma_2^2$, $p \geq 0$ even and $k_1 \geq 3$. The number of blue blocks is $\frac{p}{2}$.}
    \label{tab:Delta_p_C2}
\end{table}


\begin{table}[htb!]
    \centering
      \captionsetup{width=.95\linewidth}
    \begin{tiny}
    \begin{tblr}{colspec={|c||c|c|c|c|c|c|c|c|c|c|c|c|c|c|c|c|c|c|c|}, cell{13-17}{4-7} = {celeste}, cell{4-8}{13-16} = {celeste}, cell{10-14}{8-11}={celeste}, cell{7}{12}={celeste}}
\hline
\backslashbox{\!$j$\!}{\!$i$\!} & $0$ & $1$ & $2$ & $3$ & $4$ & $5$ & $6$ & $7$ & $8$ & $9$ & $\cdots$ & $\cdots$ & $\cdots$ & $\cdots$ & $\cdots$ & $\cdots$ & $\mathbf{i}(p)$ & $\cdots$  \\
\hline
\hline

$\vdots$ &   &   &   & & &  & & & & & & & & & & & & \SetCell[r=3]{c}{$Y_{\widehat{\beta}}$}       \\ \hline

$\vdots$ &   &   &   & &   & & & & & &  & &      & & & & $\Rone^2$  \\ \hline

$\vdots$ & & & & & & &  & & & & & & & &  $\Rone$ & & $\Tone{2}^2$  \\ \hline

$\mathbf{j}(p)$ & &  & & & & & & & & & & & & & $\Rone$ & $\Rone^2$ & & \\ \hline

$\vdots$  & &  & & & & & & & & & & & $\Rone$ & $\Rone$ & &  & & \\ \hline

$\vdots$ &   &   &  &  & & & & & & & \SetCell[r=5]{c}{$\udots$}  & & $\Tone{2}$ & $\Rone$ &        &  & &  \\ \hline

$\vdots$ & & & & &  &  & &   & & & & $\Rone$ \\  \hline

$\vdots$ & & & & &  &  & &   & & & \\  \hline

$\vdots$ & & & & &  &  & &   & & $\Rone$ & \\  \hline

$\vdots$ & & & & &  &  & &   & & $\Rone$ & \\  \hline

$\vdots$ &   &   &   & &   & &  & $\Rone$ & $\Rone$ & & & & & & & \\  \hline

$\vdots$ & & & & &  &  $\Rone$  & &  $\Tone{2}$ &  $\Rone$  & & & \\  \hline

$\vdots$ & & & & &  &  $\Rone$  & $\Rone$ &   & & & \\  \hline

$\vdots$ & & & & $ \Rone $ & $\Rone$  & & & & & &  &  & & \\ \hline

$\vdots$ & & & & $ \Tone{2} $ & $\Rone$ & & & & & & & & &  \\ \hline

$\underline{j}(\widehat{\beta}) + 4$  &   &   & $ \Rone $ & &  &   & & &  & & & & & \\ \hline

$\underline{j}(\widehat{\beta}) + 2$  & $ \Rone $ &   &   & &   & & & & & & & & & \\ \hline

$\underline{j}(\widehat{\beta})$    & $ \Rone $ &   &   & & & & & & & & & & &  \\ \hline
\end{tblr}
\end{tiny}
    \caption{$H(\widehat{\beta})$, with $\beta = \Delta^p \sigma_1^{k_1}\sigma_2^{k_2}$, $p \geq 0$ even and $k_1,k_2 \geq 3$. The number of blue blocks is $\frac{p}{2}$.}
    \label{tab:Delta_p_C3}
\end{table}

\begin{table}[htb!]
    \centering
    \begin{tiny}
    \captionsetup{width=.99\linewidth}
    \begin{tblr}{colspec={|c||c|c|c|c|c|c|c|c|c|c|c|c|c|c|c|c|c|c|c|}, cell{13-17}{4-7} = {celeste}, cell{4-8}{13-16} = {celeste}, cell{10-14}{8-11}={celeste}, cell{7}{12}={celeste}}
\hline
\backslashbox{\!$j$\!}{\!$i$\!} & $0$ & $1$ & $2$ & $3$ & $4$ & $5$ & $6$ & $7$ & $8$ & $9$ &  $\cdots$ & $\cdots$ & $\cdots$ & $\cdots$ & $\cdots$ & $\cdots$ & $\mathbf{i}(p)$ & $\cdots$  \\
\hline
\hline

$\vdots$ &   &   &  & & & & & & &  & & & & & & & & \SetCell[r=3]{c}{$Z_{\widehat{\beta}}$}       \\ \hline

$\vdots$ &   &   &   & &   & &  & & & & & & & & & & $\Rone$  \\ \hline

$\vdots$ & & & & & & &  & & & & & & & &  $\Rone$ &  & $\Tone{2}$  \\ \hline

$\mathbf{j}(p)$ & &  & & & & & & & & & & & & & $\Rone$ & $\Rone$ & & \\ \hline

$\vdots$  & &  & & & &  & & & & & & & $\Rone$ & $\Rone$ & &  & & \\ \hline

$\vdots$ &   &   &  &  & & & & & & &  \SetCell[r=5]{c}{$\udots$}  & & $\Tone{2}$ & $\Rone$ &        &  & &  \\ \hline

$\vdots$ &   &   &   & &   &  & & & & & & $\Rone$ & & & & \\  \hline

$\vdots$ & & & & &  &  & &   & & & \\  \hline

$\vdots$ & & & & &  &  & &   & &  $\Rone$ & & \\  \hline
$\vdots$ & & & & &  &  & &   &&  $\Rone$ & & \\  \hline
$\vdots$ & & & & &  &  & & $\Rone$   & $\Rone$ & & \\  \hline

$\vdots$ & & & & &  &  $\Rone$  & & $\Tone{2}$   & $\Rone$ & & \\  \hline

$\vdots$ & & & & &  &  $\Rone$  & $\Rone$ &   & & & \\  \hline

$\vdots$ & & & & $ \Rone $ & $\Rone$  & & & & & &  &  & & \\ \hline

$\vdots$ & & & & $ \Tone{2} $ & $\Rone$ & & & & & & & & &  \\ \hline

$\underline{j}(\widehat{\beta}) + 4$  &   &   & $ \Rone $ & &  &   & & &  & & & & & \\ \hline

$\underline{j}(\widehat{\beta}) + 2$  & $ \Rone $ &   &   & &   & & & & & & & & & \\ \hline

$\underline{j}(\widehat{\beta})$    & $ \Rone $ &   &   & & & & & & & & & & &  \\ \hline
\end{tblr}
\end{tiny}
    \caption{$H(\widehat{\beta})$, with $\beta = \Delta^{2 \lfloor \frac{p}{2} \rfloor}\gamma$ and $\gamma \neq \Delta$, where $\inf(\gamma) = 1$ or $\inf(\gamma)=0$ and $\operatorname{sl}(\gamma) \geq 4$. The number of blue blocks is $\lfloor \frac{p}{2} \rfloor $.}
    \label{tab:Delta_p_C4}
\end{table}

\clearpage

\bibliographystyle{amsplain}

\end{document}